\documentclass[english]{article}
\usepackage{lmodern}
\usepackage[T1]{fontenc}
\usepackage{mathtools}
\usepackage{geometry}
\usepackage[usenames,dvipsnames,svgnames,table]{xcolor}

\usepackage{babel}
\usepackage{refstyle}
\usepackage{float}

\usepackage{amsmath}
\usepackage{amsthm}
\usepackage{graphicx}
\usepackage[all]{xy}

\usepackage[unicode=true,pdfusetitle,
 bookmarks=true,bookmarksnumbered=false,bookmarksopen=false,
 breaklinks=false,pdfborder={0 0 0},pdfborderstyle={},backref=false,colorlinks=true]
 {hyperref}
\usepackage{amssymb}

\makeatletter

\AtBeginDocument{\providecommand\thmref[1]{\ref{thm:#1}}}
\AtBeginDocument{\providecommand\subsecref[1]{\ref{subsec:#1}}}
\AtBeginDocument{\providecommand\defref[1]{\ref{def:#1}}}
\AtBeginDocument{\providecommand\propref[1]{\ref{prop:#1}}}
\AtBeginDocument{\providecommand\lemref[1]{\ref{lem:#1}}}
\AtBeginDocument{\providecommand\corref[1]{\ref{cor:#1}}}
\AtBeginDocument{\providecommand\remref[1]{\ref{rem:#1}}}
\AtBeginDocument{\providecommand\exaref[1]{\ref{exa:#1}}}
\RS@ifundefined{subsecref}
  {\newref{subsec}{name = \RSsectxt}}
  {}
\RS@ifundefined{thmref}
  {\def\RSthmtxt{theorem~}\newref{thm}{name = \RSthmtxt}}
  {}
\RS@ifundefined{lemref}
  {\def\RSlemtxt{lemma~}\newref{lem}{name = \RSlemtxt}}
  {}

\theoremstyle{plain}
\newtheorem{thm}{\protect\theoremname}[subsection]
\theoremstyle{definition}
\newtheorem{defn}[thm]{\protect\definitionname}
\theoremstyle{definition}
\newtheorem{rem}[thm]{\protect\remarkname}
\theoremstyle{plain}

\theoremstyle{plain}
\newtheorem{conj}[thm]{Conjecture}
\newtheorem{lem}[thm]{\protect\lemmaname}
\theoremstyle{definition}
\newtheorem{notation}[thm]{\protect\notationname}
\theoremstyle{plain}
\newtheorem{prop}[thm]{\protect\propositionname}
\theoremstyle{plain}
\newtheorem{cor}[thm]{\protect\corollaryname}
\theoremstyle{definition}
\newtheorem{example}[thm]{\protect\examplename}

\newtheorem{theorem}{Theorem}

\@ifundefined{date}{}{\date{}}

\usepackage{dsfont}
\newref{prop}{name=Proposition~}
\newref{cor}{name=Corollary~}
\newref{rem}{name=Remark~}
\newref{def}{name=Definition~}
\newref{exa}{name=Example~}
\newref{thm}{name=Theorem~}
\newref{lem}{name=Lemma~}

\usepackage{hyperref}
\hypersetup{bookmarksnumbered,colorlinks=true}
\hypersetup{
     colorlinks   = true,
     linkcolor	  = DarkGreen,
     citecolor    = Purple
}

\makeatother

\providecommand{\corollaryname}{Corollary}
\providecommand{\definitionname}{Definition}
\providecommand{\examplename}{Example}
\providecommand{\lemmaname}{Lemma}
\providecommand{\notationname}{Notation}
\providecommand{\propositionname}{Proposition}
\providecommand{\questionname}{Question}
\providecommand{\remarkname}{Remark}
\providecommand{\theoremname}{Theorem}

\newcommand{\cocone}{\mathbin{\rotatebox[origin=c]{90}{$\triangle$}}}
\newcommand{\cone}{\mathbin{\rotatebox[origin=c]{-90}{$\triangle$}}}
\newcommand{\es}{\operatorname{{\small \normalfont{\text{\O}}}}}
\newcommand{\iso}{\xrightarrow{\,\smash{\raisebox{-0.5ex}{\ensuremath{\scriptstyle\sim}}}\,}}
\DeclareMathOperator{\supp}{supp}

\begin{document}
\global\long\def\into{\hookrightarrow}%
\global\long\def\onto{\twoheadrightarrow}%
\global\long\def\ss{\subseteq}%
\global\long\def\adj{\leftrightarrows}%
\global\long\def\oto#1{\xrightarrow{#1}}%
\global\long\def\nto{\rightarrowtail}%
\global\long\def\from{\leftarrow}%
\global\long\def\bb#1{\mathbb{#1}}%
\global\long\def\red#1{\textcolor{DarkBlue}{#1}}%
\global\long\def\white#1{\textcolor{white}{#1}}%

\global\long\def\st{\operatorname{st}}%
\global\long\def\pt{\operatorname{pt}}%
\global\long\def\nm{\operatorname{Nm}}%
\global\long\def\Id{\operatorname{Id}}%
\global\long\def\one{\mathds{1}}%
\global\long\def\bc{\operatorname{BC}}%
\global\long\def\can{\operatorname{can}}%
\global\long\def\tf{\operatorname{tf}}%
\global\long\def\tor{\operatorname{tor}}%
\global\long\def\colim{\operatorname*{\underrightarrow{\lim}}}%
\global\long\def\holim{\operatorname*{\underleftarrow{\lim}}}%
\global\long\def\map{\operatorname{Map}}%
\global\long\def\End{\operatorname{End}}%
\global\long\def\fun{\operatorname{Fun}}%
\global\long\def\hom{\operatorname{Hom}}%
\global\long\def\LMod{\operatorname{LMod}}%
\global\long\def\BMod{\operatorname{BMod}}%
\global\long\def\RMod{\operatorname{RMod}}%
\global\long\def\Mod{\operatorname{Mod}}%
\global\long\def\alg{\operatorname{Alg}}%

\global\long\def\calg{\operatorname{CAlg}}%
\global\long\def\cocalg{\operatorname{coCAlg}}%
\global\long\def\Mod{\operatorname{Mod}}%
\global\long\def\cat{\mathbf{Cat}}%
\global\long\def\Sp{\operatorname{Sp}}%
\global\long\def\comm{\operatorname{CommRing}}%
\global\long\def\frob{\operatorname{FrobRing}}%
\global\long\def\T{T\!}%
\global\long\def\K{K\!}%
\global\long\def\X{F\!}%

\title{Ambidexterity in Chromatic Homotopy Theory}
\author{Shachar Carmeli, Tomer M. Schlank, and Lior Yanovski}
\maketitle

\begin{abstract}
We extend the theory of ambidexterity developed
by M. J. Hopkins and J. Lurie and show that the $\infty$-categories
of $\T\left(n\right)$-local spectra are $\infty$-semiadditive for
all $n$, where $\T\left(n\right)$ is the telescope on a $v_{n}$-self
map of a type $n$ spectrum. This extends and provides a new
proof for the analogous result of Hopkins-Lurie on $\K\left(n\right)$-local
spectra. Moreover, we show that $\K\left(n\right)$-local and $\T\left(n\right)$-local spectra are respectively, the minimal and maximal $1$-semiadditive localizations of spectra with respect to a homotopy ring, and that all such localizations are in fact $\infty$-semiadditive.  
As a consequence, we deduce that several different notions of ``bounded chromatic height'' for homotopy rings are equivalent, and in particular, that $\T\left(n\right)$-homology of $\pi$-finite spaces depends only on the $n$-th Postnikov truncation.
A key ingredient in the proof of the main result is a construction of a certain power operation for commutative ring objects in stable $1$-semiadditive $\infty$-categories.
This is closely related to some known constructions for Morava $E$-theory and is of independent interest. 
Using this power operation we also give a new proof, and a generalization, of a nilpotence conjecture of J.P. May, which was proved by A. Mathew, N. Naumann, and J. Noel.
\end{abstract}

\begin{figure}[H]
\centering{}\includegraphics[scale=0.475]{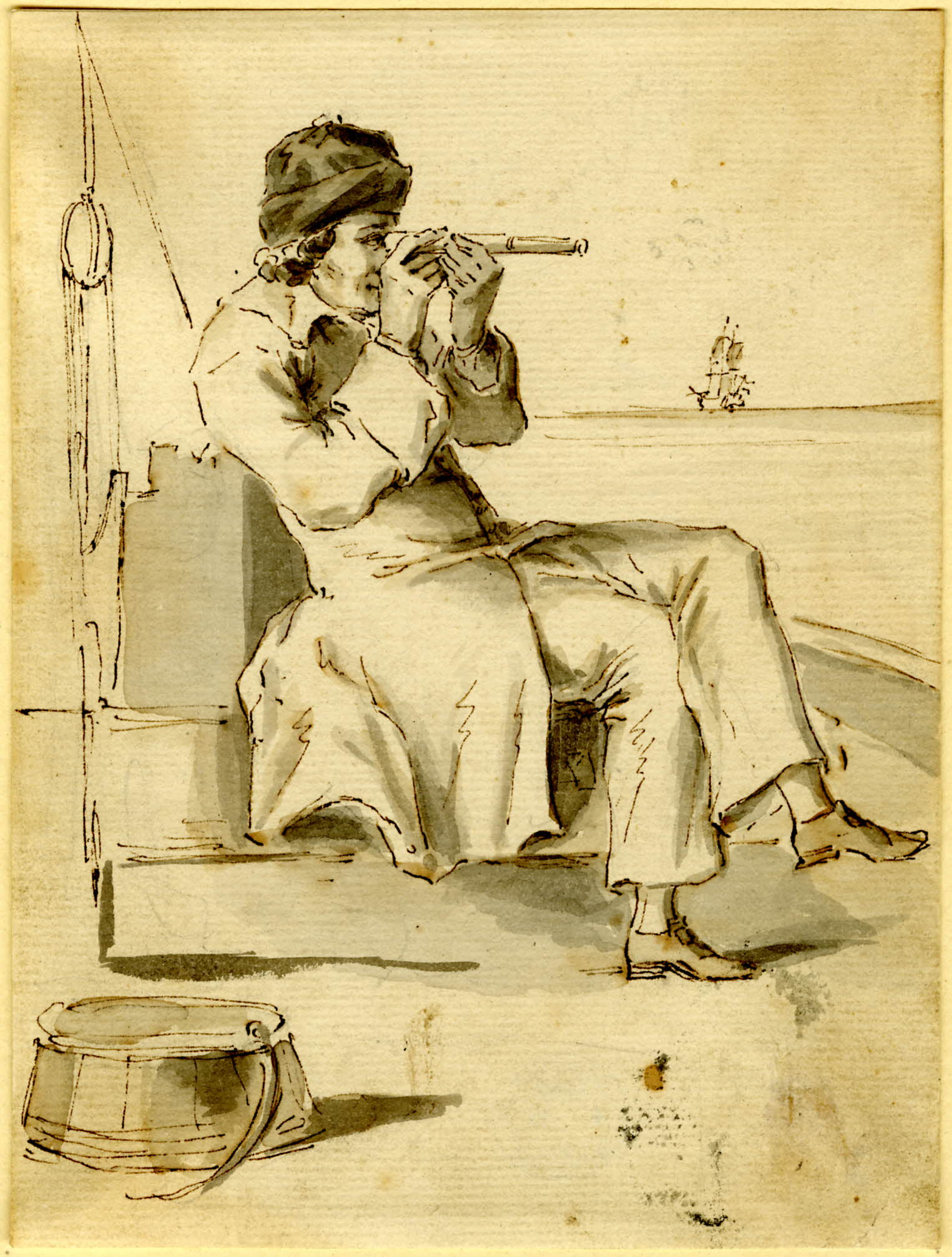}\caption{A Seaman, holding a telescope with two hands, Louis Peter Boitard
{[}{\copyright}Trustees of The British Museum{]}}
\end{figure}

\tableofcontents{}

\section{Introduction}

\subsection{Main Results}
Given an abelian group $A$ with an action of a finite group $G$, summation along the orbit provides a natural map $\nm_G\colon A_G \to A^G$ from the co-invariants to the invariants.
In general, this map may have both non-trivial kernel and cokernel. 
However, when $A$ is a rational vector space, $\nm_G$ is always an isomorphism.
Similarly, given a spectrum $X$  with an action of $G$, the
spectra of homotopy orbits $X_{hG}$ and homotopy fixed points $X^{hG}$
are also related by a canonical norm map $\nm_G\colon X_{hG}\to X^{hG}$.
As before, this map is usually far from being an equivalence. 
However, there are certain homology theories, such that when working locally with respect to them, the analogous norm map is always a local equivalence.
For a spectrum $E$, let us denote by $\Sp_{E}$ the $\infty$-category
of $E$-local spectra, and for $X\in\Sp_{E}$ with a $G$-action, we denote
by $X_{hG}$ and $X^{hG}$ the homotopy orbits and homotopy fixed
points respectively, in the $\infty$-category $\Sp_{E}$. 
\begin{thm}
[Greenlees-Hovey-Sadofsky, \cite{HState,GState}]\label{thm:Hovey_Sadofsky_Greenlees}Let
$\K\left(n\right)$ be Morava $K$-theory of height $n$. For every
$X\in\Sp_{K\left(n\right)}$ with an action of a finite group $G$,
the canonical norm map
\[
\nm_G\colon X_{hG}\iso X^{hG}\quad\in\quad\Sp_{K\left(n\right)}
\]
is an equivalence.
\end{thm}

Since $\K\left(0\right)=H\mathbb{Q}$, the case $n=0$ follows easily from the invertibility of $\nm_G$ on rational representations of $G$.
However, for $n>0$ this is a remarkable fact, showcasing the intermediary behavior of $\K\left(n\right)$-local homotopy theory, interpolating between zero and positive characteristic.   

Considering the classifying space $BG$ as
an $\infty$-groupoid, the data of an $E$-local spectrum with an action of $G$ is equivalent to a functor $F\colon BG\to\Sp_{E}$. In these terms, the homotopy orbits and homotopy fixed points of the action are then just the colimit and limit of $F$ respectively (again, in $\Sp_{E}$).
In \cite{HopkinsLurie}, Hopkins and Lurie extended \thmref{Hovey_Sadofsky_Greenlees} to more general limits and colimits.  

\begin{defn}
\label{def:m_Finite}Given $m\ge-2$, a space $A$ is called \emph{$m$-finite}
if it is $m$-truncated, has finitely many connected components and
all of its homotopy groups are finite. It is called \emph{$\pi$-finite}
if it is $m$-finite for some $m$.
\end{defn}

\begin{thm}
[Hopkins-Lurie, \cite{HopkinsLurie}]\label{thm:Hopkins_Lurie}Let
$A$ be a $\pi$-finite space. For every $F\colon A\to\Sp_{K\left(n\right)}$,
there is a canonical (and natural) equivalence 
\[
\nm_{A}\colon\colim_A F\iso\holim_A F\quad\in\quad\Sp_{K\left(n\right)}.
\]

\end{thm}
The special case, where $A=BG$ for a finite group $G$, recovers \thmref{Hovey_Sadofsky_Greenlees}.  

The canonical norms of \thmref{Hopkins_Lurie} (and \thmref{Hovey_Sadofsky_Greenlees})
can be set in the broader context of higher semiadditivity, developed in \cite{HopkinsLurie}.
Let $\mathcal{C}$ be an $\infty$-category that admits all (co)limits
indexed by $\pi$-finite spaces. For every $\pi$-finite space $A$,
we have two functors
\[
\colim\limits _{A}\,,\ \holim\limits _{A}\colon\fun\left(A,\mathcal{C}\right)\to\mathcal{C}.
\]
In \cite{HopkinsLurie}, the authors set up a general process that
attempts to construct canonical natural transformations 
\[
\nm_{A}\colon\colim\limits _{A}\to\holim\limits _{A}
\]
for all $m$-finite spaces $A$, by induction on $m$. The $m$-th
step of this process requires that all canonical norm maps for $\left(m-1\right)$-finite
spaces, that were constructed in the previous step, are \emph{isomorphisms}.
The property of an $\infty$-category $\mathcal{C}$, that these canonical
norm maps can be constructed and are isomorphisms for all $m$-finite
spaces, is called \emph{$m$-semiadditivity} (see \subsecref{Higher_Semiadditivity}).
We can thus restate \thmref{Hovey_Sadofsky_Greenlees} 
as saying that the $\infty$-category $\Sp_{K\left(n\right)}$ is $1$-semiadditive, and \thmref{Hopkins_Lurie} as saying that it
is $\infty$-semiadditive (i.e. $m$-semiadditive for all $m$).

Kuhn extended \thmref{Hovey_Sadofsky_Greenlees} in a different direction, by replacing $\K\left(n\right)$-localization with the closely related telescopic localization. Namely, let $\T\left(n\right)$ be a telescope on a $v_{n}$-self map of some type $n$ finite spectrum.
\begin{thm}
[Kuhn, \cite{Kuhn}]\label{thm:Kuhn}  
The $\infty$-category $\Sp_{\T\left(n\right)}$ is $1$-semiadditive.
\end{thm}
In view of \thmref{Kuhn} and \thmref{Hopkins_Lurie}, M. Hopkins asked whether the $\infty$-category $\Sp_{\T\left(n\right)}$ is $\infty$-semiadditive as well. 
Our first result is an affirmative answer to this question. 
\begin{theorem}[\ref{thm:Tn_Semiaddi}]
	\label{thm:telescopic_infty_semiadditive_intro}
\emph{
	The $\infty$-category $\Sp_{\T\left(n\right)}$ is $\infty$-semiadditive.
}
\end{theorem}

Our proof of \thmref{telescopic_infty_semiadditive_intro} uses the general framework of higher semiadditivity developed by Hopkins and Lurie in \cite{HopkinsLurie}, but is quite different than their proof of \thmref{Hopkins_Lurie} (see \subsecref{Outline_Proof} for an outline). Since the latter is implied by the former, our argument provides an alternative proof for \thmref{Hopkins_Lurie} as well.
 
Our next result concerns the classification of $1$-semiadditive localizations of $p$-local spectra with respect to homotopy rings\footnote{In fact, all the results apply more generally to \emph{weak rings}. That is, spectra equipped with a multiplication map and a one-sided unit, and no associativity conditions (see \defref{Weak_Ring}).}. We show that the $\infty$-categories $\Sp_{\K\left(n\right)}$ and $\Sp_{\T\left(n\right)}$ are precisely the minimal and maximal examples of such localizations.

\begin{theorem}[\ref{thm:Monochrom}]
	\label{thm:Tn_Universal}
\emph{
Let $R$ be a non-zero $p$-local homotopy ring spectrum. The $\infty$-category $\Sp_{R}$ is $1$-semiadditive if and only if there exists a (necessarily unique) integer $n\ge 0$, such that
\[
\Sp_{\K\left(n\right)}\ss\Sp_{R}\ss\Sp_{\T\left(n\right)}.
\]
}
\end{theorem}  

Equivalently, using the Nilpotence Theorem, $\Sp_R$ is $1$-semiadditive if and only if there is exactly one integer $n\ge0$ for which $R\otimes \K\left(n\right) \neq 0$, and $R \otimes H\mathbb{F}_p =0$. Namely, if $R$ is supported at a unique (finite) chromatic height\footnote{A detailed argument for this equivalence is given in the proof of \thmref{Are_Equivalent_Intro} (\ref{thm:Monochrom}).}. 

Combining \thmref{telescopic_infty_semiadditive_intro} with \thmref{Tn_Universal}, and using the arithmetic square, we show that for localizations of $\Sp$ with respect to homotopy rings, the entire hierarchy of higher semiadditivity collapses.

\begin{theorem}[\ref{cor:Semiadd_Collapse}]
	\label{thm:Main_Theorem} 
\emph{
	Let $R\in \Sp$ be a homotopy ring spectrum. The $\infty$-category $\Sp_R$ is $1$-semiadditive if and only if it is $\infty$-semiadditive. 
}
\end{theorem} 

This leads us to formulate the following general conjecture:
\begin{conj}
Every presentable, stable, and $1$-semiadditive $\infty$-category is $\infty$-semiadditive.
\end{conj}

Another remarkable property of the localizations $\Sp_{\K(n)}$ and $\Sp_{\T(n)}$, is the existence of the so-called \emph{Bousfield-Kuhn functor}, i.e. a retract of $\Omega^{\infty}{\colon}\Sp_{R}\to\mathcal{S_{*}}$.
This phenomenon turns out to be also strongly connected to higher semiadditivity. In \cite{ClausenAkhil}, the authors gave a new (and short) proof of \thmref{Kuhn}, by showing that every localization of  $\Sp$, that admits a Bousfield-Kuhn functor, is $1$-semiadditive. Combined with all of the above, the situation can be pleasantly summarized as follows:

\begin{theorem}[\ref{thm:Monochrom}]
\emph{
	\label{thm:Are_Equivalent_Intro}
 Let $R$ be a non-zero $p$-local homotopy ring spectrum.
		The following are equivalent:
		\begin{enumerate}
			\item There is exactly one integer $n\ge0$ for which \(R\otimes \K\left(n\right) \ne 0,\) and $R\otimes H\mathbb{F}_p=0$.
			\item There exists a (necessarily unique) integer $n\ge0$, such that 
		\(\Sp_{\K\left(n\right)}\ss\Sp_{R}\ss\Sp_{\T\left(n\right)}.\)
			\item Either $\Sp_{R}=\Sp_{H\bb Q}$, or $\Omega^{\infty}{\colon}\Sp_{R}\to\mathcal{S_{*}}$
			admits a retract.
			\item $\Sp_{R}$ is $1$-semiadditive.
			\item $\Sp_{R}$ is $\infty$-semiadditive.
		\end{enumerate}
}
\end{theorem}   

It seems appropriate at this point to say a few words about the results summarized in \thmref{Are_Equivalent_Intro}, in light of the (still open) Telescope Conjecture, which asserts that $\Sp_{\K\left(n\right)} = \Sp_{\T\left(n\right)}$ (see \cite{ravconj}). 
If true, the property of higher semiadditivity characterizes completely the $\K\left(n\right)$-local $\infty$-categories among localizations of $\Sp$ with respect to homotopy rings (as does the existence of the Bousfield-Kuhn functor). 
If false, the property of higher semiadditivity fails to detect the difference, but on the upside, we are provided with more examples of $\infty$-semiadditive $\infty$-categories. 
At any rate, our results corroborate, the by now well-established fact, that the Telescope Conjecture is rather subtle.

The $1$-semiadditivity of $\Sp_{T(n)}$ and $\Sp_{K(n)}$, has found many applications in chromatic homotopy theory. For example, it was used to analyze the Balmer spectrum in an equivariant setting \cite{barthel2017balmer}. 
It was also recently used in \cite{heuts2018lie} to generalize Quillen's rational homotopy theory to higher chromatic heights. 
In an upcoming work we shall use \thmref{telescopic_infty_semiadditive_intro}, i.e. the \emph{$\infty$-semiadditivity} of $\Sp_{T(n)}$, to lift the maximal abelain Galois extension of $\Sp_{K(n)}$ to $\Sp_{T(n)}$ and draw consequences for the Picard group of $\Sp_{T(n)}$.

We shall now describe an application of $\infty$-semiadditivity of $\Sp_{T(n)}$ to a matter of chromatic homotopy theory, that does not mention higher semiadditivity explicitly.

\begin{theorem}[\ref{thm:Height_Below_n}]
	\label{thm:Bounded_Height}
\emph{
	Let $R$ be a $p$-local homotopy ring spectrum and let $d\ge0$. The following are equivalent\footnote{The equivalence of (1) and (2) is well known. The new content is that they are both equivalent to (3).}:
	\begin{enumerate}
		\item $R\otimes \K\left(m\right)=0$ for all $m> d$.
		\item $R\otimes \X\left(d+1\right)=0$ for a finite spectrum $\X\left(d+1\right)$ of type $d+1$. 
		\item $R\otimes \Sigma^{\infty}A=0$ for every $d$-connected $\pi$-finite space $A$.
	\end{enumerate}
}
\end{theorem}

Namely, we obtain an equivalence of three different notions of ``height $\le d$'' for a homotopy ring: (1) the ``algebraic'' one using Morava $K$-theories, (2) the ``geometric'' one using finite complexes, and (3) the ``categorical'' one using $\pi$-finite spaces. 
The categorical height of a spectrum (i.e. the minimal $d$ for which condition (3) holds) was considered, using different terminology, by Bousfield in \cite{Bousfield82}. 
The most prominent example of such $R$ is $\K\left(n\right)$, which by \cite{RavenelWilson}, has categorical height $n$. Bousfield's work also implies that for all $n\ge0$, the spectrum $\T\left(n\right)$ has \emph{some} finite categorical height, but determining its \emph{precise} value has been an open question\footnote{By comparing with $K(n)$, it is clearly at least $n$, but not much has been known about it beyond that.}.
This can be now settled using \thmref{Bounded_Height}; as the algebraic and geometric heights of $\T\left(n\right)$ are known to be equal to $n$, so must the categorical height.

The proof of the above results relies on establishing certain consequences of $1$-semiadditivity, especially in the context of \emph{stable} $\infty$-categories.
The main one, which is central to the proof of
\thmref{telescopic_infty_semiadditive_intro}, but is also of independent interest, is the existence of certain canonical ``power operations''. 
\begin{theorem}[\ref{thm:Delta_Semi_Add}, \ref{thm:Frob_Lift}]
\emph{
Let $E\in\Sp$, such that $\Sp_{E}$ is $1$-semiadditive (e.g. $E=\T\left(n\right)$)
and let $X$ be an $\bb E_{\infty}$-algebra in $\Sp_{E}$. The commutative
ring $R=\pi_{0}X$ admits a canonical additive $p$-derivation $\delta\colon R\to R$
(see \defref{Delta}). In particular, the operation
\[
\psi\left(x\right)=x^{p}+p\delta\left(x\right)
\]
is a linear map, which is a canonical lift of the Frobenius endomorphism
modulo $p$. The operation $\delta$ (and hence $\psi$) is functorial
with respect to maps of $\bb E_{\infty}$-algebras.
}
\end{theorem}

For $K(1)$-local $\bb{E}_{\infty}$-rings, Hopkins has constructed in \cite{HopkinsK1} similar looking power operations denoted $(\psi,\theta)$. Generalizations of these operations to higher heights were studied by different authors including \cite{StricklandSym} and \cite{Rezkpower}. In particular, a canonical lift of Frobenius was constructed in \cite{StapletonLift} for the Morava $E$-theory cohomology ring of a space. However, even for $K(1)$-local rings, our power operation $\delta$ turns out to be \emph{different} from the operation $\theta$ constructed by Hopkins. We differ the detailed study of the wealth of power operations on $\bb{E}_{\infty}$-algebras in $1$-semiadditive stable symmetric monoidal $\infty$-categories to a future work. 

Employing this power operation we obtain a general criterion for detecting nilpotence in the homotopy groups of an $\mathbb{E}_{\infty}$-ring, which is inspired by (and generalizes) a conjecture of J. P. May, that was proved in \cite{MathewMay}.

\begin{theorem}[\ref{thm:May}]
\emph{
\label{thm:Main_Sofic}Let $E$ be a homotopy commutative ring spectrum,
such that $\Sp_{E}$ is $1$-semiadditive and let $R$ be an $\bb E_{\infty}$-ring
spectrum\footnote{In fact, it suffices that $R$ is an $H_\infty$-ring spectrum.}. For every $x\in\pi_{*}R$, if the image of $x$ in $\pi_{*}\left(H\bb Q\otimes R\right)$
is nilpotent, then the image of $x$ in $\pi_{*}\left(E\otimes R\right)$
is nilpotent (i.e. the single homology theory $H\bb Q$ detects nilpotence
in all $1$-semiadditive multiplicative homology theories).
}
\end{theorem}

In fact, we prove \thmref{Main_Sofic} for a wider class of spectra
$E$, which we call \emph{sofic} (see \defref{sofic}). These include all spectra whose Bousfield class is contained in a sum of spectra $E$ for which $\Sp_{E}$ is $1$-semiadditive. 
In particular, it can be applied to the Morava $E$-theories of any height and the finite localizations of the sphere $L_{n}^{f}\bb S$. 

\subsection{Background on Higher Semiadditivity\label{subsec:Higher_Semiadditivity}}

We shall now give an informal introduction to higher semiadditivity. The goal is to motivate both the concept of higher semiadditivity introduced in \cite[section 4]{HopkinsLurie} and the more general perspective on it, that we develop in this paper, using abstract norms and integration.

\subsubsection{From Norms to Integration}
Since the construction of the canonical norm maps is inductive, it
will be helpful to begin with describing some \emph{consequences}
of having invertible norm maps. This will also clarify their relation
to the classical notion of semiadditivity. For an ordinary category
$\mathcal{C}$, semiadditivity is a property, whose main feature is
the ability to sum a finite family of morphisms between two objects.
Similarly, for an $\infty$-category $\mathcal{C}$, being $m$-semiadditive
is a property, whose main feature is the ability to sum an \emph{$m$-finite}
family of morphisms between two objects. Namely, given an $m$-finite
space $A$ and a map 
\[
\varphi\colon A\to\map_{\mathcal{C}}\left(X,Y\right),
\]
we define a map 
\[
\int\limits _{A}\varphi\colon X\to Y,
\]
which we should think of as the sum (or integral) of $\varphi$ over
$A$, as the composition
\[
X\oto{\,\Delta\,}\holim\limits _{A}X\oto{\holim\varphi}\holim\limits _{A}Y\oto{\nm_{A}^{-1}}\colim\limits _{A}Y\oto{\,\nabla\,}Y.
\]

Note, that for an ordinary semiadditive category, summation over a
finite set $A$ is indeed obtained in this way using the \emph{canonical}
isomorphism
\[
\nm_{A}\colon\coprod_{A}X\iso\prod_{A}X.
\]

As a special case, for every object $X\in\mathcal{C}$, integrating
the constant $A$-family on $\Id_{X}$, produces an endomorphism $|A|\in\map_{\mathcal{C}}\left(X,X\right)$.
This generalizes the ``multiplication by $k$'' endomorphism of $X$ for an integer $k$ and should be thought of as multiplication by the ``cardinality of $A$''. 

\subsubsection{From Integration to Norms}

We now turn things around and \emph{construct} norm maps for $m$-finite
spaces by integrating some $\left(m-1\right)$-finite families of
maps. In general, given any space $A$ and a diagram $F\colon A\to\mathcal{C}$,
to specify a morphism 
\[
\nm_{A}\colon\colim\limits _{A}F\to\holim\limits _{A}F,
\]
roughly amounts to specifying a compatible collection of morphisms
\[
a,b\in A\colon\quad\nm_{A}^{a,b}\colon F\left(a\right)\to F\left(b\right).
\]
Fixing $a,b\in A$ and denoting by $A_{a,b}$ the space of paths from
$a$ to $b$, the diagram $F$ itself provides a \emph{family} of
candidates for $\nm_{A}^{a,b}$:
\[
F_{a,b}\colon A_{a,b}\to\map\left(F\left(a\right),F\left(b\right)\right).
\]
There is a priori no obvious (compatible) way to choose one of them,
but assuming we are able to integrate maps over the spaces $A_{a,b}$,
we can just ``sum them all'' 
\[
\nm_{A}^{a,b}=\int\limits _{A_{a,b}}F_{a,b}.
\]
This construction is somewhat easier to grasp when $F$ is constant
on an object $X$. In this special case, a morphism 
\[
\colim\limits _{A}X\to\holim\limits _{A}X,
\]
is the same as a map of spaces 
\[
A\times A\to\map_{\mathcal{C}}\left(X,X\right).
\]
That is, an ``$A\times A$ matrix'' of endomorphism of $X$, where
the $\left(a,b\right)\in A\times A$ entry corresponds to $\nm_{A}^{a,b}$.
The construction sketched above specializes to give $\nm_{A}^{a,b}=|A_{a,b}|$.
The construction of the norm in the general case can be thought of
as a ``twisted'' version of the one for the constant diagram.

\subsubsection{The Inductive Process}

To tie things up, we observe that if $A$ is $m$-finite, then the
path spaces $A_{a,b}$ are $\left(m-1\right)$-finite. Thus, assuming
inductively that we have \emph{invertible} canonical norm maps $\nm_{A}$
for all $\left(m-1\right)$-finite spaces $A$, we obtain a canonical
way to integrate $\left(m-1\right)$-finite families of morphisms.
As explained above, this allows us to define norm maps for all $m$-finite
spaces. It is now a \emph{property} that all those new norm maps are
isomorphisms, which in turn induces an operation of integration over
$m$-finite spaces and so on. We spell out the situation for small
values of $m$.
\begin{itemize}
\item [($-2$)]We define \emph{every} $\infty$-category to be $\left(-2\right)$-semiadditive.
Indeed, if $A$ is $\left(-2\right)$-finite, then $A\simeq\pt$ and
the canonical norm map $\nm_{\pt}$ is the identity natural transformation
of the identity functor. In particular, we get a canonical way to
sum a one point family of maps, which is just
taking the value at the point itself. 
\item [($-1$)]The only non-contractible $\left(-1\right)$-finite space
is $A=\es$. The associated norm map is the unique map 
\[
\nm_{\es}\colon0_{\mathcal{C}}\to1_{\mathcal{C}}
\]
from the initial object to the terminal object of $\mathcal{C}$.
Requiring this map to be an isomorphism is to require the existence
of a zero object. Thus, $\mathcal{C}$ is $\left(-1\right)$-semiadditive
if and only if it is \emph{pointed}. This in turn allows us to integrate
an empty family of morphisms. Namely, given $X,Y\in\mathcal{C}$,
we get a canonical zero map given by the composition 
\[
X\to1_{\mathcal{C}}\iso0_{\mathcal{C}}\to Y.
\]
\item [($0$)]A $0$-finite space is one that is equivalent to a finite
set $A$. Given a collection of objects $\left\{ X_{a}\right\} _{a\in A}$
in a \emph{pointed} $\infty$-category $\mathcal{C}$, we get a canonical
map 
\[
\nm_{A}\colon\coprod_{a\in A}X_{a}\to\prod_{a\in A}X_{a}.
\]
This map is given by the ``identity matrix'' (this uses the zero
maps, which in turn use the inverse of $\nm_{\es}$). Requiring these
maps to be isomorphisms is precisely the usual property of being \emph{semiadditive},
which allows one to sum a finite family of morphisms. 
\item [($1$)]A connected $1$-finite space is of the form $A=BG$ for a
finite group $G$. A diagram $F\colon BG\to\mathcal{C}$ is equivalent
to an object $X\in\mathcal{C}$ equipped with an action of $G$. When
$\mathcal{C}$ is \emph{semiadditive}, one can construct the canonical
norm map
\[
\nm_{BG}\colon X_{hG}\to X^{hG}
\]
and it can be identified with the classical norm of $G$. If $\mathcal{C}$
is stable, then $\nm_{BG}$ is an isomorphism if and only if its cofiber,
the Tate construction $X^{tG}$, vanishes. It is in this form that
\thmref{Hovey_Sadofsky_Greenlees} and \thmref{Kuhn} were originally
stated and proved.
\end{itemize}

\subsubsection{Relative and Axiomatic Integration}

Just like with ordinary semiadditivity, integration of $m$-finite
families of maps satisfies various compatibilities. These generalize
associativity, changing summation order, distributivity with respect to composition, etc. To conveniently manage those compatibility relations
it is useful to extend the integral operation to the relative case.
Given a map of $m$-finite spaces $q\colon A\to B$, the pullback
along $q$ functor
\[
q^{*}\colon\fun\left(B,\mathcal{C}\right)\to\fun\left(A,\mathcal{C}\right),
\]
admits a left and right adjoint, which we denote by $q_{!}$ and $q_{*}$
respectively. If $\mathcal{C}$ is $\left(m-1\right)$-semiadditive,
one can construct a canonical norm map $\nm_{q}\colon q_{!}\to q_{*}$
and it is an isomorphism when $\mathcal{C}$ is $m$-semiadditive.
Similarly to the absolute case, given objects $X,Y\in\fun\left(B,\mathcal{C}\right)$,
one can use the inverse of $\nm_{q}$ to define ``integration along
the fibers of $q$'',
\[
\int\limits _{q}\colon\map_{\fun\left(A,\mathcal{C}\right)}\left(q^{*}X,q^{*}Y\right)\to\map_{\fun\left(B,\mathcal{C}\right)}\left(X,Y\right).
\]

The approach we take in this paper is to further generalize the situation and to put it in an axiomatic framework. We define a \emph{normed
functor} 
\[
	q\colon\mathcal{D}\nto\mathcal{C},
\] 
to be a functor
\(
q^{*}\colon\mathcal{C}\to\mathcal{D},
\)
that admits a left adjoint $q_{!}$, a right adjoint $q_{*}$, and
is equipped with a natural transformation $\nm_{q}\colon q_{!}\to q_{*}$.
If this natural transformation is an isomorphism, we can use the same
formulas as above to define an abstract integration operation 
\[
\int\limits _{q}\colon\map_{\mathcal{D}}\left(q^{*}X,q^{*}Y\right)\to\map_{\mathcal{C}}\left(X,Y\right)
\]
for all $X,Y\in\mathcal{C}$. We proceed to develop a general calculus
of normed functors and integration, which can then be applied to the
context of higher semiadditivity. One advantage of this axiomatic
approach, is that it separates the formal aspects of this ``calculus''
from the rather involved inductive construction of the canonical norm
maps. Another advantage is that it unifies many seemingly different
phenomena as special cases of several general formal statements. This
renders the development of the theory more economic and streamlined. Finally, we
believe that this axiomatic framework might be of use elsewhere.

\subsection{Outline of the Proof} 
	\label{subsec:Outline_Proof}
The core result of this paper is the $\infty$-semiadditivity of $\Sp_{\T\left(n\right)}$. For the convenience of the reader, we shall now sketch the proof.
The argument is inductive on the level of semiadditivity $m$. 
The basis of the induction is $m=1$, which is given by \thmref{Kuhn}. 
Assume that $\Sp_{T\left(n\right)}$ is $m$-semiadditive. 
In order to show that $\Sp_{T\left(n\right)}$ is $\left(m+1\right)$-semiadditive, we need to prove that for every $\left(m+1\right)$-finite space $B$, the natural transformation $\nm_{B}\colon\colim_{B}\to\holim_{B}$ is an isomorphism. 
We proceed by a sequence of reductions. 
First, since $\Sp_{T\left(n\right)}$ is stable and $p$-local, by \cite[Proposition 4.4.16]{HopkinsLurie}, it suffices to show that
\begin{quotation}
(1) The norm map $\nm_{B}$ is an isomorphism for the single space
$B=B^{m+1}C_{p}$. 
\end{quotation}
Now, consider a fiber sequence of spaces
\[
\left(*\right)\quad A\to E\to B,
\]
where $A$ and $E$ are $m$-finite, and $B$ is connected and $\left(m+1\right)$-finite.
We prove that if the natural transformation $|A|$ is \emph{invertible}
(we call such $A$ \emph{amenable}), then $\nm_{B}$ is an isomorphism
(\propref{Amenable_Space}). In fact, it suffices to show that the
component of $|A|$ at the monoidal unit $\bb S_{T\left(n\right)}$
is invertible (\lemref{Box_Unit}). By abuse of notation, we denote
this component also by $|A|$.

In order to apply the above to $B=B^{m+1}C_{p}$, we introduce the
following class of ``candidates'' for $A$. We call a space $A$,
\emph{$m$-good} if it is connected, $m$-finite with $\pi_{m}A\neq0$,
and all homotopy groups of $A$ are $p$-groups. Since such $A$ is
in particular nilpotent, one can always fit it in a fiber sequence
$\left(*\right)$ with $B=B^{m+1}C_{p}$. Thus, we are reduced to
showing that
\begin{quotation}
(2) There exists an $m$-good space $A$, such that $|A|\in\pi_{0}\bb S_{T\left(n\right)}$
is invertible.
\end{quotation}
To detect invertibility in the ring $\pi_{0}\bb S_{T\left(n\right)}$,
we transport the problem into a better understood setting. Let $E_{n}$
be the Morava $E$-theory $\bb E_{\infty}$-ring spectrum of height
$n$, and let $\widehat{\Mod}_{E_n}$ be the
$\infty$-category of $\K\left(n\right)$-local $E_{n}$-modules.
The functor 
\[
E_n\widehat{\otimes}\left(-\right)\colon\Sp_{T\left(n\right)}\to\widehat{\Mod}_{E_n}
\]
is symmetric monoidal, and hence induces a map of commutative rings
\[
f\colon\pi_{0}\bb S_{T\left(n\right)}\to\pi_{0}E_{n}=\bb Z_{p}[[u_{1},\dots,u_{n-1}]].
\]
Using the Nilpotence Theorem and standard techniques of chromatic
homotopy theory, we show that an element of $\pi_{0}\bb S_{T\left(n\right)}$
is invertible, if and only if its image under $f$ is invertible (\corref{Nil_Conservativity_Telescopic}).
Moreover, the functor $E_{n}\widehat{\otimes}\left(-\right)$ is colimit
preserving. Thus, by general arguments of higher semiadditivity we can deduce that $\widehat{\Mod}_{E_n}$
is also $m$-semiadditive (\corref{Semi_Add_Mode}(2)), and moreover,
$f\left(|A|\right)$ coincides with the element $|A|$
of $\pi_{0}E_{n}$ (\corref{Integral_Functor}). Thus, we can replace
$\Sp_{T\left(n\right)}$ with the more approachable $\infty$-category
$\widehat{\Mod}_{E_n}.$ Namely,
it suffices to show
\begin{quotation}
(3) There exists an $m$-good space $A$, such that $|A|\in\pi_{0}E_{n}$
is \emph{invertible}.
\end{quotation}
By \cite[Lemma 1.33]{BobkovaG}, the image of $f$ is contained in the constants $\bb Z_{p}$. Hence, $|A|$ is invertible, if and only
if its $p$-adic valuation is zero. On $\bb Z_{p},$ we have the Fermat
quotient operation
\[
\tilde{\delta}\left(x\right)=\frac{x-x^{p}}{p},
\]
with the salient property of reducing the $p$-adic valuation of non-invertible
non-zero elements. The heart of the proof comprises of realizing the
algebraic operation $\tilde{\delta}$ in a way that acts on the elements
$|A|$ in an understood way. It is for this step that it
is crucial that our induction base is $m=1$. Namely, for
a presentable, \emph{$1$-semiadditive}, stable, $p$-local, symmetric
monoidal $\infty$-category $(\mathcal{C},\otimes, \one_{\mathcal{C}})$, we construct a ``power
operation'' (\defref{Delta} and \thmref{Delta_Semi_Add})
\[
\delta\colon\pi_{0}\left(\one_{\mathcal{C}}\right)\to\pi_{0}\left(\one_{\mathcal{C}}\right),
\]
that shares many of the formal properties of $\tilde{\delta}$.
In particular, specializing to the case $\mathcal{C}=\widehat{\Mod}_{E_n}$,
the operation $\delta$ coincides with $\tilde{\delta}$ on $\bb Z_{p}\ss\pi_{0}E_{n}$.
Moreover, for an $m$-good $A$, we have 
\[
\delta\left(|A|\right)=|A'|-|A''|,
\]
where $A'$ and $A''$ are also $m$-good (combine \defref{Delta_Semi_Add}
and \thmref{Alpha_Box}). It follows that if $|A|$ is
non-zero (and not already invertible), then at least one of $|A'|$
and $|A''|$ has \emph{lower} $p$-adic valuation than
$|A|$. The prototypical $m$-good space is the Eilenberg-MacLane
space $B^{m}C_{p}$. Hence, it suffices to show that
\begin{quotation}
(4) The element $|B^{m}C_{p}|\in\pi_{0}E_{n}$ is \emph{non-zero}.
\end{quotation}
To get a grip on the elements $|A|$, we reformulate them
in terms of the symmetric monoidal dimension (which does not refer
at all to higher semiadditivity). Let us denote by $A\otimes E_{n}$,
the colimit of the constant $A$-shaped diagram on $E_{n}$ in $\widehat{\Mod}_{E_n}$.
We show that $A\otimes E_{n}$ is a dualizable object\footnote{We show that this follows from higher semiadditivity, but it can be also deduced directly from the finite dimensionality of $K(n)_*(A)$ (\cite{RavenelWilson}). See \cite{hoveystrickland} and \cite{RognesStablyDualizable}.}, and that (\corref{Dim_Sym})
\[
\dim\left(A\otimes E_{n}\right)=|A^{S^{1}}|\quad\in\pi_{0}E_{n}.
\]
Since 
\[
|\left(B^{m}C_{p}\right)^{S^{1}}| = |B^{m}C_{p}\times B^{m-1}C_{p}| 
=  |B^{m}C_{p}||B^{m-1}C_{p}|,
\]
it suffices to show that
\begin{quotation}
(5) The element $\dim\left(B^{m}C_{p}\otimes E_{n}\right)\in\pi_{0}E_{n}$
is \emph{non-zero}.
\end{quotation}
Finally, it can be shown that $\dim\left(A\otimes E_{n}\right)$ equals
the Euler characteristic of the $2$-periodic Morava $K$-theory (\lemref{Morava_Dimension})\footnote{We prove this only for $B^{m}C_{p}$, as this suffices for our purposes,
but this is true in general.}
\[
\chi_{n}\left(A\right)=\dim_{\bb F_{p}}\K\left(n\right)_{0}A-\dim_{\bb F_{p}}\K\left(n\right)_{1}A.
\]

Hence, it suffices to prove that
\begin{quotation}
(6) The integer $\chi_{n}\left(B^{m}C_{p}\right)$ is \emph{non-zero}.
\end{quotation}
This is an immediate consequence of the explicit computation of $\K\left(n\right)_{*}(B^{m}C_{p})$,
carried out in \cite{RavenelWilson}. 

We alert the reader that at several points, this outline diverges
from the actual proof we give. Most significantly, we make use of
the fact that the steps (1)--(5) are completely formal and the ideas
involved can be formalized in a much greater generality. Instead of
the functor $E_{n}\widehat{\otimes}\left(-\right)$, we can consider any
colimit preserving symmetric monoidal functor $F\colon\mathcal{C}\to\mathcal{D}$
between stable, $p$-local, symmetric monoidal $\infty$-categories.
Given such a functor $F$, we show how to bootstrap $1$-semiadditivity
to higher semiadditivity under appropriate conditions (\thmref{Bootstrap_Machine}).
This necessitates some technical changes in the argument outlined
above\footnote{In particular, we bypass \cite{BobkovaG} using a somewhat different
and more general argument.}. It is only in the final section that we specialize to $\mathcal{C}=\Sp_{T\left(n\right)}$,
and verify the assumptions of this general criterion.

\subsection{Organization}

We now describe the content of each section of the paper.

In section 2, we develop the axiomatic framework of normed functors
and integration. We begin by developing some general calculus for
this notion and study its functoriality properties. We then study
the interaction of integration with symmetric monoidal structures and
the notion of duality. We conclude with a discussion of the property
of \emph{amenability}.

In section 3, we apply the axiomatic theory of section 2 to the
setting of local systems valued in an $m$-semiadditive $\infty$-category.
We begin by recalling the canonical norm on the pullback functor along
an $m$-finite map (introduced in \cite[Section 4.1]{HopkinsLurie}),
and its interaction with various operations. We then consider $m$-finite
colimit preserving functors between $m$-semiadditive $\infty$-categories
(a.k.a $m$-semiadditive functors), and their behavior with respect
to integration. We continue with studying the interaction of $m$-semiadditivity
with symmetric monoidal structures, duality, and dimension. Finally,
we study the behavior of equivariant powers in $1$-semiadditive $\infty$-categories,
which is used in the sequel in the construction of power operations.

In section 4, we construct the above-mentioned power operations for
$1$-semiadditive stable $\infty$-categories. First, we introduce
the algebraic notion of an additive $p$-derivation and study some
of its properties. We then construct an auxiliary operation $\alpha$
in the presence of $1$-semiadditivity. Specializing to the \emph{stable}
($p$-local) case, we construct from $\alpha$ the additive $p$-derivation $\delta$ and establish its naturality properties. 
Finally, we formulate and
prove the ``bootstrap machine'', that
gives general conditions for a $1$-semiadditive $\infty$-category to be $\infty$-semiadditive. We conclude the section with a discussion of ``nil-conservativity'' which is a natural setup to which one can apply the bootstrap machine. 

In section 5, we apply the abstract theory of sections 2--4 to chromatic
homotopy theory. After some generalities, we use the additive $p$-derivation
of section 4 to derive a generalization of a conjecture of May about nilpotence in $H_{\infty}$-rings. 
We then apply the ``bootstrap machine'' to the $1$-semiadditive $\infty$-category $\Sp_{\T\left(n\right)}$, to show that it is $\infty$-semiadditive, and deduce that  $\T\left(n\right)$-homology of $\pi$-finite spaces depends only on the $n$-th Postnikov truncation. Finally, we consider localizations with respect to general weak rings. We show, among other things, that in this setting $1$-semiadditivity implies $\infty$-semiadditivity, and that various notions of ``bounded height'' coincide.

\subsection{Acknowledgments}

We would like to thank Tobias Barthel, Agn{\`e}s Beaudry, Gijs Heuts,
and Nathaniel Stapleton for useful discussions, and Shay Ben Moshe for his valuable comments on an earlier draft of the manuscript. We would especially
like to thank Michael Hopkins, for suggesting this question and for
useful discussions. Finally, we thank the anonymous referee for his/her careful reading of the manuscript and the many helpful comments and suggestions.  

The first author is partially supported by the Adams Fellowship of
the Israeli Academy of Science. The second author is supported by
the Alon Fellowship and ISF1588/18. The third author is supported
by the ISF grant 1650/15~

The second author would like to thank the Isaac Newton Institute for
Mathematical Sciences, Cambridge, for support and hospitality during
the programme ``Homotopy Harnessing Higher Structures'', where work
on this paper was undertaken. This work was supported by EPSRC grant
no EP/K032208/1.

\subsection{Terminology and Notation}

Throughout the paper we work in the framework of $\infty$-categories (a.k.a. quasicategories), introduced by A. Joyal \cite{joyalquasicat}, and extensively developed by Lurie in \cite{htt} and \cite{ha}. We shall also use the following terminology and notation:
\begin{enumerate}
\item We use the term \emph{isomorphism} for an invertible morphism of an
$\infty$-category (i.e. an equivalence).
\item We say that a space $A$ is 
\begin{enumerate}
\item $\left(-2\right)$-finite, if it is contractible. 
\item $m$-finite for $m\ge-1$, if $\pi_{0}A$ is finite and all the fibers
of the diagonal map $\Delta_{A}\colon A\to A\times A$ are $\left(m-1\right)$-finite
(for $m\ge0$, this is equivalent to $A$ having finitely many components,
each of them $m$-truncated with finite homotopy groups).
\item $\pi$-finite, if it is $m$-finite for some integer $m\ge-2$. 
\end{enumerate}
\item We say that an $\pi$-finite space $A$ is a $p$-space, if all the
homotopy groups of $A$ are $p$-groups.
\item Given a map of spaces $q\colon A\to B$, for every $b\in B$ we denote
by $q^{-1}\left(b\right)$ the homotopy fiber of $q$ over $b$.
\item For $m\ge-2$, we say that a map of spaces $q\colon A\to B$ is $m$-finite (resp. $\pi$-finite) if $q^{-1}(b)$ is $m$-finite (resp. $\pi$-finite) for all $b\in B$.
\item Given an $\infty$-category $\mathcal{C}$, we say that $\mathcal{C}$
admits all $q$-limits (resp. $q$-colimits) if it admits all limits
(resp. colimits) of shape $q^{-1}\left(b\right)$ for all $b\in B$. 
\item Given a functor $F\colon\mathcal{C}\to\mathcal{D}$ of $\infty$-categories,
we say that $F$ preserves $q$-colimits (resp. $q$-limits) if it
preserves all colimits (resp. limits) of shape $q^{-1}\left(b\right)$
for all $b\in B$.
\item We use the notation 
\[
f\colon X\oto gY\oto hZ
\]
to denote that $f\colon X\to Z$ is the composition $h\circ g$ (which
is well defined up to a contractible space of choices). We use similar
notation for composition of more than two morphisms.
\item Given functors $F,G\colon\mathcal{C}\to\mathcal{D}$ and $H,K\colon\mathcal{D}\to\mathcal{E}$,
and natural transformations $\alpha\colon F\to G$ and $\beta\colon H\to K$,
we denote their horizontal composition by $\beta\star\alpha\colon HF\to KG$.
The vertical composition of natural transformations is denoted simply
by juxtaposition.
\item For a symmetric monoidal $\infty$-category $\mathcal{C}$, we denote
by $\calg(\mathcal{C})$ the $\infty$-category of $\mathbb{E}_{\infty}$-algebras
in $\mathcal{C}$. We denote $\cocalg(\mathcal{C})=\calg(\mathcal{C}^{op})^{op}$
the $\infty$-category of $\bb E_{\infty}$-coalgebras in $\mathcal{C}$,
where $\mathcal{C}^{op}$ is endowed with the canonical symmetric
monoidal structure induced from $\mathcal{C}$.
\item For an abelian group $A$ and $k\ge0$, we denote by $B^{k}A$ the
Eilenberg MacLane space with $k$-th homotopy group equal to $A$.
\end{enumerate}

\section{Norms and Integration}
\label{sec:norms_and_integration}
In this section, we develop an abstract formal framework of norms on
functors between $\infty$-categories and the operation of integration
on maps, that such norms induce. This framework abstracts, axiomatizes, and generalizes
the theory of norms and integrals arising from ambidexterity developed
in \cite[Section 4]{HopkinsLurie}. We develop a ``calculus'' for such integrals
and study their functoriality properties and interaction with monoidal
structures. 

\subsection{Normed Functors and Integration}

\subsubsection{Norms and Iso-Norms}

We begin by fixing some terminology regarding adjunctions of $\infty$-categories.
\begin{defn}
Let $F\colon\mathcal{C}\to\mathcal{D}$ be a functor of $\infty$-categories. 
\begin{enumerate}
\item By a \emph{left adjoint} to $F$, we mean a pair $\left(L,u\right)$,
where $L\colon\mathcal{D}\to\mathcal{C}$ is a functor and 
\[
u\colon\Id_{\mathcal{D}}\to F\circ L
\]
is a unit natural transformation in the sense of \cite[Definition 5.2.2.7]{htt}. 
\item By a \emph{right adjoint} to $F$, we mean a pair $\left(R,c\right)$,
where $R\colon\mathcal{D}\to\mathcal{C}$ is a functor and 
\[
c\colon F\circ R\to\Id_{\mathcal{D}}
\]
a counit natural transformation (i.e. satisfying the dual of \cite[Definition 5.2.2.7]{htt}).
\end{enumerate}
\end{defn}

Given a datum of a left adjoint $\left(L,u\right)$, there exists
a map $c\colon L\circ F\to\Id_{\mathcal{C}}$, such that $u$ and
$c$ satisfy the zig-zag identities up to homotopy. From this also
follows that $c$ is a counit map exhibiting $\left(F,c\right)$ as
a right adjoint to $L$. This counit map $c$ is unique up to homotopy,
and we shall therefore sometimes speak of ``the'' associated counit
map (in fact, the space of such maps together with a homotopy witnessing
\emph{one} of the zig-zag identities is contractible \cite[Proposition 4.4.7]{RiehlV}).
We shall similarly speak of the unit map $u\colon\Id_{\mathcal{C}}\to R\circ F$
associated with a right adjoint $\left(R,c\right)$. 

Adjoint functors can be composed in the following (usual) sense:
\begin{defn}
\label{def:Adj_Composition}Given a pair of composable functors
\[
\xymatrix{\mathcal{C}\ar[r]^{F} & \mathcal{D}\ar[r]^{F'} & \mathcal{E}}
,
\]
with left adjoints $\left(L,u\right)$ and $\left(L',u'\right)$ respectively,
the composite map 
\[
u''\colon\Id_{\mathcal{E}}\oto{u'}F'L'\oto uF'FLL',
\]
which is well defined up to homotopy, is a unit map exhibiting $LL'$
as left adjoint to $F'F$. We define the counit map of the composition
of right adjoints in a similar way. 
\end{defn}

The central notion we are about to study in this section is the following:
\begin{defn}
\label{def:Normed_Functor}Given $\infty$-categories $\mathcal{C}$
and $\mathcal{D}$, a \emph{normed functor} 
\[
q\colon\mathcal{D}\nto\mathcal{C},
\]
is a functor $q^{*}\colon\mathcal{C}\to\mathcal{D}$ together with
a left adjoint $\left(q_{!},u_{!}^{q}\right)$, a right adjoint $\left(q_{*},c_{*}^{q}\right)$,
and a natural transformation 
\[
\nm_{q}\colon q_{!}\to q_{*},
\]
which we call a \emph{norm}. We say that $q$ is \emph{iso-normed},
if $\nm_{q}$ is an isomorphism natural transformation. For $X\in \mathcal{C}$, we also write
$X_{q}=q_{!}q^{*}X$, and denote by $c_{!}^{q}\colon q_{!}q^{*}\to\Id$
and $u_{*}^{q}\colon\Id\to q_{*}q^{*}$, the associated counit and
unit of the respective adjunctions. We drop the superscript $q$ whenever
it is clear from the context.
\end{defn}

\begin{rem}
In subsequent sections, we shall sometimes abuse language and refer
to $\nm_{q}$ as a \emph{norm on $q^{*}$} and to $q^{*}$ itself
(with the data of $\nm_{q}$) as a normed functor. Since the left
and right adjoints of $q^{*}$ are essentially unique (when they exist),
this seems to be a rather harmless convention.
\end{rem}

There is a useful criterion for detecting when a normed functor is
iso-normed.
\begin{lem}
\label{lem:Iso_Normed_Criterion} A normed functor $q\colon\mathcal{D}\nto\mathcal{C}$
is iso-normed if and only if the norm $\nm_{q}\colon q_{!}\to q_{*}$
is an isomorphism at $q^{*}X$ for all $X\in\mathcal{C}$.
\end{lem}

\begin{proof}
The ``only if'' part is clear. For the ``if'' part, consider the
two diagrams
\[
\xymatrix@C=3pc{ & q_{!}q^{*}q_{*}\ar[d]_{\nm_{q}}^{\wr}\ar[r]^{\ c_{*}} & q_{!}\ar[d]_{\nm_{q}}\\
q_{*}\ar[r]^{u_{*}\ } & q_{*}q^{*}q_{*}\ar[r]^{\ c_{*}} & q_{*},
}
\qquad \qquad
\xymatrix@C=3pc{q_{!}\ar[d]^{\nm_{q}}\ar[r]^{u_{!}\ } & q_{!}q^{*}q_{!}\ar[d]_{\wr}^{\nm_{q}}\ar[r]^{\ c_{!}} & q_{!}\\
q_{*}\ar[r]^{u_{!}\ } & q_{*}q^{*}q_{!},
}
\]
which commute by naturality of the (co)unit maps. By the zig-zag identities,
the composition along the bottom row in the left diagram is the identity.
Thus, the left diagram shows that $\nm_{q}$ has a right inverse.
Similarly, the right diagram shows that $\nm_{q}$ has a left inverse
and therefore $\nm_{q}$ is an isomorphism.
\end{proof}
Given a functor $q^{*}\colon\mathcal{C}\to\mathcal{D}$ with a left
adjoint $\left(q_{!},u_{!}^{q}\right)$ and a right adjoint $\left(q_{*},c_{*}^{q}\right)$,
the data of a natural transformation $\nm_{q}\colon q_{!}\to q_{*}$
is equivalent to the data of its mate $\nu_{q}\colon q^{*}q_{!}\to\Id$.
Moreover,
\begin{lem}
\label{lem:Norm_Counit} Let $q\colon\mathcal{D}\nto\mathcal{C}$
be a normed functor. For every $Y\in\mathcal{D}$, the map $\nm_{q}\colon q_{!}\to q_{*}$
is an isomorphism at $Y\in\mathcal{D}$ if and only if the mate $\nu_{q}\colon q^{*}q_{!}\to\Id$
is a counit map at $Y$. Namely, for all $X\in\mathcal{C}$, the composition
\[
\map_{\mathcal{C}}\left(X,q_{!}Y\right)\oto{q^{*}}\map_{\mathcal{D}}\left(q^{*}X,q^{*}q_{!}Y\right)\oto{\nu\circ-}\map_{\mathcal{D}}\left(q^{*}X,Y\right)
\]
is a homotopy equivalence.
\end{lem}

\begin{proof}
For every $X\in\mathcal{C}$, consider the commutative diagram in
the homotopy category of spaces:
\[
\xymatrix@C=4pc{\map_{\mathcal{C}}\left(X,q_{!}Y\right)\ar[d]^{q^{*}}\ar[r]^{\nm_{q}\circ-} & \map_{\mathcal{C}}\left(X,q_{*}Y\right)\ar[d]^{q^{*}}\ar[dr]^{\sim}\\
\map_{\mathcal{D}}\left(q^{*}X,q^{*}q_{!}Y\right)\ar[r]^{\nm_{q}\circ-}\ar@/_{2pc}/[rr]_{\nu_{q}\circ-} & \map_{\mathcal{D}}\left(q^{*}X,q^{*}q_{*}Y\right)\ar[r]^{c_{*}\circ-}  & \map_{\mathcal{D}}\left(q^{*}X,Y\right).
}
\]

By the Yoneda lemma, $\nm_{q}$ is an isomorphism at $Y$ if and only
if the top map in the diagram is an isomorphism for all $X\in\mathcal{C}$.
By 2-out-of-3, this is the case if and only if the composition of the top map
and the diagonal map is an isomorphism for all $X$. Since the diagram
commutes, this is if and only if the composition of the left vertical
map with the long bottom map is an isomorphism for all $X$, which
is by definition if and only if $\nu_{q}$ is a counit at $Y$.
\end{proof}
\begin{notation}
When $\nm_{q}$ is an isomorphism at $q^*X$, and hence $\nu_{q}$ is
a counit at $X$, we denote the associated \emph{unit} by $\mu_{q,X}\colon X\to q_{!}q^{*}X=X_{q}$.
If $q$ is iso-normed, we let $\mu_{q}\colon\Id\to q_{!}q^{*}$ be
the unit natural transformation associated with $\nu_{q}$. As usual,
we drop the subscript $q$, whenever the map is understood from the
context.
\end{notation}

\begin{rem}
We will use the two points of view, that of a norm $\nm_{q}\colon q_{!}\to q_{*}$
and that of a ``wrong way counit'' $\nu_{q}\colon q^{*}q_{!}\to\Id$
interchangeably. Each point of view has its own advantages. We note
that the definition using $\nu_{q}$ seems to be slightly more general
as it is available even if $q^{*}$ does not (a priori) admit a right
adjoint. In practice, we are mainly interested in situations where
$\nu_{q}$ is indeed a counit map for an adjunction, exhibiting $q_{!}$
as a right adjoint of $q^{*}$. Thus, the gain in generality is rather
negligible. 
\end{rem}

\begin{defn}
\label{def:Norm_Construction}We define the identity normed functor
and composition of normed functors (up to homotopy) as follows.
\begin{enumerate}
\item (Identity) For every $\infty$-category $\mathcal{C}$, the identity
normed functor $\Id\colon\mathcal{C}\nto\mathcal{C}$ consists of
the identity functor $\Id\colon\mathcal{C}\to\mathcal{C}$ viewed
as a left and right adjoint to itself using the identity natural transformation
$\Id\to\Id$ as the (co)unit map and with the identity natural transformation
$\Id\to\Id$ as the norm. 
\item (Composition) Given a pair of composable normed functors
\[
\xymatrix{\mathcal{E}\ \ar@{>->}[r]^{p} & \mathcal{D}\ \ar@{>->}[r]^{q} & \mathcal{C}}
,
\]
we define their composition $qp\colon\mathcal{E}\nto\mathcal{C}$
by composing the adjunctions (\defref{Adj_Composition})
\[
\left(qp\right)^{*}=p^{*}q^{*},\quad\left(qp\right)_{!}=q_{!}p_{!},\quad\left(qp\right)_{*}=q_{*}p_{*}
\]
and take the norm map to be the horizontal composition of the norms
(the order does not matter) 
\[
q_{!}p_{!}\oto{\nm_{q}}q_{*}p_{!}\oto{\nm_{p}}q_{*}p_{*}.
\]
We denote the norm of the composite by $\nm_{qp}$. If $p$ and $q$
are iso-normed, then so is $qp$. 
\end{enumerate}
\end{defn}

\begin{rem}
\label{rem:Cat_Norm}It is possible to define an $\infty$-category
$\widehat{\cat}_{\infty}^{\nm}$, whose objects are $\infty$-categories
and morphisms are normed functors, such that the above constructions
give the identity morphisms and composition in the homotopy category.
This $\infty$-category captures the higher coherences manifest in
the above definitions. We intend to elaborate on this point in a future
work, but for the purposes of this one, which will not use the higher
coherences in any way, we shall be content with the above explicit
definitions up to homotopy.
\end{rem}

\subsubsection{Integration}

The main feature of iso-normed functors is that they allow us to define
a formal notion of ``integration'' of maps.
\begin{defn}
Let $q\colon\mathcal{D}\nto\mathcal{C}$ be an iso-normed functor.
For every $X,Y\in\mathcal{C}$, we define an \emph{integral} map 
\[
\int\limits _{q}\colon\map_{\mathcal{D}}\left(q^{*}X,q^{*}Y\right)\to\map_{\mathcal{C}}\left(X,Y\right),
\]
which is natural in $X$ and $Y$, as the composition
\[
\map_{\mathcal{D}}\left(q^{*}X,q^{*}Y\right)\oto{q_{*}}\map_{\mathcal{C}}\left(q_{*}q^{*}X,q_{*}q^{*}Y\right)\oto{\nm_{q}^{-1}}\map_{\mathcal{C}}\left(q_{*}q^{*}X,q_{!}q^{*}Y\right)\oto{c_{!}\circ-\circ u_{*}}\map_{\mathcal{C}}\left(X,Y\right).
\]
\end{defn}

\begin{rem}
\label{rem:Integration_Unit}Alternatively, using the \emph{wrong
way unit }$\mu_{q}\colon\Id\to q_{!}q^{*}$, one can define the integral
as the composition
\[
\map_{\mathcal{D}}\left(q^{*}X,q^{*}Y\right)\oto{q_{!}}\map_{\mathcal{C}}\left(q_{!}q^{*}X,q_{!}q^{*}Y\right)\oto{c_{!}\circ-\circ\mu}\map_{\mathcal{C}}\left(X,Y\right).
\]
\end{rem}

As a special case we have
\begin{defn}
Let $q\colon\mathcal{D}\nto\mathcal{C}$ be an iso-normed functor.
For every $X\in\mathcal{C}$, we define a map 
\[
|q|_{X}\colon X\to X
\]
by 
\[
|q|_{X}\coloneqq \int\limits _{q}q^{*}\Id_{X}=\int\limits _{q}\Id_{q^{*}X}.
\]
These are the components of the natural endomorphism $|q|=c_{!}^{q}\circ\mu_{q}$
of $\Id_{\mathcal{C}}$.
\end{defn}

Integration satisfies a form of ``homogeneity''.
\begin{prop}
[Homogeneity]\label{prop:Homogenity}Let $q\colon\mathcal{D}\nto\mathcal{C}$
be an iso-normed functor and let  $X,Y,Z\in\mathcal{C}$.
\begin{enumerate}
\item For all maps $f\colon q^{*}X\to q^{*}Y$ and $g\colon Y\to Z$ we
have
\[
g\circ\left(\int\limits _{q}f\right)=\int\limits _{q}\left(q^{*}g\circ f\right)\quad\in\hom_{h\mathcal{C}}\left(X,Z\right).
\]
\item For all maps $f\colon X\to Y$ and $g\colon q^{*}Y\to q^{*}Z$ we
have
\[
\left(\int\limits _{q}g\right)\circ f=\int\limits _{q}\left(g\circ q^{*}f\right)\quad\in\hom_{h\mathcal{C}}\left(X,Z\right).
\]
\end{enumerate}
\end{prop}

\begin{proof}
For (1), consider the commutative diagram
\[
\xymatrix@C=3pc{X\ar[d]^{\mu}\ar[r]^{\mu} & q_{!}q^{*}X\ar[d]^{f}\ar[r]^{f} & q_{!}q^{*}Y\ar[d]^{g}\ar[r]^{c_{!}} & Y\ar[d]^{g}\\
q_{!}q^{*}X\ar[r]^{f} & q_{!}q^{*}Y\ar[r]^{g} & q_{!}q^{*}Z\ar[r]^{c_{!}} & Z.
}
\]

The composition along the top and then right path is $g\circ\int\limits _{q}f$,
while the composition along the left and then bottom path is $\int\limits _{q}\left(q^{*}g\circ f\right)$
(see \remref{Integration_Unit}). 

For (2), consider the diagram
\[
\xymatrix@C=3pc{X\ar[d]^{f}\ar[r]^{\mu} & q_{!}q^{*}X\ar[d]^{f}\ar[r]^{f} & q_{!}q^{*}Y\ar[d]^{g}\ar[r]^{g} & q_{!}q^{*}Z\ar[d]^{c_{!}}\\
Y\ar[r]^{\mu} & q_{!}q^{*}Y\ar[r]^{g} & q_{!}q^{*}Z\ar[r]^{c_{!}} & Z
}
\]
and apply an analogous argument.
\end{proof}
Integration also satisfies a form of ``Fubini's Theorem''.
\begin{prop}
[Higher Fubini's Theorem]\label{prop:Fubini} Given a pair of composable
iso-normed functors
\[
\xymatrix{\mathcal{E}\ \ar@{>->}[r]^{p} & \mathcal{D}\ \ar@{>->}[r]^{q} & \mathcal{C}}
,
\]
for all $X,Y\in\mathcal{C}$, and $f\colon p^{*}q^{*}X\to p^{*}q^{*}Y$,
we have
\[
\int\limits _{q}\left(\int\limits _{p}f\right)=\int\limits _{qp}f\qquad\in\hom_{h\mathcal{C}}\left(X,Y\right).
\]
\end{prop}

\begin{proof}
Since $q$ and $p$ are iso-normed, we can construct the following
diagram
\[
\xymatrix@C=3pc{ & q_{*}p_{*}p^{*}q^{*}X\ar[r]^{f} & q_{*}p_{*}p^{*}q^{*}Y\ar[d]_{\nm_{p}^{-1}}\ar[r]^{\nm_{qp}^{-1}} & q_{!}p_{!}p^{*}q^{*}Y\ar@{=}[d]\ar[rd]^{c_{!}^{qp}}\\
X\ar[dr]_{u_{*}^{q}}\ar[ur]^{u_{*}^{qp}} &  & q_{*}p_{!}p^{*}q^{*}Y\ar[d]_{c_{!}^{p}}\ar[r]^{\nm_{q}^{-1}} & q_{!}p_{!}p^{*}q^{*}Y\ar[d]_{c_{!}^{p}} & Y.\\
 & q_{*}q^{*}X\ar[uu]_{u_{*}^{p}}\ar[r]^{\int\limits _{p}\negmedspace f} & q_{*}q^{*}Y\ar[r]^{\nm_{q}^{-1}} & q_{!}q^{*}Y\ar[ru]_{c_{!}^{q}}
}
\]
The triangles and the bottom right square commute for formal reasons.
The top right square commutes by the way norms are composed (\defref{Norm_Construction}(2))
and the left rectangle commutes by definition of $\int\limits _{p}f$.
Thus, the composition along the top path, which is $\int\limits _{qp}f$,
is homotopic to the composition along the bottom path, which is $\int\limits _{q}\left(\int\limits _{p}f\right)$.
\end{proof}

\subsection{Ambidextrous Squares \& Beck-Chevalley Conditions}

In this section we study functoriality properties of norms and integrals
and develop further the ``calculus of integration''. 

\subsubsection{Beck-Chevalley Conditions}

We begin by recalling some standard material regarding commuting squares
involving adjoint functors (e.g. see beginning of \cite[Section 7.3.1]{htt}).
A commutative square of functors 
\[
\qquad \quad
\vcenter{
\xymatrix@C=3pc{\mathcal{C}\ar[d]_{q^{*}}\ar[r]^{F_{\mathcal{C}}} & \tilde{\mathcal{C}}\ar[d]^{\tilde{q}^{*}}\\
	\mathcal{D}\ar[r]^{F_{\mathcal{D}}} & \tilde{\mathcal{D}}
}}
\qquad\left(\square\right)
\]

is formally a natural isomorphism
\[
F_{\mathcal{D}}q^{*}\iso\tilde{q}^{*}F_{\mathcal{C}}.
\]

If the vertical functors admit left adjoints $q_{!}\dashv q^{*}$
and $\tilde{q}_{!}\dashv\tilde{q}^{*}$ (suppressing the units),
we get a $\bc_{!}$ (Beck-Chevalley) natural transformation
\[
\beta_{!}\colon\tilde{q}_{!}F_{\mathcal{D}}\oto{u_!^q}
\tilde{q}_{!}F_{\mathcal{D}}q^*q_!\iso
\tilde{q}_{!}\tilde{q}^{*}F_{\mathcal{C}}q_! \oto{c_!^{\tilde{q}}}
F_{\mathcal{C}}q_{!}.
\]
Similarly, if the vertical functors admit right adjoints $q^{*}\dashv q_{*}$
and $\tilde{q}^{*}\dashv\tilde{q}_{*}$, we get a $\bc_{*}$ (Beck-Chevalley)
natural transformation

\[
\beta_{*}\colon F_{\mathcal{C}}q_{*}\oto{u_*^{\tilde{q}}}
\tilde{q}_*\tilde{q}^*F_{\mathcal{C}}q_{*}\iso
\tilde{q}_*F_{\mathcal{D}}q^*q_{*}\oto{c_*^q}
\tilde{q}_{*}F_{\mathcal{D}}.
\]

\begin{defn}
We say that the square $\square$ satisfies the $\text{BC}_{!}$ (resp.
$\text{BC}_{*}$) condition, if $q^{*}$ and $\tilde{q}^{*}$ admit
left (resp. right) adjoints and the map $\beta_{!}$ (resp. $\beta_{*}$)
is an isomorphism.
\end{defn}

\begin{rem}
It may happen that in $\square$, the horizontal functors $F_{\mathcal{C}}$
and $F_{\mathcal{D}}$ also have left or right adjoints. In this case,
there are other BC maps one can write. To avoid confusion, we will
always speak about the BC maps with respect to the \emph{vertical}
functors. 
\end{rem}

Given a commutative square $\square$ as above, we denote $u_{*}=u_{*}^{q}$
and $\tilde{u}_{*}=u_{*}^{\tilde{q}}$ and similarly for other (co)unit
maps (when they exist). It is an easy verification using the zig-zag
identities, that the BC-maps are compatible with these units and counits
in the following sense.
\begin{lem}
\label{lem:BC_Co_Units} Given a commutative square of functors $\square$,
such that $q^{*}$ and $\tilde{q}^{*}$ admit left (resp. right) adjoints,
the following four diagrams commute up to homotopy (when they are
defined)
\[
(1)\xymatrix{ & \red{F_{\mathcal{C}}}{}q_{*}q^{*}\ar[d]^{\beta_{*}}\\
\red{F_{\mathcal{C}}}{}\ar[ru]^{u_{*}}\ar[rd]_{\tilde{u}_{*}} & \tilde{q}_{*}\red{F_{\mathcal{D}}}{}q^{*}\ar[d]^{\wr}\\
 & \tilde{q}_{*}\tilde{q}^{*}\red{F_{\mathcal{C}}}{}
}
\qquad
(2)\quad\xymatrix{\red{F_{\mathcal{D}}}{}q^{*}q_{*}\ar[d]^{\wr}\ar[rd]^{c_{*}}\\
\tilde{q}^{*}\red{F_{\mathcal{C}}}{}q_{*}\ar[d]^{\beta_{*}} & \red{F_{\mathcal{D}}}{}\\
\tilde{q}^{*}\tilde{q}_{*}\red{F_{\mathcal{D}}}{}\ar[ru]_{\tilde{c}_{*}}
}
(3)\xymatrix{ & \tilde{q}^{*}\tilde{q}_{!}\red{F_{\mathcal{D}}}{}\ar[d]^{\beta_{!}}\\
\red{F_{\mathcal{D}}}{}\ar[ru]^{\tilde{u}_{!}}\ar[rd]_{u_{!}} & \tilde{q}^{*}\red{F_{\mathcal{C}}}{}q_{!}\ar[d]^{\wr}\\
 & \red{F_{\mathcal{D}}}{}q^{*}q_{!}
}
\qquad
(4)\quad\xymatrix{\tilde{q}_{!}\tilde{q}^{*}\red{F_{\mathcal{C}}}{}\ar[d]^{\wr}\ar[rd]^{\tilde{c}_{!}}\\
\tilde{q}_{!}\red{F_{\mathcal{D}}}{}q^{*}\ar[d]^{\beta_{!}} & \red{F_{\mathcal{C}}}{}.\\
\red{F_{\mathcal{C}}}{}q_{!}q^{*}\ar[ru]_{c_{!}}
}
\]
\end{lem}

The $\bc$ maps also satisfy some naturality properties with respect
to horizontal and vertical pasting, as well as multiplication and
exponentiation of squares. We begin with horizontal pasting. Given
a commutative diagram of $\infty$-categories and functors
\[
\quad \qquad \qquad
\vcenter{
\xymatrix@C=3pc{\mathcal{C}\ar[d]^{q^{*}}\ar[r]^{F_{\mathcal{C}}} &  \mathcal{\tilde{C}}\ar[d]^{\tilde{q}^{*}}\ar[r]^{G_{\mathcal{C}}} &  \mathcal{\tilde{\tilde{C}}}\ar[d]^{\tilde{\tilde{q}}^{*}}\\
\mathcal{D}\ar[r]^{F_{\mathcal{D}}} & \tilde{\mathcal{D}}\ar[r]^{G_{\mathcal{D}}}  & \mathcal{\tilde{\tilde{D}}},
}}
\qquad\left(*\right)
\]
we call the outer square the \emph{horizontal pasting} of the
left and right small squares. The following is easy to verify.
\begin{lem}
\label{lem:Horizontal_Pasting_Formula}Given a horizontal pasting
diagram $\left(*\right)$ as above,
\begin{enumerate}
\item The $\bc_{!}$-map for the outer square is homotopic to the composition
of the $\bc_{!}$ maps for the left and right squares
\[
\tilde{\tilde{q}}_{!}\red{G_{\mathcal{D}}}{}\red{F_{\mathcal{D}}}{}\to\red{G_{\mathcal{C}}}{}\tilde{q}_{!}\red{F_{\mathcal{D}}}{}\to\red{G_{\mathcal{C}}}{}\red{F_{\mathcal{C}}}{}q_{!}.
\]
\item The $\bc_{*}$-map for the outer square is homotopic to the composition
of the $\bc_{*}$ maps for the left and right squares
\[
\red{G_{\mathcal{C}}}{}\red{F_{\mathcal{C}}}{}q_{*}\to\red{G_{\mathcal{C}}}{}\tilde{q}_{*}\red{F_{\mathcal{D}}}{}\to\tilde{\tilde{q}}_{*}\red{G_{\mathcal{D}}}{}\red{F_{\mathcal{D}}}{}.
\]
\end{enumerate}
\end{lem}

This immediately implies the following horizontal pasting lemma for
$\bc$ conditions.
\begin{cor}
\label{cor:Horizontal_Pasting_BC}Given a horizontal pasting diagram
 $\left(*\right)$ as above, denote by $\square_{L}$, $\square_{R}$
and $\square$, the left, right and outer squares respectively.
\begin{enumerate}
\item If $\square_{L}$ and $\square_{R}$ satisfy the $\bc_{!}$ (resp.
$\bc_{*}$) condition, then so does $\square$.
\item If $\square_{R}$ and $\square$ satisfy the $\bc_{!}$ (resp. $\bc_{*}$)
condition and $G_{\mathcal{C}}$ is conservative, the so does $\square_{L}$.
\end{enumerate}
\end{cor}

We now turn to vertical pasting. Given a commutative diagram of $\infty$-categories
and functors
\[
\qquad \qquad \qquad
\vcenter{
\xymatrix@C=3pc{\mathcal{C}\ar[d]_{q^{*}}\ar[r]^{F_{\mathcal{C}}} &  \mathcal{\tilde{C}}\ar[d]^{\tilde{q}^{*}}\\
\mathcal{D}\ar[d]_{p^{*}}\ar[r]^{F_{\mathcal{D}}} & \tilde{\mathcal{D}}\ar[d]^{\tilde{p}^{*}}\\
\mathcal{E}\ar[r]^{F_{\mathcal{E}}} & \tilde{\mathcal{E}},
}}
\qquad\left(**\right)
\]
we call the big outer square (i.e. rectangle) the \emph{vertical pasting} of the top
and bottom small squares. The following is easy to verify.
\begin{lem}
\label{lem:Vertical_Pasting_Formula}Given a vertical pasting diagram
$\left(**\right)$ as above, 
\begin{enumerate}
\item The $\bc_{!}$-map for the outer square is homotopic to the composition
of the $\bc_{!}$ maps for the top and bottom squares
\[
\tilde{q}_{!}\tilde{p}_{!}\red{F_{\mathcal{E}}}{}\to\tilde{q}_{!}\red{F_{\mathcal{D}}}{}p_{!}\to\red{F_{\mathcal{C}}}{}q_{!}p_{!}.
\]
\item The $\bc_{*}$-map for the outer square is homotopic to the composition
of the $\bc_{*}$ maps for the top and bottom squares
\[
\red{F_{\mathcal{C}}}{}q_{*}p_{*}\to\tilde{q}_{*}\red{F_{\mathcal{D}}}{}p_{*}\to\tilde{q}_{*}\tilde{p}_{*}\red{F_{\mathcal{E}}}{}.
\]
\end{enumerate}
\end{lem}

Again, this immediately implies the following vertical pasting lemma
for $\bc$ conditions.
\begin{cor}
\label{cor:Vertical_Pasting_BC}Given a vertical pasting diagram $\left(**\right)$
as above, denote by $\square_{T}$, $\square_{B}$ and $\square$,
the top, bottom, and outer squares respectively. If $\square_{T}$
and $\square_{B}$ satisfy the $\bc_{!}$ (resp. $\bc_{*}$) condition,
then so does $\square$. 
\end{cor}

Finally, the $\bc$ conditions are also natural with respect to multiplication
and exponentiation.
\begin{lem}
\label{lem:Exponential_Rule_BC}Given a pair of squares corresponding under the adjunction
$
(-)\times\mathcal{E} \dashv \fun(\mathcal{E},-),
$
\[
\qquad \qquad
\vcenter{
\xymatrix@R=2.25pc@C=3pc{\mathcal{C}\times\mathcal{E}\ar[d]_{q^{*}\times\Id}\ar[r]^{F_{\mathcal{C}}}  & \tilde{\mathcal{C}}\ar[d]^{\tilde{q}^{*}}\\
\mathcal{D}\times\mathcal{E}\ar[r]^{F_{\mathcal{D}}} &  \tilde{\mathcal{D}}
}}
\quad\left(\square_{1}\right),\qquad \qquad
\vcenter{
\xymatrix@C=3pc{\mathcal{C}\ar[d]_{q^{*}}\ar[r]^{\widehat{F}_{\mathcal{C}}\quad} &  \fun(\mathcal{E},\tilde{\mathcal{C}})\ar[d]^{\left(q^{*}\right)^{\mathcal{E}}}\\
\mathcal{D}\ar[r]^{\widehat{F}_{\mathcal{D}}\quad} & \fun(\mathcal{E},\tilde{\mathcal{D}})
}}
\quad\left(\square_{2}\right).
\]
the square $\square_{1}$ satisfies the $\bc_{!}$ (resp. $\bc_{*}$)
if and only if $\square_{2}$ satisfies the $\bc_{!}$ (resp. $\bc_{*}$)
condition.
\end{lem}

\begin{proof}
Under the canonical equivalence of $\infty$-categories 
\[
\fun(\mathcal{D}\times\mathcal{E},\tilde{\mathcal{C}})\simeq\fun(\mathcal{D},\fun(\mathcal{E},\tilde{\mathcal{C}})),
\]
the $\text{BC}_{!}$ (resp. $\text{BC}_{*}$) map for $\square_{1}$
corresponds to the $\text{BC}_{!}$ (resp. $\text{BC}_{*}$) map of
$\square_{2}$ and isomorphisms correspond to isomorphisms.
\end{proof}

\subsubsection{Normed and Ambidextrous Squares}

We now consider commuting squares of $\infty$-categories, where the
vertical functors are \emph{normed}.
\begin{defn}
\label{def:Normed_Ambi_Square}We define:
\end{defn}

\begin{enumerate}
\item A \emph{normed square} is a pair of normed functors $q\colon\mathcal{D}\nto\mathcal{C}$
and $\tilde{q}\colon\tilde{\mathcal{D}}\nto\tilde{\mathcal{C}}$,
together with a commutative diagram 
\[
\qquad \quad
\vcenter{
\xymatrix@C=3pc{\mathcal{C}\ar[d]_{q^{*}}\ar[r]^{F_{\mathcal{C}}} & \tilde{\mathcal{C}}\ar[d]^{\tilde{q}^{*}}\\
\mathcal{D}\ar[r]^{F_{\mathcal{D}}} & \tilde{\mathcal{D}}.
}}
\qquad\left(*\right)
\]
It is \emph{iso-normed} if $q$ and $\tilde{q}$ are iso-normed.
\item Given a normed square as in (1), we have an associated \emph{norm-diagram}:
\[
\qquad \quad
\vcenter{
\xymatrix@C=3pc{\red{F_{\mathcal{C}}}{}q_{!}\ar[r]^{\nm_{q}} & \red{F_{\mathcal{C}}}{}q_{*}\ar[d]^{\beta_{*}}\\
\tilde{q}_{!}\red{F_{\mathcal{D}}}{}\ar[u]^{\beta_{!}}\ar[r]^{\nm_{\tilde{q}}} & \tilde{q}_{*}\red{F_{\mathcal{D}}}{}.
}}
\qquad\left(\square\right)
\]
\item A \emph{weakly ambidextrous }square is a normed square, such that
the associated norm diagram $\square$ commutes up to homotopy. An
\emph{ambidextrous square }is a weakly ambidextrous square that is
iso-normed (note that an ambidextrous square satisfies the $\bc_{!}$
condition if and only if it satisfies the $\bc_{*}$ condition).
\end{enumerate}
\begin{rem}
We shall often abuse language and say that $\left(*\right)$ is a
normed (or ambidextrous) square implying by this that we also have
normed functors $q$ and $\tilde{q}$ as in the definition.
\end{rem}

As with any definition regarding norms, we can recast the definition
of an ambidextrous square in terms of wrong way counits. As this will
be used in the sequel, we shall spell this out.
\begin{lem}
\label{lem:Triangle_Unit_Counit_Norm_Diagram} Let $\left(*\right)$
be a normed square as in \defref{Normed_Ambi_Square}(1). Consider
the diagrams (where $\cocone$ is defined only when $\left(*\right)$
is iso-normed).
\[
\qquad \qquad 
\vcenter{
\xymatrix{\red{F_{D}}{}q^{*}q_{!}\ar[rd]^{\nu_{q}}\\
\tilde{q}^{*}\red{F_{\mathcal{C}}}{}q_{!}\ar[u]^{\wr} & \red{F_{D}}\\
\tilde{q}^{*}\tilde{q}_{!}\red{F_{\mathcal{D}}}{}\ar[u]^{\beta_{!}}\ar[ru]_{\nu_{\tilde{q}}}}}
\qquad \left(\cone\right), \qquad \qquad
\vcenter{
\xymatrix{ & \tilde{q}_{!}\tilde{q}^{*}\red{F_{\mathcal{C}}}{}\ar[d]^{\wr}\\
\red{F_{\mathcal{C}}}{}\ar[ru]^{\mu_{\tilde{q}}}\ar[rd]_{\mu_{q}} & \tilde{q}_{!}\red{F_{\mathcal{D}}}{}q^{*}\ar[d]^{\beta_{!}}\\
 & \red{F_{\mathcal{C}}}{}q_{!}q^{*}.
}}
\qquad\left(\cocone\right)
\]
\begin{enumerate}
\item The norm-diagram $\square$ commutes if and only if the diagram $\cone$
commutes.
\item If $\left(*\right)$ is iso-normed, satisfies the $\bc_{!}$ condition
and the norm-diagram $\square$ commutes, then the diagram $\cocone$
commutes.
\end{enumerate}
\end{lem}

\begin{proof}
We begin with (1). The norm-diagram $\square$ commutes if and only
if the two maps $\tilde{q}_{!}F_{\mathcal{D}}\to\tilde{q}_{*}F_{\mathcal{D}}$
are homotopic. This holds if and only if their mates $\tilde{q}^{*}\tilde{q}_{!}F_{\mathcal{D}}\to F_{\mathcal{D}}$
are homotopic. To compute the mate, one applies $\tilde{q}^{*}$ and
post-composes with the counit $\tilde{c}_{*}\colon\tilde{q}^{*}\tilde{q}_{*}\to\Id$
(of the right way adjunction). Now, consider the diagram
\[
\xymatrix@C=3pc{\red{F_{D}}{}q^{*}q_{!}\ar[r]^{\nm_{q}} & \red{F_{D}}{}q^{*}q_{*}\ar[d]^{\wr}\ar[rd]^{c_{*}}\\
\tilde{q}^{*}\red{F_{\mathcal{C}}}{}q_{!}\ar[u]^{\wr}\ar[r]^{\nm_{q}}  & \tilde{q}^{*}\red{F_{\mathcal{C}}}{}q_{*}\ar[d]^{\beta_{*}} & \red{F_{D}}.\\
\tilde{q}^{*}\tilde{q}_{!}\red{F_{\mathcal{D}}}{}\ar[u]^{\beta_{!}}\ar[r]^{\nm_{\tilde{q}}} & \tilde{q}^{*}\tilde{q}_{*}\red{F_{\mathcal{D}}}{}\ar[ru]_{\tilde{c}_{*}}
}
\]
The triangle on the right commutes by \lemref{BC_Co_Units}(2). The
composition of the top maps is $F_{\mathcal{D}}\nu_{q}$ and of the
bottom maps is $\nu_{\tilde{q}}F_{\mathcal{D}}$. Hence, $\square$
commutes, if and only if $\cone$ commutes.

We now turn to (2). To check the commutativity of $\cocone$,
we may replace $\beta_{!}$ with its inverse. By assumption, all maps
in $\square$ are isomorphisms. Thus, the map $\beta_{!}^{-1}$ in
$\cocone$ is homotopic to the composition

\[
\red{F_{\mathcal{C}}}q_{!}\oto{\nm_{q}}\red{F_{\mathcal{C}}}q_{*}\oto{\beta_{*}}\tilde{q}_{*}\red{F_{\mathcal{D}}}\oto{\left(\nm_{\tilde{q}}\right)^{-1}}\tilde{q}_{!}\red{F_{\mathcal{D}}}.
\]
Unwinding the definitions, this exhibits $\beta_{!}^{-1}$ as the $\bc_{*}$-map of the \emph{wrong
way} adjunctions $q^{*}\dashv q_{!}$ and $\tilde{q}^{*}\dashv\tilde{q}_{!}$.
The commutativity of $\cocone$ now follows from the compatibility
of $\bc$-maps with units (\lemref{BC_Co_Units}(1)). 
\end{proof}
The main feature of ambidextrous squares is that they behave well
with respect to the integral operation.
\begin{prop}
\label{prop:Integral_Ambi}Let 
\[
\qquad \quad
\vcenter{
\xymatrix@C=3pc{\mathcal{C}\ar[d]_{q^{*}}\ar[r]^{F_{\mathcal{C}}} & \tilde{\mathcal{C}}\ar[d]^{\tilde{q}^{*}}\\
\mathcal{D}\ar[r]^{F_{\mathcal{D}}} & \tilde{\mathcal{D}}
}}
\qquad\left(\square\right)
\]

be an ambidextrous square that satisfies the $\bc_{!}$ condition
(and hence the $\bc_{*}$ condition). For all $X,Y\in\mathcal{C}$
and $f\colon q^{*}X\to q^{*}Y,$ we have
\[
F_{\mathcal{C}}\left(\int_{q}f\right)=\int_{\tilde{q}}F_{\mathcal{D}}\left(f\right)\quad\in\hom_{h\tilde{\mathcal{C}}}\left(F_{\mathcal{C}}X,F_{\mathcal{C}}Y\right).
\]
In particular, for all $X\in\mathcal{C}$, we have
\[
F_{\mathcal{C}}\left(|q|_{X}\right)=|\tilde{q}|_{F_{\mathcal{C}}\left(X\right)}\quad\in\hom_{h\tilde{\mathcal{C}}}\left(F_{\mathcal{C}}X,F_{\mathcal{C}}X\right).
\]
\end{prop}

\begin{proof}
Since $\square$ is iso-normed, we can construct the following diagram:

\[
\xymatrix@C=3pc{ & \red{F_{\mathcal{C}}}{}q_{*}q^{*}X\ar[d]_{\beta_{*}}^{\wr}\ar[r]^{f} & \red{F_{\mathcal{C}}}{}q_{*}q^{*}Y\ar[d]_{\beta_{*}}^{\wr}\ar[r]^{\nm_{q}^{-1}} & \red{F_{\mathcal{C}}}{}q_{!}q^{*}Y\ar[dr]^{c_{!}}\\
\red{F_{\mathcal{C}}}{}X\ar[ru]^{u_{*}}\ar[dr]_{\tilde{u}_{*}} & \tilde{q}_{*}\red{F_{\mathcal{D}}}{}q^{*}X\ar@{-}[d]^{\wr}\ar[r]^{f} & \tilde{q}_{*}\red{F_{\mathcal{D}}}{}q^{*}Y\ar@{-}[d]^{\wr}\ar[r]^{\nm_{\tilde{q}}^{-1}} & \tilde{q}_{!}\red{F_{\mathcal{D}}}{}q^{*}Y\ar@{-}[d]^{\wr}\ar[u]_{\wr}^{\beta_{!}} & \red{F_{\mathcal{C}}}{}Y.\\
 & \tilde{q}_{*}\tilde{q}^{*}\red{F_{\mathcal{C}}}{}X\ar[r]^{f} & \tilde{q}_{*}\tilde{q}^{*}\red{F_{\mathcal{C}}}{}Y\ar[r]^{\nm_{\tilde{q}}^{-1}} & \tilde{q}_{!}\tilde{q}^{*}\red{F_{\mathcal{C}}}{}Y\ar[ru]_{\tilde{c}_{!}}
}
\]

The left and right triangles commute by the compatibility of $\bc$ maps
with (co)units (\lemref{BC_Co_Units}, diagrams (1), and (4) respectively).
The top right square commutes by the assumption that the square $\square$ is
ambidextrous and satisfies the $\bc$ conditions and the rest of
the squares commute for trivial reasons. Hence, the composition along
the top path is homotopic to the composition along the bottom path,
which proves the first claim. The second claim follows from the first
applied to the map $f=q^{*}\Id_{X}$.
\end{proof}

\subsubsection{Calculus of Normed Squares}

As discussed before, squares of functors can be pasted horizontally
and vertically. We extend these operations to \emph{normed} squares
and consider their compatibility with the notion of ambidexterity.
We begin with horizontal pasting. Given normed functors 
\[
q\colon\mathcal{D}\nto\mathcal{C},\quad\tilde{q}\colon\tilde{\mathcal{D}}\nto\tilde{\mathcal{C}},\quad\tilde{\tilde{q}}\colon\mathcal{\tilde{\tilde{D}}}\nto\tilde{\tilde{\mathcal{C}}},
\]

and a commutative diagram
\[
\qquad \qquad 
\vcenter{
\xymatrix@C=3pc{\mathcal{C}\ar[d]^{q^{*}}\ar[r]^{F_{\mathcal{C}}}  & \mathcal{\tilde{C}}\ar[d]^{\tilde{q}^{*}}\ar[r]^{G_{\mathcal{C}}}  & \mathcal{\tilde{\tilde{C}}}\ar[d]^{\tilde{\tilde{q}}^{*}}\\
\mathcal{D}\ar[r]^{F_{\mathcal{D}}}  & \tilde{\mathcal{D}}\ar[r]^{G_{\mathcal{D}}}  & \mathcal{\tilde{\tilde{D}}},
}}
\qquad\left(*\right)
\]

we call the big outer \emph{normed} square the \emph{horizontal pasting}
of the left and right small \emph{normed} squares. We have the following
horizontal pasting lemma for ambidexterity.
\begin{lem}
[Horizontal Pasting]\label{lem:Horizontal_Pasting_Ambi}Let $\left(*\right)$
be a horizontal pasting diagram of normed squares as above. We denote
by $\square_{L}$, $\square_{R}$ and $\square$, the left, right,
and outer normed squares respectively. If $\square_{L}$ and $\square_{R}$
are (weakly) ambidextrous, then so is $\square$.
\end{lem}

\begin{proof}
Consider the following diagram composed of whiskerings of the norm
diagrams of $\square_{L}$ and $\square_{R}$ (with all horizontal
maps the respective $\bc$-maps).
\[
\xymatrix{\red{G_{\mathcal{C}}}{}\red{F_{\mathcal{C}}}{}q_{!}\ar[d]^{\nm_{q}} & \red{G_{\mathcal{C}}}{}\tilde{q}_{!}\red{F_{\mathcal{D}}}{}\ar[d]^{\nm_{\tilde{q}}}\ar[l] & \tilde{\tilde{q}}_{!}\red{G_{\mathcal{D}}}{}\red{F_{\mathcal{D}}}{}\ar[d]^{\nm_{\tilde{\tilde{q}}}}\ar[l]\\
\red{G_{\mathcal{C}}}{}\red{F_{\mathcal{C}}}{}q_{*}\ar[r] & \red{G_{\mathcal{C}}}{}\tilde{q}_{*}\red{F_{\mathcal{D}}}{}\ar[r] & \tilde{\tilde{q}}_{*}\red{G_{\mathcal{D}}}{}\red{F_{\mathcal{D}}}.
}
\]
By \lemref{Horizontal_Pasting_Formula}, the outer square is the norm
diagram for $\square$, which implies the claim.
\end{proof}
We now turn to vertical pasting. Given normed functors
\[
q\colon\mathcal{D}\nto\mathcal{C},\quad\tilde{q}\colon\tilde{\mathcal{D}}\nto\mathcal{\tilde{C}},\quad p\colon\mathcal{E}\nto\mathcal{D},\quad\tilde{p}\colon\tilde{\mathcal{E}}\nto\tilde{\mathcal{D}}
\]
and a commutative diagram 
\[
\qquad \qquad
\vcenter{
\xymatrix@C=3pc{\mathcal{C}\ar[d]_{q^{*}}\ar[r]^{F_{\mathcal{C}}}  & \mathcal{\tilde{C}}\ar[d]^{\tilde{q}^{*}}\\
\mathcal{D}\ar[d]_{p^{*}}\ar[r]^{F_{\mathcal{D}}}  & \tilde{\mathcal{D}}\ar[d]^{\tilde{p}^{*}}\\
\mathcal{E}\ar[r]^{F_{\mathcal{E}}}  & \tilde{\mathcal{E}},
}}
\qquad\left(**\right)
\]

we call the big outer normed square, with respect to the compositions
of normed functors $qp$ and $\tilde{q}\tilde{p}$, the \emph{vertical
pasting} of the top and bottom small normed squares. We have the following
vertical pasting lemma for ambidexterity.
\begin{lem}
[Vertical Pasting]\label{lem:Vertical_Pasting_Ambi}Let $\left(**\right)$
be a vertical pasting diagram of normed squares as above. We denote
by $\square_{T}$, $\square_{B}$ and $\square$, the top, bottom, and outer normed squares respectively. If $\square_{T}$ and $\square_{B}$
are (weakly) ambidextrous, then so is $\square$.
\end{lem}

\begin{proof}
Consider the following diagram composed of whiskerings of the norm
diagrams of $\square_{T}$ and $\square_{B}$ (with all horizontal
maps the respective $\bc$-maps).
\[
\xymatrix{\red{F_{\mathcal{C}}}{}q_{!}p_{!}\ar[d]^{\nm_{q}} & \tilde{q}_{!}\red{F_{\mathcal{D}}}{}p_{!}\ar[d]^{\nm_{\tilde{q}}}\ar[l] & \tilde{q}_{!}\tilde{p}_{!}\red{F_{\mathcal{E}}}{}\ar[d]^{\nm_{\tilde{q}}}\ar[l]\\
\red{F_{\mathcal{C}}}{}q_{*}p_{!}\ar[d]^{\nm_{p}}\ar[r] & \tilde{q}_{*}\red{F_{\mathcal{D}}}{}p_{!}\ar[d]^{\nm_{p}} & \tilde{q}_{*}\tilde{p}_{!}\red{F_{\mathcal{E}}}{}\ar[d]^{\nm_{\tilde{p}}}\ar[l]\\
\red{F_{\mathcal{C}}}{}q_{*}p_{*}\ar[r] & \tilde{q}_{*}\red{F_{\mathcal{D}}}{}p_{*}\ar[r] & \tilde{q}_{*}\tilde{p}_{*}\red{F_{\mathcal{E}}}.
}
\]
By \lemref{Vertical_Pasting_Formula}, the outer diagram is the norm
diagram for $\square$. Thus, it is enough to check that all four
small squares commute. The top right and bottom left squares commute
for trivial reasons. The top left and bottom right squares are whiskerings
of the norm diagrams of $\square_{T}$ and $\square_{B}$ respectively
and hence commute by assumption.
\end{proof}

\subsection{Monoidal Structure and Duality}

In this section, we study the interaction of norms and integration
with (symmetric) monoidal structures on the source and target $\infty$-categories.
Under suitable hypotheses, this interaction allows us to reduce questions
about ambidexterity to questions about duality.

\subsubsection{Tensor Normed Functors}
\begin{defn}
\label{def:Tensor_Normed_Functor}Let $\mathcal{C}$ and $\mathcal{D}$
be monoidal $\infty$-categories. A \emph{$\otimes$-normed functor}
from $\mathcal{D}$ to $\mathcal{C}$, is a normed functor $q\colon\mathcal{D}\nto\mathcal{C}$,
such that $q^{*}$ is monoidal (and hence $q_{!}$ is colax monoidal
by the dual of \cite[Corollary 7.3.2.7]{ha}) and for all $Y\in\mathcal{D}$
and $X\in\mathcal{C}$, the compositions of the canonical maps 
\[
q_{!}\left(Y\otimes\left(q^{*}X\right)\right)\to\left(q_{!}Y\right)\otimes\left(q_{!}q^{*}X\right)\oto{\Id\otimes c_{!}}\left(q_{!}Y\right)\otimes X
\]
and
\[
q_{!}\left(\left(q^{*}X\right)\otimes Y\right)\to\left(q_{!}q^{*}X\right)\otimes\left(q_{!}Y\right)\oto{c_{!}\otimes\Id}X\otimes\left(q_{!}Y\right)
\]
are isomorphisms.
\end{defn}

\begin{rem}
The above definition does not depend on the norm and is actually just
a property of the functor $q^{*}$. However, we shall only be interested
in this property in the context of normed functors.
\end{rem}

\begin{notation}
To make diagrams involving (co)units more readable, we shall employ
the following graphical convention. When writing a unit map of an
adjunction whiskered by some functors, we enclose in parenthesis the
effected terms in the target. Similarly, when writing a counit map
of an adjunction whiskered by some functors, we underline the effected
terms in the source.
\end{notation}

We adopt the definitions and terminology of \cite{HopkinsLurie} regarding
duality in monoidal $\infty$-categories. In the situation of \defref{Tensor_Normed_Functor}, substituting $q^*\one_{\mathcal{C}}$ for $Y$, gives a natural isomorphism from the functor $q_{!}q^{*}$ to the functor $\one_{q}\otimes-$,
where $\one_{q}=q_{!}q^{*}\one_{\mathcal{C}}$. We can therefore consider
the map
\[
\varepsilon\colon\one_{q}\otimes\one_{q}\simeq q_{!}\underline{q^{*}q_{!}}q^{*}\one_{\mathcal{C}}\oto{\nu}\underline{q_{!}q^{*}}\one_{\mathcal{C}}\oto{c_{!}}\one_{\mathcal{C}}.
\]

\begin{prop}
\label{prop:Iso_Normed_Duality}Let $q\colon\mathcal{D}\nto\mathcal{C}$
be a $\otimes$-normed functor of monoidal $\infty$-categories. The following are equivalent:
\begin{enumerate}
\item $\nm_{q}$ is an isomorphism natural transformation (i.e. $q$ is
iso-normed).
\item $\nm_{q}$ is an isomorphism at $q^{*}\one_{\mathcal{C}}$. 
\item The map $\varepsilon\colon\one_{q}\otimes\one_{q}\to\one_{\mathcal{C}}$
is a duality datum (exhibiting $\one_{q}$ as a self dual object in
$\mathcal{C}$).
\end{enumerate}
\end{prop}

\begin{proof}
(1) $\implies$ (2) is obvious. Assume (2). The map $\nm_{q}\colon q_{!}\to q_{*}$
has a mate $\nu\colon q^{*}q_{!}\to\Id$. By \lemref{Norm_Counit},
since $\nm_{q}$ is an isomorphism at $q^{*}\one_{\mathcal{C}}$,
the map $\nu$ is a counit map at $q^{*}\one_{\mathcal{C}}$ and has
an associated unit map $\mu_{\one}\colon\one_{\mathcal{C}}\to q_{!}q^{*}\one_{\mathcal{C}}$.
Let
\[
\eta\colon\one_{\mathcal{C}}\oto{\mu_{\one}}\left(q_{!}q^{*}\right)\one_{\mathcal{C}}\oto{u_{!}}q_{!}\left(q^{*}q_{!}\right)q^{*}\one_{\mathcal{C}}=\one_{q}\otimes\one_{q}.
\]

We prove (3) by showing that $\varepsilon$ and $\eta$ satisfy the
zig-zag identities. As above, we identify $\one_{q}$ with $q_{!}q^{*}\one_{\mathcal{C}}$
and $\one_{q}\otimes\one_{q}$ with $q_{!}q^{*}q_{!}q^{*}\one_{\mathcal{C}}$.
For the first zig-zag identity, consider the diagram
\[
\xymatrix{q_{!}q^{*}\one_{\mathcal{C}}\ar[rr]^{\mu_{\one}}\ar@{=}[rrd]\ar@/^{2pc}/[rrr]^{\eta\star\Id} &  & (q_{!}\underline{q^{*})q_{!}}q^{*}\one_{\mathcal{C}}\ar[r]^{u_{!}\quad}\ar[d]^{\nu} & q_{!}\left(q^{*}q_{!}\right)\underline{q^{*}q_{!}}q^{*}\one_{\mathcal{C}}\ar[d]^{\nu}\ar@/^{2.5pc}/@<1ex>[dd]^{\Id\star\varepsilon}\\
 &  & q_{!}q^{*}\one_{\mathcal{C}}\ar[r]^{u_{!}\quad}\ar@{=}[rd] & q_{!}(q^{*}\underline{q_{!})q^{*}}\one_{\mathcal{C}}\ar[d]^{c_{!}}\\
 &  &  & q_{!}q^{*}\one_{\mathcal{C}}.
}
\]
The square commutes by the interchange law for natural transformations. The upper triangle by
the definition of $\mu_{\one}$ (i.e. the corresponding zig-zag identity
at $\one_{\mathcal{C}}$) and the bottom by the zig-zag identities
for $u_{!}$ and $c_{!}$. For the second zig-zag identity, consider
a similar diagram
\[
\xymatrix{q_{!}q^{*}\ar[rr]^{\mu_{\one}}\one_{\mathcal{C}}\ar@{=}[rrd]\ar@/^{2pc}/[rrr]^{\Id\star\eta} &  & q_{!}\underline{q^{*}(q_{!}}q^{*})\one_{\mathcal{C}}\ar[r]^{u_{!}\quad}\ar[d]^{\nu} & q_{!}\underline{q^{*}q_{!}}\left(q^{*}q_{!}\right)q^{*}\one_{\mathcal{C}}\ar[d]^{\nu}\ar@/^{2.5pc}/@<1ex>[dd]^{\varepsilon\star\Id}\\
 &  & q_{!}q^{*}\ar[r]^{u_{!}\quad}\one_{\mathcal{C}}\ar@{=}[rd] & \underline{q_{!}(q^{*}}q_{!})q^{*}\one_{\mathcal{C}}\ar[d]^{c_{!}}\\
 &  &  & q_{!}q^{*}\one_{\mathcal{C}}.
}
\]

Assume (3). By \lemref{Iso_Normed_Criterion} and \lemref{Norm_Counit},
it is enough to show that $\nu$ is a counit at $q^{*}X$ for all
$X\in\mathcal{C}$. Consider the following diagram
\[
\scalebox{0.9}{
\xymatrix{ & \map\left(Y,\one_{q}\otimes X\right)\ar[dl]_{\sim}\ar[rr]^{\one_{q}\otimes-} &  & \map\left(\one_{q}\otimes Y,\one_{q}\otimes\one_{q}\otimes X\right)\ar[ld]_{\sim}\ar[dd]^{\varepsilon\circ-}\\
\map\left(Y,q_{!}q^{*}X\right)\ar[r]^{q^{*}}\ar@{-->}[rd] & \map\left(q^{*}Y,\underline{q^{*}q_{!}}q^{*}X\right)\ar[d]^{\nu\circ-}\ar[r]^{q_{!}} & \map\left(q_{!}q^{*}Y,q_{!}\underline{q^{*}q_{!}}q^{*}X\right)\ar[d]^{\nu\circ-}\\
 & \map\left(q^{*}Y,q^{*}X\right)\ar[r]^{q_{!}}\ar[rd]_{\sim} & \map\left(q_{!}q^{*}Y,\underline{q_{!}q^{*}}X\right)\ar[d]^{c_{!}\circ-} & \map\left(\one_{q}\otimes Y,X\right)\ar[ld]_{\sim}\\
 &  & \map\left(q_{!}q^{*}Y,X\right)
}}
\]

The triangles commute by definition and the rest by naturality. The
composition along the top and then right path is an isomorphism since
$\varepsilon$ is an evaluation map of a duality datum on $\one_{q}$.
Thus, the dashed arrow is an isomorphism by 2-out-of-3, which proves
that $\nu$ is a counit at $q^{*}X$. 
\end{proof}
\begin{rem}
A similar result is given in \cite[Proposition 5.1.8]{HopkinsLurie}.
\end{rem}

\subsubsection{Tensor Normed Squares}

The following is the analogous notion to a normed square in the monoidal
setting.
\begin{defn}
A \emph{$\otimes$-normed square} is a pair of $\otimes$-normed functors
$q\colon\mathcal{D}\nto\mathcal{C}$ and $\tilde{q}\colon\tilde{\mathcal{D}}\nto\tilde{\mathcal{C}}$
and a commutative square of monoidal $\infty$-categories and monoidal
functors
\[
\qquad \quad
\vcenter{
\xymatrix@C=3pc{\mathcal{C}\ar[d]_{q^{*}}\ar[r]^{F_{\mathcal{C}}} & \tilde{\mathcal{C}}\ar[d]^{\tilde{q}^{*}}\\
\mathcal{D}\ar[r]^{F_{\mathcal{D}}} & \tilde{\mathcal{D}}.
}}
\qquad\left(*\right)
\]

For a $\otimes$-normed square $\left(*\right)$ as above, we define
a colax natural transformation of functors
\[
\theta\colon\left(-\right)_{\tilde{q}}F_{\mathcal{C}}=\tilde{q}_{!}\tilde{q}^{*}F_{\mathcal{C}}\simeq\tilde{q}_{!}F_{\mathcal{D}}q^{*}\oto{\beta_{!}}F_{\mathcal{C}}q_{!}q^{*}=F_{\mathcal{C}}\left(-\right)_{q}.
\]
Using the isomorphisms from \defref{Tensor_Normed_Functor} we define
the natural isomorphisms

\[
L_{q}\colon\left(X\otimes Y\right)_{q}=q_{!}q^{*}\left(X\otimes Y\right)\simeq q_{!}\left(q^{*}X\otimes q^{*}Y\right)\iso q_{!}q^{*}X\otimes Y=X_{q}\otimes Y,
\]
\[
R_{q}\colon\left(X\otimes Y\right)_{q}=q_{!}q^{*}\left(X\otimes Y\right)\simeq q_{!}\left(q^{*}X\otimes q^{*}Y\right)\iso X\otimes q_{!}q^{*}Y=X\otimes Y_{q}.
\]

We shall need a technical lemma regarding the compatibility of the
maps $L$, $R$, and $\theta$.
\end{defn}

\begin{lem}
\label{lem:Tensor_Square_Compatibility}Let $\left(*\right)$ be a
$\otimes$-normed square as above. For all $X,Y\in\mathcal{C}$, the
following diagram:
\[
\scalebox{0.95}{
\xymatrix{F_{\mathcal{C}}\left(X\otimes Y\right)_{\tilde{q}\tilde{q}}\ar[d]^{\theta_{X\otimes Y}}\ar@{-}[r]^{\sim} & \left(F_{\mathcal{C}}\left(X\right)\otimes F_{\mathcal{C}}\left(Y\right)\right)_{\tilde{q}\tilde{q}}\ar[r]^{R_{\tilde{q}}} & \left(F_{\mathcal{C}}\left(X\right)\otimes F_{\mathcal{C}}\left(Y\right)_{\tilde{q}}\right)_{\tilde{q}}\ar[d]^{\Id\otimes\theta_{Y}}\ar[r]^{L_{\tilde{q}}} & F_{\mathcal{C}}\left(X\right)_{\tilde{q}}\otimes F_{\mathcal{C}}\left(Y\right)_{\tilde{q}}\ar[d]^{\Id\otimes\theta_{Y}}\\
F_{\mathcal{C}}\left(\left(X\otimes Y\right)_{q}\right)_{\tilde{q}}\ar[d]^{\theta_{\left(X\otimes Y\right)_{q}}}\ar[r]^{R_{q}} & F_{\mathcal{C}}\left(X\otimes Y_{q}\right)_{\tilde{q}}\ar[d]^{\theta_{X\otimes Y_{q}}}\ar@{-}[r]^{\sim} & \left(F_{\mathcal{C}}\left(X\right)\otimes F_{\mathcal{C}}\left(Y_{q}\right)\right)_{\tilde{q}}\ar[r]^{L_{\tilde{q}}} & F_{\mathcal{C}}\left(X\right)_{\tilde{q}}\otimes F_{\mathcal{C}}\left(Y_{q}\right)\ar[d]^{\theta_{X}\otimes\Id}\\
F_{\mathcal{C}}\left(\left(X\otimes Y\right)_{qq}\right)\ar[r]^{R_{q}} & F_{\mathcal{C}}\left(\left(X\otimes Y_{q}\right)_{q}\right)\ar[r]^{L_{q}} & F_{\mathcal{C}}\left(X_{q}\otimes Y_{q}\right)\ar@{-}[r]^{\sim} & F_{\mathcal{C}}\left(X_{q}\right)\otimes F_{\mathcal{C}}\left(Y_{q}\right)
}}
\]
 commutes up to homotopy.
\end{lem}

\begin{proof}
The top right square commutes by naturality of $L_{\tilde{q}}$ and
the bottom left square commutes by naturality of $\theta$. We now
show the commutativity of the top left rectangle (the commutativity
of the bottom right rectangle is completely analogous). By unwinding
the definition of $R_{q}$, the top left rectangle is obtained by
applying $\left(-\right)_{\tilde{q}}$ to the following diagram
\[
\scalebox{0.95}{
\xymatrix{\left(F_{\mathcal{C}}\left(X\right)\otimes F_{\mathcal{C}}\left(Y\right)\right)_{\tilde{q}}\ar@{-}[d]^{\wr}\ar[r]\ar@/^{2pc}/[rr]^{R_{\tilde{q}}} & F_{\mathcal{C}}\left(X\right)_{\tilde{q}}\otimes F_{\mathcal{C}}\left(Y\right)_{\tilde{q}}\ar[d]^{\theta_{X}\otimes\theta_{Y}}\ar[r]^{\tilde{c}_{!}\otimes\Id} & F_{\mathcal{C}}\left(X\right)\otimes F_{\mathcal{C}}\left(Y\right)_{\tilde{q}}\ar[d]^{\Id\otimes\theta_{Y}}\\
F_{\mathcal{C}}\left(X\otimes Y\right)_{\tilde{q}}\ar[d]^{\theta_{X\otimes Y}} & F_{\mathcal{C}}\left(X_{q}\right)\otimes F_{\mathcal{C}}\left(Y_{q}\right)\ar@{-}[d]^{\wr}\ar[r]^{c_{!}\otimes\Id} & F_{\mathcal{C}}\left(X\right)\otimes F_{\mathcal{C}}\left(Y_{q}\right)\ar@{-}[d]^{\wr}\\
F_{\mathcal{C}}\left(\left(X\otimes Y\right)_{q}\right)\ar[r]\ar@/_{2pc}/[rr]_{R_{q}} & F_{\mathcal{C}}\left(X_{q}\otimes Y_{q}\right)\ar[r]^{c_{!}\otimes\Id} & F_{\mathcal{C}}\left(X\otimes Y_{q}\right).
}}
\]
The left rectangle commutes by the monoidality of $\theta$ and the
bottom right square commutes by naturality. The top right square is
a tensor product of two squares 
\[
\qquad \quad
\vcenter{
\xymatrix@C=3pc{F_{\mathcal{C}}\left(X\right)_{\tilde{q}}\ar[d]^{\theta_{X}}\ar[r]^{\tilde{c}_{!}} & F_{\mathcal{C}}\left(X\right)\ar[d]^{\Id}\\
F_{\mathcal{C}}\left(X_{q}\right)\ar[r]^{c_{!}} & F_{\mathcal{C}}\left(X\right)
}}
\qquad\left(\square_{1}\right),\qquad \qquad
\vcenter{
\xymatrix@C=3pc{F_{\mathcal{C}}\left(Y\right)_{\tilde{q}}\ar[d]^{\theta_{Y}}\ar[r]^{\Id} & F_{\mathcal{C}}\left(Y\right)_{\tilde{q}}\ar[d]^{\theta_{Y}}\\
F_{\mathcal{C}}\left(Y_{q}\right)\ar[r]^{\Id} & F_{\mathcal{C}}\left(Y_{q}\right).
}}
\qquad\left(\square_{2}\right)
\]
 The square $\square_{2}$ commutes for trivial reasons and the square
$\square_{1}$ commutes by the compatibility of $\bc$-maps with counits
(\lemref{BC_Co_Units}(4)).
\end{proof}
The main fact we shall use about $\otimes$-normed squares is the
following:
\begin{prop}
\label{prop:Tensor_Ambi}Let $\left(*\right)$ be a $\otimes$-normed square
as above. Assume that $\left(*\right)$ is weakly ambidextrous and
satisfies the $\bc_{!}$-condition. If $q$ is iso-normed, then $\tilde{q}$
is iso-normed and the $\bc_{*}$ condition is satisfied as well.
\end{prop}

\begin{proof}
By the assumption of the $\bc_{!}$-condition, the operation $\theta$
is an isomorphism. Observe that $\one_{\tilde{q}}\simeq F_{\mathcal{C}}\left(\one_{\mathcal{C}}\right)_{\tilde{q}}$
and consider the following diagram:

\[
\xymatrix@C=3pc{\one_{\tilde{q}}\otimes\one_{\tilde{q}}\ar[d]_{\theta\otimes\theta}^{\wr}\ar[r]^-{\quad L_{\tilde{q}}^{-1}R_{\tilde{q}}^{-1}\quad} & \tilde{q}_{!}\underline{\tilde{q}^{*}\tilde{q}_{!}}\tilde{q}^{*}F_{\mathcal{C}}\left(\one_{\mathcal{C}}\right)\ar[dd]^{\theta\star\theta}\ar[r]^-{\tilde{\nu}} & \underline{\tilde{q}_{!}\tilde{q}^{*}}F_{\mathcal{C}}\left(\one_{\mathcal{C}}\right)\ar[dd]^{\theta}\ar[rd]^{\tilde{c}_{!}}\\
F_{\mathcal{C}}\left(\one_{q}\right)\otimes F_{\mathcal{C}}\left(\one_{q}\right)\ar[d]^{\wr} &  &  & F_{\mathcal{C}}\left(\one_{\mathcal{C}}\right)\simeq\one_{\tilde{\mathcal{C}}}.\\
F_{\mathcal{C}}\left(\one_{q}\otimes\one_{q}\right)\ar[r]^-{L_{q}^{-1}R_{q}^{-1}} & F_{\mathcal{C}}\left(q_{!}\underline{q^{*}q_{!}}q^{*}\one_{\mathcal{C}}\right)\ar[r]^-{\nu} & F_{\mathcal{C}}\left(\underline{q_{!}q^{*}}\one_{\mathcal{C}}\right)\ar[ru]_{c_{!}}
}
\]
The middle rectangle and the triangle commute by the compatibility
of BC maps with counits (\lemref{BC_Co_Units}(4)). The left rectangle
commutes by applying \lemref{Tensor_Square_Compatibility} with $X=Y=\one_{\mathcal{C}}$.
By \propref{Iso_Normed_Duality}, $\varepsilon_{q}\colon\one_{q}\otimes\one_{q}\to\one_{\mathcal{C}}$
is a duality datum and since $F_{\mathcal{C}}$ is monoidal, 
\[
F_{\mathcal{C}}\left(\varepsilon_{q}\right)\colon F_{\mathcal{C}}\left(\one_{q}\right)\otimes F_{\mathcal{C}}\left(\one_{q}\right)\to F_{\mathcal{C}}\left(\one_{\mathcal{C}}\right)\simeq\one_{\tilde{\mathcal{C}}}
\]
is a duality datum as well. The commutativity of the above diagram,
identifies $F_{\mathcal{C}}\left(\varepsilon_{q}\right)$ with $\varepsilon_{\tilde{q}}$
and hence $\varepsilon_{\tilde{q}}$ is a duality datum for $\one_{\tilde{q}}$.
By \propref{Iso_Normed_Duality} again, $\tilde{q}$ is iso-normed.
Finally, the $\bc_{*}$ condition is satisfied by 2-out-of-3 for the
norm diagram.
\end{proof}

\subsection{Amenability}
\begin{defn}
An iso-normed functor $q\colon\mathcal{D}\nto\mathcal{C}$ is called
\emph{amenable}, if $|q|$ is an isomorphism natural transformation. 
\end{defn}

\begin{rem}
The name is inspired by the notion of amenability in geometric group
theory. Given an object $X\in\mathcal{C}$, the integral operation
\[
\int\limits _{q}\colon\map\left(q^{*}X,q^{*}X\right)\to\map\left(X,X\right)
\]
can be thought of intuitively as ``\emph{summation} over the fibers
of $q$''. Amenability allows us to ``\emph{average} over the fibers
of $q$'' by multiplying the integral with $|q|^{-1}$.
This is especially suggestive in the prototypical example of local-systems,
which we study in the next section.
\end{rem}

\begin{lem}
\label{lem:Amenable_Conservative}Let 
\[
\xymatrix@C=3pc{\mathcal{C}\ar[d]_{q^{*}}\ar[r]^{F_{\mathcal{C}}} & \tilde{\mathcal{C}}\ar[d]^{\tilde{q}^{*}}\\
\mathcal{D}\ar[r]^{F_{\mathcal{D}}} & \tilde{\mathcal{D}}
}
\]
be an ambidextrous square, such that $F_{\mathcal{C}}$ is conservative.
If $\tilde{q}$ is amenable, then $q$ is amenable.
\end{lem}

\begin{proof}
Given $X\in\mathcal{C}$, since the square is ambidextrous, we have
by \propref{Integral_Ambi},
\[
F_{\mathcal{C}}\left(|q|_{X}\right)=|\tilde{q}|_{F_{\mathcal{C}}\left(X\right)}.
\]
The claim follows from the assumption that $F_{\mathcal{C}}$ is conservative.
\end{proof}
The next result demonstrates how can amenability be profitably used
for ``averaging''.
\begin{thm}
[Higher Maschke's Theorem]\label{thm:Amenable_Section} Let $q\colon\mathcal{D}\nto\mathcal{C}$
be an iso-normed functor. If $q$ is amenable, then for every $X\in\mathcal{C}$
the counit map $c_{!}\colon q_{!}q^{*}X\to X$ has a section (i.e.
left inverse) up to homotopy. In particular, every object of $\mathcal{C}$
is a retract of an object in the essential image of $q_{!}$.
\end{thm}

\begin{proof}
Let $X\in\mathcal{C}$. By the zig-zag identities, 
\[
q^{*}X\oto{u_{!}}(q^{*}\underline{q_{!})q^{*}}X\oto{c_{!}}q^{*}X
\]
is the identity on $q^{*}X$. Integrating along $q$ and using \propref{Homogenity}(1),
we get
\[
|q|_{X}=\int\limits _{q}\Id_{\left(q^{*}X\right)}=\int\limits _{q}\left(q^{*}\left(c_{!}\right)_{X}\circ\left(u_{!}\right)_{q^{*}X}\right)=\left(c_{!}\right)_{X}\circ\int\limits _{q}\left(u_{!}\right)_{q^{*}X}.
\]
Hence, if $|q|_{X}$ is an isomorphism, then $\left(c_{!}\right)_{X}$
has a section up to homotopy.
\end{proof}
\begin{thm}
[Cancellation Theorem]\label{thm:Ambi_Cancellation}Let
\[
\xymatrix{\mathcal{E}\ \ar@{>->}[r]^{p} & \mathcal{D}\ \ar@{>->}[r]^{q} & \mathcal{C}}
\]
be a pair of normed functors. If $p$ is amenable and $qp$
is iso-normed, then $q$ is iso-normed.
\end{thm}

\begin{proof}
The map $\nm_{qp}$ is given by the composition
\[
q_{!}p_{!}\oto{\nm_{q}}q_{*}p_{!}\oto{\nm_{p}}q_{*}p_{*}.
\]
Since $\nm_{qp}$ and $\nm_{p}$ are isomorphisms, so is $q_{!}p_{!}\oto{\nm_{q}}q_{*}p_{!}$.
By \thmref{Amenable_Section}, every $X\in\mathcal{D}$ is a retract
of $p_{!}Y$ for some $Y\in\mathcal{E}$. Isomorphisms are closed
under retracts, and so $\nm_{q}$ is an isomorphism for every $X\in\mathcal{D}$.
\end{proof}
\begin{rem}
This is essentially the same argument as the one used in the proof
of \cite[Proposition 4.4.16]{HopkinsLurie}.
\end{rem}

\section{Local-Systems and Ambidexterity}
\label{sec:local_systems}
The main examples of normed functors that we are interested in are
the ones provided by the theory of higher semiadditivity developed in \cite{HopkinsLurie}
and further in \cite{Harpaz}. In what follows, we first briefly recall
the relevant definitions and explain how they fit into the abstract
framework developed in the previous section. Then we apply the theory
of the previous section to this special case. The theory developed
in \cite{HopkinsLurie} is set up in a rather general framework of
Beck-Chevalley fibrations. Even though this framework fits into our
theory of normed functors, for concreteness and clarity, we shall
confine ourselves to the special case of local systems.

\subsection{Local-Systems and Canonical Norms}

Let $\mathcal{C}$ be an $\infty$-category and let $A$ be a space
viewed as an $\infty$-groupoid. We call $\fun\left(A,\mathcal{C}\right)$
the $\infty$-category of $\mathcal{C}$-valued \emph{local systems}
on $A$. Let $q\colon A\to B$ be a map of spaces and assume that
$\mathcal{C}$ admits all $q$-limits and $q$-colimits. The functor of precomposition with $q$, denote by
\[
q^{*}\colon\fun\left(B,\mathcal{C}\right)\to\fun\left(A,\mathcal{C}\right),
\]
admits both a left adjoint $q_{!}$ and a right adjoint $q_{*}$ (given
by left and right Kan extension respectively). We shall define, after
\cite[Section 4.1]{HopkinsLurie}, a class of \emph{weakly $\mathcal{C}$-ambidextrous}
maps $q$, to which we associate a canonical norm map $\nm_{q}\colon q_{!}\to q_{*}$.
This norm map gives rise to a normed functor
\[
q^{\can}\colon\fun\left(A,\mathcal{C}\right)\nto\fun\left(B,\mathcal{C}\right).
\]

A map $q$ is called \emph{$\mathcal{C}$-ambidextrous }if it is weakly
$\mathcal{C}$-ambidextrous and the associated canonical norm is an
isomorphism (i.e. $q^{\can}$ is iso-normed). 

\subsubsection{Base Change \& Canonical Norms}

We begin with some terminology regarding the operation of base change
for local-systems.
\begin{defn}
\label{def:Base_Change_Square}Given an $\infty$-category $\mathcal{C}$
and a pullback diagram of spaces 
\[
\qquad \quad
\vcenter{
\xymatrix@C=3pc{\tilde{A}\ar[d]_{\tilde{q}}\ar[r]^{s_{A}} & A\ar[d]^{q}\\
\tilde{B}\ar[r]^{s_{B}} & B
}}
\qquad\left(*\right)
\]
the associated \emph{base-change square} (of $\mathcal{C}$-valued
local-systems) is 
\[
\qquad \quad
\vcenter{
\xymatrix@C=3pc{\fun\left(B,\mathcal{C}\right)\ar[d]_{q^{*}}\ar[r]^{s_{B}^{*}} & \fun(\tilde{B},\mathcal{C})\ar[d]^{\tilde{q}^{*}}\\
\fun\left(A,\mathcal{C}\right)\ar[r]^{s_{A}^{*}} & \fun(\tilde{A},\mathcal{C}).
}}
\qquad\left(\square\right)
\]
\end{defn}

\begin{lem}
\label{lem:Base_Change_BC}Let $\mathcal{C}$ be an $\infty$-category
and let $\left(*\right)$ be a pullback diagram of spaces as in \defref{Base_Change_Square}
above. If $\mathcal{C}$ admits all $q$-colimits (resp. $q$-limits),
then the associated base-change square $\square$ satisfies the $\bc_{!}$
(resp. $\bc_{*}$ condition).
\end{lem}

\begin{proof}
For $\bc_{!}$ this is \cite[Proposition 4.3.3]{HopkinsLurie} (note
we only need $q$-colimits). For $\bc_{*}$ a completely analogous
argument works.
\end{proof}
The construction of the canonical norm rests on the following more
general construction.
\begin{defn}
\label{def:Diagonal_Induction}Let $q\colon A\to B$ be a map of spaces
and let $\delta\colon A\to A\times_{B}A$ be the diagonal of $q$. Let $\mathcal{C}$
be an $\infty$-category that admits all $q$-(co)limits and $\delta$-(co)limits.
Given an isomorphism natural transformation 
\[
\nm_{\delta}\colon\delta_{!}\iso\delta_{*},
\]
 we define the \emph{diagonally induced} norm map 
\[
\nm_{q}\colon q_{!}\to q_{*}
\]

as follows. Consider the commutative diagram
\[
\xymatrix@C=3pc{A\ar[rd]^{\delta}\ar@/^{1pc}/@{=}[rrd]\ar@/_{1pc}/@{=}[rdd]\\
 & A\times_{B}A\ar[d]^{\pi_{2}}\ar[r]^{\pi_{1}} & A\ar[d]^{q}\\
 & A\ar[r]^{q} & B.
}
\]

To the iso-norm $\nm_{\delta}$, corresponds a wrong way unit map
$\mu_{\delta}\colon\Id\to\delta_{!}\delta^{*}$. By \lemref{Base_Change_BC},
the base change square associated with $\left(*\right)$ satisfies
the $\bc_{!}$ condition, and so we can define the composition 
\[
\nu_{q}\colon q^{*}q_{!}\oto{\beta_{!}^{-1}}\left(\pi_{2}\right)_{!}\pi_{1}^{*}\oto{\mu_{\delta}}\left(\pi_{2}\right)_{!}\delta_{!}\delta^{*}\pi_{1}^{*}\iso\Id.
\]
We define $\nm_{q}\colon q_{!}\to q_{*}$ to be the mate of $\nu_{q}$
under the adjunction $q^{*}\dashv q_{*}$.
\end{defn}

\begin{rem}
\label{rem:Diagonal_Induction_Integral}In light of \cite[Remark 4.1.9]{HopkinsLurie},
we can informally say that the diagonally induced norm map on $q$
is obtained by integrating the identity map along the diagonal $\delta$.
Though we shall not use this perspective, it is helpful to keep it
in mind.
\end{rem}

Note that if $q\colon A\to B$ is $m$-truncated for some $m\ge-1$,
then $\delta$ is $\left(m-1\right)$-truncated. This allows us to
define canonical norm maps inductively on the level of truncatedness
of the map. 
\begin{defn}
\label{def:Ambidexterity}Let $\mathcal{C}$ be an $\infty$-category
and $m\ge-2$ an integer. A map of spaces $q\colon A\to B$ is called 
\begin{enumerate}
\item \emph{weakly $m$-$\mathcal{C}$-ambidextrous}, if $q$ is $m$-truncated,
$\mathcal{C}$ admits $q$-(co)limits and either of the two holds:
\begin{itemize}
\item $m=-2$, in which case the inverse of $q^{*}$ is both a left and
right adjoint of $q^{*}$. We define the \emph{canonical norm} map
on $q^{*}$ to be the identity of some inverse of $q^{*}$. 
\item $m\ge-1$, and the diagonal $\delta\colon A\to A\times_{B}A$ of $q$
is \emph{$\left(m-1\right)$-$\mathcal{C}$-ambidextrous}.\emph{ }In
this case we define the \emph{canonical norm} on $q^{*}$ to be the
diagonally induced one from the canonical norm of $\delta$.
\end{itemize}
\item \emph{$m$-$\mathcal{C}$-ambidextrous}, if it is weakly $m$-$\mathcal{C}$-ambidextrous
and its canonical norm map is an isomorphism.
\end{enumerate}
A map of spaces $q\colon A\to B$ is called (weakly) $\mathcal{C}$-ambidextrous
if it is (weakly) $m$-$\mathcal{C}$-ambidextrous for some $m$. 
\end{defn}

By \cite[Proposition 4.1.10 (5)]{HopkinsLurie}, the canonical norm
associated with a map $q\colon A\to B$, that is $m$-truncated for
some $m$, is independent of $m$. 
\begin{defn}
In the situation of \defref{Ambidexterity}, given a  map $q\colon A\to B$ that is weakly $\mathcal{C}$-ambidextrous, we define the associated \emph{canonical normed
functor} 
\[
q_{\mathcal{C}}^{\can}\colon\fun\left(A,\mathcal{C}\right)\nto\fun\left(B,\mathcal{C}\right),
\]
by 
\[
\left(q_{\mathcal{C}}^{\can}\right)^{*}=q^{*},\quad\left(q_{\mathcal{C}}^{\can}\right)_{!}=q_{!},\quad\left(q_{\mathcal{C}}^{\can}\right)_{*}=q_{*},
\]
and the norm map $\nm_{q}\colon q_{!}\to q_{*}$ the canonical norm
of \defref{Ambidexterity}. 
\end{defn}

Note that the normed functor $q_{\mathcal{C}}^{\can}$ is iso-normed
if and only if $q$ is $\mathcal{C}$-ambidextrous. We add the following
definition.
\begin{defn}
Let $\mathcal{C}$ be an $\infty$-category. A $\mathcal{C}$-ambidextrous
map $q\colon A\to B$ is called \emph{$\mathcal{C}$-amenable} if
$q^{\can}$ is amenable.
\end{defn}

\begin{notation}
\label{not:Canonical_Norm_Notation}Given a weakly $\mathcal{C}$-ambidextrous
map of spaces $q\colon A\to B$, we write $q^{\can}$ for $q_{\mathcal{C}}^{\can}$
if $\mathcal{C}$ is understood from the context. We also write $\left(-\right)_{q}$,
$\int_{q}$ and $|q|$ instead of $\left(-\right)_{q^{\can}},$
$\int_{q^{\can}}$ and $|q^{\can}|$. For a map $q\colon A\to\pt$,
we shall also say that $A$ is (weakly) $\mathcal{C}$-ambidextrous
or amenable if $q$ is, and write $\left(-\right)_{A}$, $\int_{A}$
, and $|A|$ instead of $\left(-\right)_{q}$, $\int_{q}$
and $|q|$.
\end{notation}

The next proposition ensures that the canonical norms are preserved
under base change, compositions and identity as in \defref{Norm_Construction}.
\begin{prop}
\label{prop:Canonical_Norms_Compatibility}Let $\mathcal{C}$ be an
$\infty$-category.
\begin{enumerate}
\item (Identity) Given an isomorphism of spaces $q\colon A\iso B$, the
functor $q^{*}$ is $\mathcal{C}$-ambidextrous and its canonical
norm is the identity of the left and right adjoint inverse of $q^{*}$.
\item (Composition) Given (weakly) $\mathcal{C}$-ambidextrous maps $q\colon A\to B$
and $p\colon B\to C$, the composition $pq\colon A\to C$ is (weakly)
$\mathcal{C}$-ambidextrous and $\left(pq\right)^{\can}$ can be identified
with $p^{\can}q^{\can}$.
\item (Base-change) Let $\left(*\right)$ be a pullback diagram of spaces
as in \defref{Base_Change_Square}. If $q$ is (weakly) $\mathcal{C}$-ambidextrous,
then $\tilde{q}$ is (weakly) $\mathcal{C}$-ambidextrous and the
associated base-change square 
\[
\xymatrix@C=3pc{\fun\left(B,\mathcal{C}\right)\ar[d]_{q^{*}}\ar[r]^{s_{B}^{*}} & \fun(\tilde{B},\mathcal{C})\ar[d]^{\tilde{q}^{*}}\\
\fun\left(A,\mathcal{C}\right)\ar[r]^{s_{A}^{*}} & \fun(\tilde{A},\mathcal{C})
}
\]
is (weakly) ambidextrous.
\end{enumerate}
\end{prop}

\begin{proof}
(1) follows directly from the definition. (2) is the content of \cite[Remark 4.2.4]{HopkinsLurie}.
(3) is a restatement of \cite[Remark 4.2.3]{HopkinsLurie}.
\end{proof}
The following is a central notion for this paper.
\begin{defn}\label{def:Semiadd}
Let $m\ge-2$ be an integer. An $\infty$-category $\mathcal{C}$
is called $m$-semiadditive, if it admits all $m$-finite limits and
$m$-finite colimits and every $m$-finite map of spaces is $\mathcal{C}$-ambidextrous.
It is called $\infty$-semiadditive if it is $m$-semiadditive for
all $m$.
\end{defn}

\begin{rem}
Our definition of $m$-semiadditivity agrees with \cite[Definition 3.1]{Harpaz}
and differs slightly from \cite[Definition 4.4.2]{HopkinsLurie} in
that we do not require $\mathcal{C}$ to admit \emph{all} small colimits,
but only $m$-finite ones. Note that using the ``wrong way counit''
perspective, one could phrase $m$\textendash semiadditivity without
the \emph{assumption} that $\mathcal{C}$ admits $m$-finite limits,
but this would then be a direct \emph{consequence}. Thus, \defref{Semiadd}
is somewhat more general then \cite[Definition 4.4.2]{HopkinsLurie}. 
\end{rem}

\subsubsection{Base Change and Integration}

We can now apply the theory of integration developed in the previous
section to the canonically normed functors associated with ambidextrous
maps.
\begin{example}
\label{exa:Integral_Sum}(see \cite[Remark 4.4.11]{HopkinsLurie})
Let $\mathcal{C}$ be a 0-semiadditive $\infty$-category (e.g. $\mathcal{C}$ is stable). For every \emph{finite set} $A$, the map $q\colon A\to\pt$ is $\mathcal{C}$-ambidextrous.
Given $X,Y\in\mathcal{C}$, a map of local systems $f\colon q^{*}X\to q^{*}Y$,
can be viewed as a collection of maps $\left\{ f_{a}\colon X\to Y\right\} _{a\in A}$.
We have 
\[
\int\limits _{A}f=\sum_{a\in A}f_{a}\quad\in\hom_{h\mathcal{C}}\left(X,Y\right).
\]
\end{example}

Applying the general theory of integration to base change squares,
we get
\begin{prop}
\label{prop:Base_Change_Ambi}Let $\mathcal{C}$ be an $\infty$-category
and let $\left(*\right)$ be a pullback diagram of spaces as in \defref{Base_Change_Square},
such that $q$ (and hence $\tilde{q}$) is $\mathcal{C}$-ambidextrous.
For all $X,Y\in\fun\left(B,\mathcal{C}\right)$ and $f\colon q^{*}X\to q^{*}Y$,
we have
\[
s_{B}^{*}\int\limits _{q}f=\int\limits _{\tilde{q}}s_{A}^{*}f\quad\in\hom_{h\fun\left(\tilde{B},\mathcal{C}\right)}\left(s_{B}^{*}X,s_{B}^{*}Y\right).
\]
In particular, for all $X\in\fun\left(B,\mathcal{C}\right)$ we have
\[
s_{B}^{*}|q|_{X}=|\tilde{q}|_{s_{B}^{*}X}\quad\in\hom_{h\fun\left(\tilde{B},\mathcal{C}\right)}\left(s_{B}^{*}X,s_{B}^{*}X\right).
\]
\end{prop}

\begin{proof}
Denote by $\square$ the associated base-change square. By \propref{Canonical_Norms_Compatibility}(3),
$\square$ is ambidextrous and by \lemref{Base_Change_BC}, it satisfies
the $\bc_{!}$ condition. Now, the result follows from \propref{Integral_Ambi}.
\end{proof}
As a consequence, we get a form of ``distributivity'' for integration.
\begin{cor}
\label{cor:Distributivity}Let $\mathcal{C}$ be an $\infty$-category
and let $q_{1}\colon A_{1}\to B$ and $q_{2}\colon A_{2}\to B$ be
two $\mathcal{C}$-ambidextrous maps of spaces. Consider the pullback
square
\[
\xymatrix@C=3pc{A_{2}\times_{B}A_{1}\ar[d]_{\pi_{2}}\ar[r]^{\pi_{1}}\ar[dr]|{q_{2}\times_{B}q_{1}} & A_{1}\ar[d]^{q_{1}}\\
A_{2}\ar[r]^{q_{2}} & B.
}
\]
The map $q_{2}\times_{B}q_{1}$ is $\mathcal{C}$-ambidextrous and
for all $X,Y,Z\in\fun\left(B,\mathcal{C}\right)$ and maps 
\[
f_{1}\colon q_{1}^{*}X\to q_{1}^{*}Y,\quad f_{2}\colon q_{2}^{*}Y\to q_{2}^{*}Z,
\]
we have
\[
\int\limits _{q_{2}\times_{B}q_{1}}\left(\pi_{2}^{*}f_{2}\circ\pi_{1}^{*}f_{1}\right)=\int\limits _{q_{2}}f_{2}\circ\int\limits _{q_{1}}f_{1}\quad\in\hom_{h\fun\left(B,\mathcal{C}\right)}\left(X,Z\right).
\]

In particular, for every $X\in\fun\left(B,\mathcal{C}\right)$, we
have
\[
|q_{2}\times_{B}q_{1}|_{X}=|q_{2}|_{X}\circ|q_{1}|_{X}\quad\in\hom_{h\fun\left(B,\mathcal{C}\right)}\left(X,X\right).
\]
\end{cor}

\begin{proof}
The map $\pi_{2}$ is $\mathcal{C}$-ambidextrous by \propref{Canonical_Norms_Compatibility}(3)
and therefore $q_{2}\times_{B}q_{1}=q_{2}\pi_{2}$ is $\mathcal{C}$-ambidextrous
by \propref{Canonical_Norms_Compatibility}(2). We now start from
the left hand side and use \propref{Fubini}, \propref{Homogenity}(1),
\propref{Base_Change_Ambi} and \propref{Homogenity}(2) (in that
order).

\[
\int\limits _{q_{2}\times_{B}q_{1}}\left(\pi_{2}^{*}f_{2}\circ\pi_{1}^{*}f_{1}\right)=\int\limits _{q_{2}\pi_{2}}\left(\pi_{2}^{*}f_{2}\circ\pi_{1}^{*}f_{1}\right)=\int\limits _{q_{2}}\int\limits _{\pi_{2}}\left(\pi_{2}^{*}f_{2}\circ\pi_{1}^{*}f_{1}\right)=
\]
\[
\int\limits _{q_{2}}\left(f_{2}\circ\int\limits _{\pi_{2}}\pi_{1}^{*}f_{1}\right)=\int\limits _{q_{2}}\left(f_{2}\circ q_{2}^{*}\int\limits _{q_{1}}f_{1}\right)=\int\limits _{q_{2}}f_{2}\circ\int\limits _{q_{1}}f_{1}.
\]

The second claim follows from applying the first to $f_{2}=q_{2}^{*}\Id_{X}$
and $f_{1}=q_{1}^{*}\Id_{X}$.
\end{proof}
As another consequence, we obtain the additivity property of the integral.
\begin{prop}[Integral Additivity]
\label{prop:Integral_Additivity}Let $\mathcal{C}$ be a $0$-semiadditive
$\infty$-category and let $q_i\colon A_i\to B$ for $i=1,\dots,k$ be $\mathcal{C}$-ambidextrous maps. Then,
\[
\left(q_{1},\dots,q_{k}\right)\colon A_{1}\sqcup\cdots\sqcup A_{k}\to B
\]
is $\mathcal{C}$-ambidextrous and for all $X,Y\in\fun\left(B,\mathcal{C}\right)$ and maps $f_{i}\colon q_{i}^{*}X\to q_{i}^{*}Y$ for $i=1,\dots,k$,
we have 
\[
\int\limits _{\left(q_{1},\dots,q_{k}\right)}\left(f_{1},\dots,f_{k}\right)=\sum_{i=1}^{k}\left(\int\limits _{q_{i}}f_{i}\right)\quad\in\hom_{h\fun\left(B,\mathcal{C}\right)}\left(X,Y\right).
\]
\end{prop}

\begin{proof}
By induction, we may assume $k=2$. Write $\left(q_{1},q_{2}\right)$ as a composition
\[
A_{1}\sqcup A_{2}\oto{q_{1}\sqcup q_{2}}B\sqcup B\oto{\nabla}B,
\]
where $\nabla$ is the fold map. By \cite[Proposition 4.3.5]{HopkinsLurie}, the map $q_1\sqcup q_2$ is $\mathcal{C}$-ambidextrous. Consider the pullback square of spaces, with $j_1$ the natural inclusion inclusion of the first summand,
\[
\qquad \quad
\vcenter{
\xymatrix@C=3pc{A_{1}\ar[d]_{q_{1}}\ar[r]^-{j_{1}} & A_{1}\sqcup A_{2}\ar[d]^{q_{1}\sqcup q_{2}}\\
B\ar[r]^-{j_{1}} & B\sqcup B.
}}
\qquad\left(*\right)
\]
By \propref{Base_Change_Ambi} applied to the base-change square of
$\left(*\right)$, we get that 
\[
j_{1}^{*}\left(\int\limits _{q_{1}\sqcup q_{2}}\left(f_{1},f_{2}\right)\right)\simeq\int\limits _{q_{1}}f_{1}.
\]
Applying the analogous argument to the second component, we get
\[
\int\limits _{q_{1}\sqcup q_{2}}\left(f_{1},f_{2}\right)=\left(\int\limits _{q_{1}}f_{1},\int\limits _{q_{2}}f_{2}\right).
\]

Since $\nabla\colon B\sqcup B\to B$ is $0$-finite and $\mathcal{C}$ is $0$-semiadditive, $\nabla$ is $\mathcal{C}$-ambidextrous and the map $(q_1,q_2)$ is $\mathcal{C}$-ambidextrous as a composition of two such (\propref{Canonical_Norms_Compatibility}(2)). Using Fubini's Theorem (\propref{Fubini}), and a direct computation from the definition of the integral over $\nabla$ (identical to \exaref{Integral_Sum}) we get
\[
\int\limits _{\left(q_{1},q_{2}\right)}\left(f_{1},f_{2}\right)\simeq\int\limits _{\nabla}\int\limits _{q_{1}\sqcup q_{2}}\left(f_{1},f_{2}\right)=\int\limits _{\nabla}\left(\int\limits _{q_{1}}f_{1},\int\limits _{q_{2}}f_{2}\right)=\int\limits _{q_{1}}f_{1}+\int\limits _{q_{2}}f_{2}.
\]
\end{proof}

\subsubsection{Amenable Spaces}

Ambidexterity of the base-change square has also a corollary for the
notion of amenability.
\begin{cor}
\label{cor:Amenable_Base_Change}Let $\mathcal{C}$ be an $\infty$-category
and let $\left(*\right)$ be a pullback diagram of spaces as in \defref{Base_Change_Square}.
If $s_{B}$ is surjective on connected components and $\tilde{q}$
is $\mathcal{C}$-amenable, then $q$ is $\mathcal{C}$-amenable.
\end{cor}

\begin{proof}
Since $s_B$ is surjective on connected components, the $\mathcal{C}$-ambidexterity of $\tilde{q}$ implies the $\mathcal{C}$-ambidexterity of $q$ by \cite[Corollary 4.3.6]{HopkinsLurie}. Thus, by \propref{Canonical_Norms_Compatibility}(3), the diagram $\square$ of \defref{Base_Change_Square} is ambidextrous. Since $s_{B}$ is surjective on connected components, $s_{B}^{*}$ is conservative and the claim follwos from \lemref{Amenable_Conservative}. 
\end{proof}
The following two propositions give the core properties of amenable
spaces.
\begin{prop}
\label{prop:Amenable_Space}Let $\mathcal{C}$ be an $\infty$-category
and let $A\to E\oto pB$ be a fiber sequence of weakly $\mathcal{C}$-ambidextrous
spaces, where $B$ is connected. If $E$ is $\mathcal{C}$-ambidextrous
and $A$ is $\mathcal{C}$-amenable, then $B$ is $\mathcal{C}$-ambidextrous.
\end{prop}

\begin{proof}
By assumption, $A$ is $\mathcal{C}$-amenable and $B$ is connected,
hence by \corref{Amenable_Base_Change}, the map $p$ is $\mathcal{C}$-amenable.
Denote $q\colon B\to\pt$ and consider the pair of composable canonically
normed functors 
\[
\xymatrix{\fun\left(E,\mathcal{C}\right)\ \ar@{>->}[r]^{p^{\can}} & \fun\left(B,\mathcal{C}\right)\ \ar@{>->}[r]^{q^{\can}} & \fun\left(\pt,\mathcal{C}\right).}
\]

Since $p^{\can}$ is amenable and $\left(qp\right)^{\can}=q^{\can}p^{\can}$
is iso-normed, by \thmref{Ambi_Cancellation}, $q^{\can}$ is iso-normed.
In other words, the map $q$ (namely, the space $B$) is $\mathcal{C}$-ambidextrous.
\end{proof}
\begin{prop}
\label{prop:Amenable_Contractible}Let $\mathcal{C}$ be an $\infty$-category
and let $A$ be a connected space, such that $\mathcal{C}$ admits
all $A$-(co)limits and $\Omega A$-(co)limits. Denoting $q\colon A\to\pt$,
if $\Omega A$ is $\mathcal{C}$-amenable, then the counit map 
\[
c_{!}^{q}\colon q_{!}q^{*}\to\Id,
\]
is an isomorphism.
\end{prop}

\begin{proof}
Let $e\colon\pt\to A$ be a choice of a base point. The composition
\[
\Id=q_{!}\underline{e_{!}e^{*}}q^{*}\oto{c_{!}^{e}}\underline{q_{!}q^{*}}\oto{c_{!}^{q}}\Id
\]

is the counit of the adjunction 
\[
\Id=q_{!}e_{!}\dashv e^{*}q^{*}=\Id,
\]
and hence an isomorphism. Thus, the whiskering $q_{!}c_{!}^{e}q^{*}$
is a left inverse of $c_{!}^{q}$ up to isomorphism. It therefore
suffices to show that $c_{!}^{e}$ has itself a left inverse. Since
$A$ is connected and $\Omega A$ is $\mathcal{C}$-amenable, the
map $e$ is $\mathcal{C}$-amenable by \corref{Amenable_Base_Change}.
Thus, by \thmref{Amenable_Section}, the map $c_{!}^{e}$ has a left
inverse.
\end{proof}

\subsubsection{Higher Semiadditivity \& Spans }

We conclude with recalling from \cite{Harpaz} some results regarding
the universality of spans of $m$-finite spaces among $m$-semiadditive
$\infty$-categories. These results are useful in reducing questions
about general $m$-semiadditive categories to the universal case,
in which they are sometimes easier to solve.

Let $\mathcal{S}_{m}\ss\mathcal{S}$ be the full subcategory spanned
by $m$-finite spaces and let $\mathcal{S}_{m}^{m}$ be the $\infty$-category
of spans of $m$-finite spaces (see \cite{Barwick}). Roughly, 
\begin{itemize}
\item The objects of $\mathcal{S}_{m}^{m}$ are $m$-finite spaces. 
\item A morphism from $A$ to $B$ is a span $A\from E\to B$, where $E$
is $m$-finite as well. 
\item Composition, up to homotopy, is given by pullback of spans. 
\end{itemize}
By \cite[Section 2.2]{Harpaz}, the $\infty$-category $\mathcal{S}_{m}^{m}$
of spans of $m$-finite spaces inherits a symmetric monoidal structure
from the Cartesian symmetric monoidal structure on $\mathcal{S}_{m}$.
While this symmetric monoidal structure is not itself Cartesian, the
unit is $\pt\in\mathcal{S}_{m}^{m}$ and the tensor of two maps $A_{1}\xleftarrow{q_{1}}E_{1}\oto{r_{1}}B_{1}$
and $A_{2}\xleftarrow{q_{2}}E_{2}\oto{r_{2}}B_{2}$ is equivalent
to 
\[
A_{1}\times A_{2}\xleftarrow{q_{1}\times q_{2}}E_{1}\times E_{2}\oto{r_{1}\times r_{2}}B_{1}\times B_{2}.
\]

One of the main results of \cite{Harpaz} is that $\mathcal{S}_{m}^{m}$
canonically acts on any $m$-semiadditive $\infty$-category (and
the existence of such an action is in fact equivalent to $m$-semiadditivity).
Formally,
\begin{thm}
[Harpaz, {{\cite[Corollary 5.2]{Harpaz}}}]\label{thm:=00005BHarpaz=00005DSpan_Action}
For every $m$-semiadditive $\mathcal{C}$, there is a unique monoidal
$m$-finite colimit preserving functor $\mathcal{S}_{m}^{m}\to\fun\left(\mathcal{C},\mathcal{C}\right)$. 
\end{thm}

Unwinding the definition of this action, we get that
\begin{itemize}
\item The image of an $m$-finite space $a\colon A\to\pt$ is equivalent
to the functor 
\[
\left(-\right)_{A}=a_{!}a^{*}\colon\mathcal{C}\to\mathcal{C}
\]
(i.e. colimit over the constant $A$-shaped diagram).
\item The image of a ``right way'' arrow $A\xleftarrow{=}A\oto rB$ is
homotopic to the right way counit map
\[
\left(-\right)_{A}=a_{!}a^{*}\simeq b_{!}\underline{r_{!}r^{*}}b^{*}\oto{c_{!}^{r}}b_{!}b^{*}=\left(-\right)_{B},
\]
where $a\colon A\to\pt$ and $b\colon B\to\pt$ are the unique maps
(i.e. it is the natural map induced on colimits).
\item The image of a ``wrong way'' arrow $B\xleftarrow{q}A\xrightarrow{=}A$
is homotopic to the wrong way unit map 
\[
\left(-\right)_{B}=b_{!}b^{*}\oto{\mu_{q}}b_{!}\left(q_{!}q^{*}\right)b^{*}\simeq a_{!}a^{*}=\left(-\right)_{A}
\]
(which can informally be thought of as ``integration along the fibers
of $q$'').
\item The natural transformation $|A|$ at $\pt\in \mathcal{S}_m^m$, is given by the span $\pt \from A \to \pt.$
\end{itemize}
\begin{rem}
If one is only interested in this functor on the level of homotopy
categories (as we are),
\[
h\mathcal{S}_{m}^{m}\to h\fun\left(\mathcal{C},\mathcal{C}\right),
\]
one can use the above formulas as a \emph{definition}. The compatibility
with composition can be verified using \cite[Proposition 4.2.1 (2)]{HopkinsLurie}.
\end{rem}

\subsection{Higher Semiadditive Functors}

In this section, we study $m$-finite colimit preserving functors between
$m$-semiadditive $\infty$-categories and study their behavior with
respect to integration. We call such functors\emph{ $m$-semiadditive.}
\begin{defn}
Let $F\colon\mathcal{C}\to\mathcal{D}$ be a functor of $\infty$-categories
and $q\colon A\to B$ a map of spaces. We define the \emph{$\left(F,q\right)$-square
}to be the commutative square
\[
\xymatrix@C=3pc{\fun\left(B,\mathcal{C}\right)\ar[d]_{q^{*}}\ar[r]^{F_{*}} & \fun\left(B,\mathcal{D}\right)\ar[d]^{q^{*}}\\
\fun\left(A,\mathcal{C}\right)\ar[r]^{F_{*}} & \fun\left(A,\mathcal{D}\right),
}
\]
where the horizontal functors are post-composition with $F$ and the
vertical functors are pre-composition with $q$. If $q$ is weakly $\mathcal{C}$ and $\mathcal{D}$ ambidextrous, then this square is canonically normed.
\end{defn}

\begin{prop}
\label{prop:BC_Functor}Let $F\colon\mathcal{C}\to\mathcal{D}$ be
a functor of $\infty$-categories and $q\colon A\to B$ a map of spaces.
If $\mathcal{C}$ and $\mathcal{D}$ admit, and $F$ preserves, all
$q$-colimits (resp. $q$-limits), then the $\left(F,q\right)$-square
satisfies the $\bc_{!}$ (resp. $\bc_{*}$) condition.
\end{prop}

\begin{proof}
This follows from the point-wise formulas for the left and right Kan
extensions.
\end{proof}
The following is the main result of this section.
\begin{thm}
\label{thm:Ambi_Functors}Let $F\colon\mathcal{C}\to\mathcal{D}$
be a functor of $\infty$-categories which preserves $\left(m-1\right)$-finite
colimits. Let $q\colon A\to B$ be an $m$-finite map of spaces. If
$q$ is (weakly) $\mathcal{C}$-ambidextrous and (weakly) $\mathcal{D}$-ambidextrous,
then the $\left(F,q\right)$-square is (weakly) ambidextrous. 
\end{thm}

\begin{proof}
The statement about ambidexterity follows immediately from the ambidexterity of $q$ and the statement about weak ambidexterity.
We shall prove the latter by induction on $m$. For $m=-2$, both vertical
maps in the $\left(F,q\right)$-square are isomorphisms, and so the
claim follows from \propref{Canonical_Norms_Compatibility}(1). We
therefore assume $m\ge-1$. Consider the diagram
\[
\qquad \quad
\vcenter{
\xymatrix@C=3pc{A\ar[rd]^{\delta}\ar@/^{1pc}/@{=}[rrd]\ar@/_{1pc}/@{=}[rdd]\\
 & A\times_{B}A\ar[d]^{\pi_{2}}\ar[r]^{\pi_{1}} & A\ar[d]^{q}\\
 & A\ar[r]^{q} & B.
}}
\qquad\left(\heartsuit\right)
\]
The square in the diagram induces a $\bc_{!}$ map $\beta_!\colon\left(\pi_{2}\right)_{!}\pi_{1}^{*}\to q^{*}q_{!}$,
which is an isomorphism by \lemref{Base_Change_BC}. By definition,
$\nu_{q}^{\mathcal{C}}$ is the composition of maps 
\[
q^{*}q_{!}\xrightarrow{\beta^{-1}}\left(\pi_{2}\right)_{!}\pi_{1}^{*}\xrightarrow{\mu_{\delta}}\left(\pi_{2}\right)_{!}\delta_{!}\delta^{*}\pi_{1}^{*}\simeq\Id.
\]

By \lemref{Triangle_Unit_Counit_Norm_Diagram}(1), it suffices to
show that the wrong way counit diagram of $q$ commutes. This will
follow from the commutativity of the (solid) diagram:
\[
\xymatrix@C=3pc{q^{*}q_{!}\red{F_{*}}{}\ar[dd]^{\beta_{!}}\ar[r]^-{\beta^{-1}} & \left(\pi_{2}\right)_{!}\left(\pi_{1}\right)^{*}\red{F_{*}}{}\ar@{-->}[dd]^{\wr}\ar[r]^-{\mu_{\delta}^{\mathcal{D}}} & \left(\pi_{2}\right)_{!}\delta_{!}\delta^{*}\left(\pi_{1}\right)^{*}\red{F_{*}}{}\ar@{-->}[d]^{\wr}\ar@/^{1pc}/[rrdd]^{\sim}\\
 &  & \left(\pi_{2}\right)_{!}\delta_{!}\delta^{*}\red{F_{*}}{}\left(\pi_{1}\right)^{*}\ar@{-->}[d]^{\wr}\\
q^{*}\red{F_{*}}{}q_{!}\ar[dd]^{\wr} & \left(\pi_{2}\right)_{!}\red{F_{*}}{}\left(\pi_{1}\right)^{*}\ar@{-->}[dd]\ar@{-->}[ru]^{\mu_{\delta}^{\mathcal{D}}}\ar@{-->}[rd]_{\mu_{\delta}^{\mathcal{C}}} & \left(\pi_{2}\right)_{!}\delta_{!}\red{F_{*}}{}\delta^{*}\left(\pi_{1}\right)^{*}\ar@{-->}[d]\ar@{-->}[rr]^{\sim} &  & \red{F_{*}}.\\
 &  & \left(\pi_{2}\right)_{!}\red{F_{*}}{}\delta_{!}\delta^{*}\left(\pi_{1}\right)^{*}\ar@{-->}[d]\\
\red{F_{*}}{}q^{*}q_{!}\ar[r]^-{\beta^{-1}} & \red{F_{*}}{}\left(\pi_{2}\right)_{!}\left(\pi_{1}\right)^{*}\ar[r]^-{\mu_{\delta}^{\mathcal{C}}} & \red{F_{*}}{}\left(\pi_{2}\right)_{!}\delta_{!}\delta^{*}\left(\pi_{1}\right)^{*}\ar@/_{1pc}/[rruu]_{\sim}
}
\]
The two trapezoids and the upper triangle commute for formal reasons.
The bottom triangle commutes by \lemref{Vertical_Pasting_Formula}(1)
and the fact that $\pi_{2}\circ\delta=\Id$. For the rectangle on
the left, it is enough to prove the commutativity of the associated
rectangular diagram with $\beta$ instead of $\beta^{-1}$ in both
horizontal lines, which we denote by $\left(*\right)$. We can now
consider the commutative cubical diagram
\[
\vcenter{
\xymatrix@R=1pc@C=1pc{\fun\left(B,\mathcal{C}\right)\ar[dd]\sb(0.3){q^{*}}
\ar[rrr]\sp(0.6){\red{F_{*}}{}}\ar[rd]^{q^{*}}\ar@{..>}[rrrrd] &  &  & \fun\left(B,\mathcal{D}\right)\ar@{->}'[d][dd]\sp(0.3){q^{*}}\ar[rd]^{q^{*}}\\
 & \fun\left(A,\mathcal{C}\right)\ar[dd]\sb(0.3){\pi_{2}^{*}}\ar[rrr]\sp(0.4){\red{F_{*}}{}} &  &  & \fun\left(A,\mathcal{D}\right)\ar[dd]\sp(0.65){\pi_{2}^{*}}\\
\fun\left(A,\mathcal{C}\right)\ar@{..>}[rrrrd]\ar@{->}'[r][rrr]\sp(0.4){\red{F_{*}}{}\qquad}\ar[rd]^{\pi_{1}^{*}} &  &  & \fun\left(A,\mathcal{D}\right)\ar[rd]^{\pi_{1}^{*}}\\
 & \fun\left(A\times_{B}A,\mathcal{C}\right)\ar[rrr]\sp(0.4){\red{F_{*}}{}} &  &  & \fun\left(A\times_{B}A,\mathcal{D}\right).
}}
\qquad\left(\spadesuit\right)
\]

Applying \lemref{Horizontal_Pasting_Formula}(1) once to the back
and then right face of $\left(\spadesuit\right)$ and once to the
left and then front face of $\left(\spadesuit\right)$, we get two
presentations of the $\bc_{!}$ map of the diagram
\[
\xymatrix@C=3pc{
\fun(B,\mathcal{C})\ar[d]_{q^*}\ar@{..>}[r] & \fun(A,\mathcal{D})\ar[d]^{\pi_{2}^{*}} \\
\fun(A,\mathcal{C})\ar@{..>}[r] & \fun(A\times_{B}A,\mathcal{D}).
}
\]
These two presentations correspond precisely to the two paths
in $\left(*\right)$.

It is left to check the commutativity of the triangle in the middle,
which is a whiskering of the diagram
\[
\qquad
\vcenter{
\xymatrix{ & \delta_{!}\delta^{*}\red{F_{*}}{}\ar[d]^{\wr}\\
\red{F_{*}}{}\ar[ru]^{\mu_{\delta}^{\mathcal{D}}}\ar[rd]_{\mu_{\delta}^{\mathcal{C}}} & \delta_{!}\red{F_{*}}{}\delta^{*}\ar[d]^{\beta_{!}}\\
 & \red{F_{*}}{}\delta_{!}\delta^{*}.
}}
\qquad\left(\cocone\right)
\]

The map $\delta$ is an $\left(m-1\right)$-finite map that is both
$\mathcal{C}$-ambidextrous and $\mathcal{D}$-ambidextrous. By assumption,
$F$ preserves $\left(m-1\right)$-finite colimits and so, by the
inductive hypothesis, the norm diagram of the $\left(F,\delta\right)$-square
commutes. Thus, $\cocone$ commutes by \lemref{Triangle_Unit_Counit_Norm_Diagram}(2).
\end{proof}
As a corollary, we get a higher analogue of a known fact about $0$-semiadditive
categories.
\begin{cor}
\label{cor:Semi_Add_Functors}Let $F\colon\mathcal{C}\to\mathcal{D}$
be a functor of $m$-semiadditive $\infty$-categories. The functor
$F$ preserves $m$-finite colimits if and only if it preserves $m$-finite
limits.
\end{cor}

\begin{proof}
We proceed by induction on $m$. For $m=-2$, there is nothing to
prove. For $m\ge-1$, assume by induction the claim holds for $m-1$.
Since $\mathcal{C}$ and $\mathcal{D}$ are in particular $\left(m-1\right)$-semiadditive
and $F$ preserves either $\left(m-1\right)$-colimits or $\left(m-1\right)$-limits,
we deduce that $F$ preserves both. For every $m$-finite $A$, consider
the map $q\colon A\to\pt$. Since $\mathcal{C}$ and $\mathcal{D}$
are in particular $\left(m-1\right)$-semiadditive and $F$ preserves
$\left(m-1\right)$-colimits, by \thmref{Ambi_Functors}, the $\left(F,q\right)$-square
is weakly ambidextrous. Since $\mathcal{C}$ and $\mathcal{D}$ are
$m$-semiadditive, the $\left(F,q\right)$-square is in fact ambidextrous.
It follows that the $\left(F,q\right)$-square satisfies the $\bc_{!}$
condition if and only if it satisfies the $\bc_{*}$ condition. Namely,
$F$ preserves $A$-shaped colimits if and only if it preserves $A$-shaped
limits.
\end{proof}

\begin{cor}
\label{cor:product_semiadditive}
Let $\{\mathcal{C}_i\}_{i\in I}$ be a collection of $m$-semiadditive $\infty$-categories. The $\infty$-category $\mathcal{C}\coloneqq \prod_{i\in I}\mathcal{C}_i$ is $m$-semiadditive.   
\end{cor}

\begin{proof}
We proceed by induction on $m$. 
For $m=-2$ there is nothing to prove, and so we may assume that $m\ge -1$. Let $q\colon A\to B$ be an $m$-finite map of spaces. 
By induction, $\mathcal{C}$ is $(m-1)$-semiadditive, and hence $q$ is weakly $\mathcal{C}$-ambidextrous. 
In particular, the map $\nm_q^{\mathcal{C}}$ is defined and it is left to show that it is an isomorphism.
For every $i\in I$, the map $q$ is $\mathcal{C}_i$-ambidextrous and the projection $\pi_i \colon \mathcal{C}\to \mathcal{C}_i$ preserves colimits. Thus, by \thmref{Ambi_Functors}, the $(\pi_i,q)$-square is weakly ambidextrous. Additionally, as $\pi_i$ commutes with limits and colimits, the $(\pi_i,q)$-square satisfies both the $\bc_!$ and $\bc_*$ conditions. The $\mathcal{C}_i$-ambidexterity of $q$ implies now that the natural transformation \[
	\pi_i \nm_q^{\mathcal{C}}\colon \pi_i q_! \to \pi_i q_*
\] 
	is a natural isomorphism. Finally, since the collection $\{\pi_i\}_{i\in I}$ is jointly conservative, we deduce that $\nm_q^{\mathcal{C}}$ is an isomorphism.     
\end{proof}

\begin{defn}
Let $\mathcal{C}$ and $\mathcal{D}$ be $m$-semiadditive $\infty$-categories.
A functor $F\colon\mathcal{C}\to\mathcal{D}$ is called \emph{$m$-semiadditive},
if\emph{ }it preserves $m$-finite (co)limits.
\end{defn}

The fundamental property of $m$-semiadditive functors, which justifies
their name, is
\begin{cor}
\label{cor:Integral_Functor} Let $F\colon\mathcal{C}\to\mathcal{D}$
be an $m$-semiadditive functor and let $q\colon A\to B$ be an $m$-finite map
of spaces.
For all $X,Y\in\fun\left(B,\mathcal{C}\right)$ and $f\colon q^{*}X\to q^{*}Y$,
we have 
\[
F\left(\int\limits _{q}f\right)=\int\limits _{q}F\left(f\right)\quad\in\quad\hom_{h\fun\left(B,\mathcal{D}\right)}\left(FX,FY\right).
\]
In particular, for all $X\in\fun\left(B,\mathcal{C}\right)$ we have
\[
F\left(|q|_{X}\right)=|q|_{F\left(X\right)}\quad\in\quad\hom_{h\fun\left(B,\mathcal{D}\right)}\left(FX,FX\right).
\]
\end{cor}

\begin{proof}
The $\left(F,q\right)$-square is ambidextrous by \thmref{Ambi_Functors}
and satisfies the $\bc$ conditions by \propref{BC_Functor}, and
so the claim follows form \propref{Integral_Ambi}.
\end{proof}
\begin{rem}
In view of \remref{Diagonal_Induction_Integral}, one can reinterpret
\thmref{Ambi_Functors} informally, as saying that 
\[
\int\limits _{\delta}F\left(\Id\right)=F\left(\int\limits _{\delta}\Id\right),
\]
where $\delta\colon A\to A\times_{B}A$ is the diagonal of $q\colon A\to B$.
Since $\delta$ is $\left(m-1\right)$-finite, this in turn follows
inductively from \corref{Integral_Functor}. Turning this argument
into a rigorous proof requires some categorical maneuvers that we
preferred to avoid.
\end{rem}

\subsubsection{Multivariate Functors}

We now discuss a multivariate version of higher semiadditive functors. 
\begin{defn}
\label{def:External_Peoduct}Let $\mathcal{C}_{1},\dots,\mathcal{C}_{k}$
and $\mathcal{D}$ be $\infty$-categories and $F\colon\prod\limits _{i=1}^{k}\mathcal{C}_{i}\to\mathcal{D}$
a functor. Given a collection of diagrams $X_{i}\colon A_{i}\to\mathcal{C}_{i}$
for $i=1,\dots,k$, their external product $X_{1}\boxtimes\dots\boxtimes X_{k}$
is defined to be the composition 
\[
\prod_{i=1}^{k}A_{i}\oto{\prod_{i=1}^{k}X_{i}}\prod_{i=1}^{k}\mathcal{C}_{i}\oto F\mathcal{D}.
\]
This assembles to give a functor 
\[
\boxtimes\colon\prod_{i=1}^{k}\fun\left(A_{i},\mathcal{C}_{i}\right)\to\fun\left(\prod_{i=1}^{k}A_{i},\mathcal{D}\right).
\]
Given a collection of maps of spaces $q_{i}\colon A_{i}\to B_{i}$
for $i=1,\dots,k$, we obtain the associated \emph{external product
square}:
\[
\qquad \qquad
\vcenter{
\xymatrix@R=3pc{\prod\limits _{i=1}^{k}\fun\left(B_{i},\mathcal{C}_{i}\right)\ar[d]^{\prod\limits _{i=1}^{k}q_{i}^{*}}\ar[rr]^{\boxtimes} &  & \fun\left(\prod\limits _{i=1}^{k}B_{i},\mathcal{D}\right)\ar[d]^{\left(\prod\limits _{i=1}^{k}q_{i}\right)^{*}}\\
\prod\limits _{i=1}^{k}\fun\left(A_{i},\mathcal{C}_{i}\right)\ar[rr]^{\boxtimes} &  & \fun\left(\prod\limits _{i=1}^{k}A_{i},\mathcal{D}\right).
}}
\qquad\left(*\right)
\]
\end{defn}

\begin{prop}
\label{prop:BC_External_Product}Let $\mathcal{C}_{1},\dots,\mathcal{C}_{k}$
and $\mathcal{D}$ be $\infty$-categories and $F\colon\prod\limits _{i=1}^{k}\mathcal{C}_{i}\to\mathcal{D}$
a functor. Additionally, let $q_{i}\colon A_{i}\to B_{i}$ for $i=1,\dots,k$
be a collection of maps of spaces. If $F$ preserves all $q_{i}$-colimits
(resp. $q_{i}$-limits) in the $i$-th coordinate, then the external
product square $\left(*\right)$ satisfies the $\text{BC}_{!}$ (resp.
$\text{BC}_{*}$) condition.
\end{prop}

\begin{proof}
We proceed by a sequence of reductions. First, by induction on $k$ and
horizontal pasting (\corref{Horizontal_Pasting_BC}), we can reduce
to $k=2$. Write $q_{1}\times q_{2}$ as a composition 
\[
A_{1}\times A_{2}\oto{q_{1}\times\Id}B_{1}\times A_{2}\oto{\Id\times q_{2}}B_{1}\times B_{2}.
\]
The diagram 
\[
\xymatrix@R=2pc@C=3pc{\fun\left(B_{1},\mathcal{C}_{1}\right)\times\fun\left(B_{2},\mathcal{C}_{2}\right)\ar[d]^{\Id\times q_{2}^{*}}\ar[r]^-{\boxtimes} & \fun\left(B_{1}\times B_{2},\mathcal{D}\right)\ar[d]^{\left(\Id\times q_{2}\right)^{*}}\\
\fun\left(B_{1},\mathcal{C}_{1}\right)\times\fun\left(A_{2},\mathcal{C}_{2}\right)\ar[d]^{q_{1}^{*}\times\Id}\ar[r]^-{\boxtimes} & \fun\left(B_{1}\times A_{2},\mathcal{D}\right)\ar[d]^{\left(q_{1}\times\Id\right)^{*}}\\
\fun\left(A_{1},\mathcal{C}_{1}\right)\times\fun\left(A_{2},\mathcal{C}_{2}\right)\ar[r]^-{\boxtimes} & \fun\left(A_{1}\times A_{2},\mathcal{D}\right)
}
\]

exhibits $\left(*\right)$ as a vertical pasting of the top and bottom
squares. Hence, by \corref{Vertical_Pasting_BC}, it is enough to
show that each of them satisfies the $\bc_{!}$ (resp. $\bc_{*}$)
condition. We will focus on the bottom square (the argument for the
top square is analogous). Since (co)limits in $A_{2}$-local systems
are computed point-wise, the external product functor
\[
F_{A_{2}}\colon\mathcal{C}_{1}\times\fun\left(A_{2},\mathcal{C}_{2}\right)\to\fun\left(A_{2},\mathcal{D}\right)
\]
preserves in each coordinate the (co)limits which are preserved by
$F$. By replacing the $\infty$-category $\mathcal{C}_{2}$ with
$\fun\left(A_{2},\mathcal{C}_{2}\right)$, the $\infty$-category
$\mathcal{D}$ with $\fun\left(A_{2},\mathcal{D}\right)$ and the
functor $F$ with $F_{A_{2}}$, we may assume without loss of generality
that $A_{2}=\Delta^{0}$. The bottom square becomes 
\[
\xymatrix@R=2pc@C=3pc{\fun\left(B_{1},\mathcal{C}_{1}\right)\times\mathcal{C}_{2}\ar[d]^{q_{1}^{*}\times\Id}\ar[r]^{\quad\boxtimes} & \fun\left(B_{1},\mathcal{D}\right)\ar[d]^{q_{1}^{*}}\\
\fun\left(A_{1},\mathcal{C}_{1}\right)\times\mathcal{C}_{2}\ar[r]^{\quad\boxtimes} & \fun\left(A_{1},\mathcal{D}\right).
}
\]

By the exponential rule (\lemref{Exponential_Rule_BC}), it is enough
to show that the left square in the following diagram satisfies the
$\bc_{!}$ (resp. $\bc_{*}$) condition:

\[
\xymatrix@R=2pc@C=3pc{\fun\left(B_{1},\mathcal{C}_{1}\right)\ar[d]^{q_{1}^{*}}\ar[r]^{\boxtimes\qquad} & \fun\left(\mathcal{C}_{2},\fun\left(B_{1},\mathcal{D}\right)\right)\ar[d]^{\left(q_{1}^{\mathcal{C}}\right)^{*}}\ar[r]^{\sim} & \fun\left(B_{1},\fun\left(\mathcal{C}_{2},\mathcal{D}\right)\right)\ar[d]^{q_{1}^{*}}\\
\fun\left(A_{1},\mathcal{C}_{1}\right)\ar[r]^{\boxtimes\qquad} & \fun\left(\mathcal{C}_{2},\fun\left(A_{1},\mathcal{D}\right)\right)\ar[r]^{\sim} & \fun\left(A_{1},\fun\left(\mathcal{C}_{2},\mathcal{D}\right)\right).
}
\]
Equivalently, it is enough to show that the outer square $\square$
satisfies the $\bc_{!}$ (resp. $\bc_{*}$) condition. Observe that
$\square$ is the $\left(F^{\vee},q_{1}\right)$-square for the functor
\[
F^{\vee}\colon\mathcal{C}_{1}\to\fun\left(\mathcal{C}_{2},\mathcal{D}\right),
\]
which is the mate of $F$. From the assumption on $F$, the functor
$F^{\vee}$ preserves $q_{1}$-colimits (resp. $q_{1}$-limits) and
therefore $\square$ satisfies the $\text{BC}_{!}$ (resp. $\text{BC}_{*}$)
condition by the univariate version of \propref{BC_Functor}.
\end{proof}
\begin{defn}
Let $\mathcal{C}_{1},\dots,\mathcal{C}_{k}$ and $\mathcal{D}$ be
$m$-semiadditive $\infty$-categories. An \emph{$m$-semiadditive
multi-functor} $F\colon\prod\limits _{i=1}^{k}\mathcal{C}_{i}\to\mathcal{D}$
is a functor that preserves $m$-finite colimits in each coordinate
separately.
\end{defn}

In particular, we get
\begin{cor}
\label{cor:Semi_Add_Multi_Functor} Let $\mathcal{C}_{1},\dots,\mathcal{C}_{k}$
and $\mathcal{D}$ be $m$-semiadditive $\infty$-categories. Let
$F\colon\prod\limits _{i=1}^{k}\mathcal{C}_{i}\to\mathcal{D}$ be
an $m$-semiadditive multi-functor. For every collection of $m$-finite
maps $q_{i}\colon A_{i}\to B_{i}$ for $i=1,\dots k$, the external
product square $\left(*\right)$ from \defref{External_Peoduct} satisfies
both $\bc$-conditions.
\end{cor}

\subsection{Symmetric Monoidal Structure}

In this section, we study the interaction of higher semiadditivity
with (symmetric) monoidal structures. 

\subsubsection{Monoidal Local Systems}

Let $\left(\mathcal{C},\otimes,\one\right)$ be a (symmetric) monoidal
$\infty$-category. For every space $A$, the $\infty$-category $\fun\left(A,\mathcal{C}\right)$ 
acquires a point-wise (symmetric) monoidal structure. Moreover, given
a map of spaces $q\colon A\to B$, the functor
\[
q^{*}\colon\fun\left(B,\mathcal{C}\right)\to\fun\left(A,\mathcal{C}\right)
\]
is (symmetric) monoidal in a canonical way (\cite[Example 3.2.4.4]{ha}). 
\begin{prop}
\label{prop:Local_Systems_Tensor_Normed}Let $\left(\mathcal{C},\otimes,\one\right)$
be a monoidal $\infty$-category. Let $q\colon A\to B$ be a weakly
$\mathcal{C}$-ambidextrous map of spaces, such that $\otimes$ distributes
over $q$-colimits. The normed functor 
\[
q^{\can}\colon\fun\left(A,\mathcal{C}\right)\nto\fun\left(B,\mathcal{C}\right)
\]
is $\otimes$-normed in a canonical way (see \defref{Tensor_Normed_Functor}).
\end{prop}

\begin{proof}
Consider the diagram
\[
\xymatrix@C=4pc{
q_!(q^*X \otimes Y)\ar@{-->}@/_2pc/[ddr]\ar[r]^-{u_{!,X} \otimes u_{!,Y}} & 
q_!((q^*\underline{q_!)q^*}X \otimes (q^*q_!)Y)\ar[d]^{\wr}\ar[r]^-{c_{!,X}\otimes \Id} & 
q_!(q^*X \otimes q^*q_!Y)\ar[d]^{\wr} \\
 & \underline{q_!q^*}(\underline{q_!q^*}X \otimes q_!Y)\ar[d]^{c_{!,(X\otimes q_!Y)}}\ar[r]^-{c_{!,X}\otimes \Id} & 
\underline{q_!q^*}(X \otimes q_!Y)\ar[d]^{c_{!,(X\otimes q_!Y)}} \\
 & \underline{q_!q^*}X \otimes q_!Y\ar[r]^-{c_{!,X}\otimes \Id} & 
X \otimes q_!Y.
}
\]
The triangle on the left commutes by definition, where the dashed arrow is induced by the colax monoidality of $q_!$. The rest of the diagram commutes for formal reasons. The composition along the bottom path of the diagram is the second map in \defref{Tensor_Normed_Functor} and we shall show it is an isomorphism (the proof for the first one follows by symmetry).  Since the diagram commutes, it suffices to show that the composition along the top and then right path is an isomorphism. By the zig-zag identities, the latter is homotopic to the composition
\[
\xymatrix@C=4pc{
	q_!(q^*X \otimes Y)\ar[r]^-{\Id \otimes u_{!,Y}} & 
	q_!(q^*X \otimes q^*q_!Y)\iso
	q_!q^*(X \otimes q_!Y)\ar[r]^-{c_{!,(X\otimes q!Y)}} & 
	X \otimes q_!Y.
}
\]
Finally, this composition is by definition the $\bc_!$ map $\beta_!$ for the square
\[
\xymatrix@C=4pc{
\fun(B,\mathcal{C})\ar[d]^{q^*}\ar[r]^-{X \otimes (-)} &
\fun(B,\mathcal{C})\ar[d]^{q^*} \\
\fun(A,\mathcal{C})\ar[r]^-{q^*X \otimes (-)} &
\fun(A,\mathcal{C}).
}
\]
To see that $\beta_!$ is an isomorphism, it is enough to check this after pulling back to every point $b\in B$. This in turn follows from the assumption that $\otimes$ distributes over $q$-colimits.
\end{proof}

This allows us to apply the general results about $\otimes$-normed
functors to the setting of local systems.
\begin{cor}
\label{cor:Semi_Add_Mode}Let $F\colon\mathcal{C}\to\mathcal{D}$
be an $m$-finite colimit preserving monoidal functor between monoidal
categories that admit, and the tensor product distributes over, $m$-finite
colimits.
\begin{enumerate}
\item An $m$-finite map of spaces $q\colon A\to B$, that is $\mathcal{C}$-ambidextrous
and weakly $\mathcal{D}$-ambidextrous, is $\mathcal{D}$-ambidextrous.
\item If $\mathcal{C}$ is $m$-semiadditive, then $\mathcal{D}$ is also
$m$-semiadditive.
\end{enumerate}
\end{cor}

\begin{proof}
By \propref{Local_Systems_Tensor_Normed}, $q^{\can}$ is $\otimes$-normed.
By \thmref{Ambi_Functors}, the $\left(F,q\right)$-square is weakly ambidextrous.
Since $F$ preserves $m$-finite colimits, the $\left(F,q\right)$-square
satisfies the $\bc_{!}$-condition. (1) now follows from \propref{Tensor_Ambi}.
We prove (2) by induction on $m$. For $m=-2$, there is nothing to
prove, and so we assume $m\ge-1$. By the inductive hypothesis, we
may assume $\mathcal{D}$ is $\left(m-1\right)$-semiadditive. In
this case, every $m$-finite map $q\colon A\to B$ is weakly $\mathcal{D}$-ambidextrous
and $\mathcal{C}$-ambidextrous, hence by (1), is $\mathcal{D}$-ambidextrous.
\end{proof}
The following definition is the natural notion of (symmetric) monoidal
structure in the realm of $m$-semiadditive $\infty$-categories.
\begin{defn}
An \emph{$m$-semiadditively} (symmetric) monoidal $\infty$-category,
is an $m$-semiadditive (symmetric) monoidal $\infty$-category $\mathcal{C}$,
such that the tensor product distributes over $m$-finite colimits.
\end{defn}

\begin{lem}
\label{lem:Box_Unit}Let $\left(\mathcal{C},\otimes,\one\right)$
be an $m$-semiadditively monoidal $\infty$-category and $A$ an
$m$-finite space. 
\begin{enumerate}
\item For every $X\in\mathcal{C}$, we have $|A|_{X}\simeq\Id_{X}\otimes|A|_{\one}$.
\item $A$ is $\mathcal{C}$-amenable if and only if $|A|_{\one}$
is an isomorphism.
\end{enumerate}
\end{lem}

\begin{proof}
We start with (1). Given an object $X\in\mathcal{C}$, the functor
$F_{X}\colon\mathcal{C}\to\mathcal{C}$, given by 
\[
F_{X}\left(Y\right)=X\otimes Y,
\]
preserves $m$-finite colimits. Thus, by \corref{Integral_Functor}
we have:
\[
\Id_{X}\otimes|A|_{\one}=F_{X}\left(|A|_{\one}\right)=|A|_{F_{X}\left(\one\right)}=|A|_{X}.
\]

(2) is an immediate corollary of (1).
\end{proof}
\begin{notation}
For an $m$-semiadditively symmetric monoidal $\infty$-category $\left(\mathcal{C},\otimes,\one\right)$
and an $m$-finite space $A$, we abuse notation by identifying $|A|_{\one}$
with $|A|$. If we want to emphasize the $\infty$-category
$\mathcal{C}$, we write $|A|_{\mathcal{C}}$. By \lemref{Box_Unit},
this conflation of terminology is rather harmless.
\end{notation}

We also have the following consequence for dualizability.
\begin{prop}
\label{cor:Ambi_Dualizable_Spaces} Let $\left(\mathcal{C},\otimes,\one\right)$
be a monoidal $\infty$-category. For every $\mathcal{C}$-ambidextrous space $A$ such that $\otimes$ distributes over $A$-colimits, the object
$\one_{A}$ (see Notation \ref{not:Canonical_Norm_Notation}) is dualizable. In particular, if $\left(\mathcal{C},\otimes,\one\right)$
is $m$-semiadditively monoidal $\infty$-category, then $\one_{A}$
is dualizable for every $m$-finite space $A$.
\end{prop}

\begin{proof}
By \propref{Local_Systems_Tensor_Normed}, the map $q\colon A\to\pt$
corresponds to a $\otimes$-normed functor 
\[
q^{\can}\colon\fun\left(A,\mathcal{C}\right)\nto\mathcal{C}
\]
and by definition $\one_{A}=\one_{q}=q_{!}q^{*}\one$. Thus, the claim
follows from \propref{Iso_Normed_Duality}.
\end{proof}

\subsubsection{Symmetric Monoidal Dimension}

We now specialize to the \emph{symmetric} monoidal case. We begin
with recalling the definition of dimension for a dualizable object
of a symmetric monoidal $\infty$-category. As in \cite[Section 5.1]{HopkinsLurie}, a dualizable object $X$
in a symmetric monoidal $\infty$-category $\left(\mathcal{C},\otimes,\one\right)$
has a notion of dimension, which is defined as follows. Let $X^{\vee}$
be the dual of $X$ and let 
\[
\varepsilon\colon X^{\vee}\otimes X\to\one,\quad\eta\colon\one\to X\otimes X^{\vee}
\]
be the evaluation and coevaluation maps respectively. 
\begin{defn}
We denote by 
\[
\dim_{\mathcal{C}}\left(X\right)\in\End_{\mathcal{C}}\left(\one\right)
\]
the composition
\[
\one\oto{\eta}X\otimes X^{\vee}\oto{\sigma}X^{\vee}\otimes X\oto{\varepsilon}\one,
\]
where $\sigma$ is the swap map of the symmetric monoidal structure.
We say that a space $A$ is dualizable in $\mathcal{C}$, if $\one_{A}$
is dualizable in $\mathcal{C}$ and we denote 
\[
\dim_{\mathcal{C}}\left(A\right)=\dim_{\mathcal{C}}\left(\one_{A}\right).
\]
\end{defn}

Dualizability of $m$-finite spaces in $\mathcal{S}_{m}^{m}$ assumes
a particularly simple form. 
\begin{prop}
\label{prop:Dim_Sym_Span}Every $m$-finite space $A$ is self dual
in $\mathcal{S}_{m}^{m}$ and satisfies 
\[
\dim_{\mathcal{S}_{m}^{m}}\left(A\right)=(\pt\from A^{S^{1}}\to\pt)=|A^{S^{1}}|\in\End_{\mathcal{S}_{m}^{m}}\left(\pt\right).
\]
\end{prop}

\begin{proof}
It is straightforward to check that the spans
\[
\varepsilon\colon(A\times A\xleftarrow{\Delta}A\to\pt)
\]
\[
\eta\colon(\pt\from A\oto{\Delta}A\times A),
\]
satisfy the zig-zag identities and therefore $\varepsilon$ is a duality
pairing exhibiting $A$ as self dual. Moreover, since $\varepsilon\circ\sigma$
is homotopic to $\varepsilon$, where $\sigma\colon X\times X\to X\times X$
is the symmetric monoidal swap, we get $\dim\left(A\right)=\varepsilon\circ\eta$
. Computing the relevant pullback explicitly, 
\[
\xymatrix@R=1pc@C=1pc{ &  & \ar[dl]\quad A^{S^{1}}\ar[dr]\\
 & \ar[dl]A\ar[dr] &  & \ar[dl]A\ar[dr]\\
\quad\pt\quad &  & A\times A &  & \quad\pt\quad
}
\]
we obtain the desired result.
\end{proof}
As a symmetric monoidal $\infty$-category, $\mathcal{S}_{m}^{m}$
has also the following universal property.
\begin{thm}
[Harpaz, {{\cite[Corollary 5.8 ]{Harpaz}}}]\label{thm:=00005BHarpaz=00005DSymm_Mon}
Let $\left(\mathcal{C},\otimes,\one\right)$ be an $m$-semiadditively
symmetric monoidal $\infty$-category. There exists a unique $m$-semiadditive
symmetric monoidal functor $\mathcal{S}_{m}^{m}\to\mathcal{C}$ and
its underlying functor is $\one_{\left(-\right)}$.
\end{thm}

From this we immediately get
\begin{cor}
\label{cor:Dim_Sym}Let $\left(\mathcal{C},\otimes,\one\right)$ be
an $m$-semiadditively symmetric monoidal $\infty$-category. Every
$m$-finite space $A$ is dualizable in $\mathcal{C}$ and 
\[
\dim_{\mathcal{C}}\left(A\right)=|A^{S^{1}}|\quad\in\hom_{h\mathcal{C}}\left(\one_{\mathcal{C}},\one_{\mathcal{C}}\right).
\]
In particular, if $A$ is a loop space (e.g. $A=B^{k}C_{p}$), we
have 
\[
\dim_{\mathcal{C}}\left(A\right)=|A||\Omega A|.
\]
\end{cor}

\begin{proof}
By \thmref{=00005BHarpaz=00005DSymm_Mon}, there is a canonical $m$-finite
colimit preserving symmetric monoidal functor $F\colon\mathcal{S}_{m}^{m}\to\mathcal{C}$.
Since $F\left(A\right)=\one_{A}$ and $F$ is symmetric monoidal,
we have 
\[
F\left(\dim_{\mathcal{S}_{m}^{m}}A\right)=\dim_{\mathcal{C}}\left(\one_{A}\right).
\]
Since $F$ also preserves $m$-finite colimits, we have by \corref{Integral_Functor},
that 
\[
F\left(|B|_{\pt}\right)=|B|_{F\left(\pt\right)}=|B|_{\one_{\mathcal{C}}}
\]
 for all $m$-finite $B$. We are therefore reduced to the universal
case $\mathcal{C}=\mathcal{S}_{m}^{m}$, which is given by \propref{Dim_Sym_Span}.
The last claim follows from the fact that if $A$ is a
loop-space, then $A^{S^{1}}\simeq A\times\Omega A$ and \corref{Distributivity}. 
\end{proof}

\subsection{Equivariant Powers}

Let $\mathcal{C}$ be a symmetric monoidal $\infty$-category and
$p$ a prime. As we shall recall below, for every object $X\in\mathcal{C}$, the $p$-th tensor
power $X^{\otimes p}$ carries a natural action of the cyclic group
$C_{p}\ss\Sigma_{p}$. Moreover, given a map $f\colon X\to Y$, we
get a $C_{p}$-equivariant morphism $f^{\otimes p}\colon X^{\otimes p}\to Y^{\otimes p}$.
Namely, there is a functor 
\[
\Theta^{p}\colon\mathcal{C}\to\fun\left(BC_{p},\mathcal{C}\right),
\]
whose composition with $e^{*}\colon\fun\left(BC_{p},\mathcal{C}\right)\to\mathcal{C}$
(where $e\colon\pt\to BC_{p}$) is homotopic to the $p$-th power
functor $\left(-\right)^{\otimes p}\colon\mathcal{C}\to\mathcal{C}$.
In this section, we study the functor $\Theta^{p}$, its naturality
and additivity properties.

\subsubsection{Functoriality \& Integration}

We begin by describing $\Theta^{p}$ formally. It will be useful to work in the greater level of generality of $\mathcal{C}$-valued \emph{local-systems} instead of single objects. 
Given a simplicial set $K$ we define the $C_{p}$-equivariant $p$-power of $K$ to be the simplicial set
$K^p_{hC_p} = (K^p\times EC_p)/C_p.$
For $K=\mathcal{C}$ a quasi-category, one can easily varify that $\mathcal{C}^p_{hC_p}$ is a quasi-category as well. 
Moreover, since the $C_p$ action on $\mathcal{C}^p\times EC_p$ is \emph{free}, the quasi-category $\mathcal{C}^p_{hC_p}$ is a model for the $\infty$-categorical quotient of $\mathcal{C}^p$ by $C_p$\footnote{Compare \cite[Section 6.1.4]{ha}, where the analogous construction of $\Sigma_n$-equivariant powers is discussed.}. In particular, we can consider this construction for every $A\in \mathcal{S}$, which we also denote by 
$A\wr C_{p}=\left(A^{p}\right)_{hC_{p}}.$ 
\begin{lem} \label{lem:Wr_Fiber_Product}
    The functor 
    $\left(-\right)\wr C_{p}\colon \mathcal{S} \to \mathcal{S}$ 
    preserves fiber products.
\end{lem}
\begin{proof}
    The functor $\left(-\right)\wr C_{p}$ can be identified with the composition
    \[
    \mathcal{S} \oto{e_*} \fun(BC_p,\mathcal{S})
    \simeq \mathcal{S}_{/BC_p} \oto{\pi} \mathcal{S},
\]
where $\pt \oto{e} BC_p$ is a choice of a base point. The functor $e_*$ preserves limits as it is a right adjoint, and the canonical projection $\pi$ preserves limits of contractible shape \cite[Proposition 4.4.2.9]{ha} 
\end{proof}

The construction $ \left(-\right)\wr C_{p}$ induces a functor 
\[
\left(-\right)_{hC_{p}}^{p}:\fun\left(A,\mathcal{C}\right)\to\fun(\left(A^{p}\right)_{hC_{p}},\left(\mathcal{C}^{p}\right)_{hC_{p}}).
\]

Using this have the following:
\begin{defn}
\label{def:Theta}Given a symmetric monoidal $\infty$-category $\mathcal{C}$,
we define the functor 
\[
\Theta_A^{p}:\fun\left(A,\mathcal{C}\right)\to\fun\left(A\wr C_{p},\mathcal{C}\right)
\]
to be the composition of $\left(-\right)_{hC_{p}}^{p}$ with 
\[
\left(\mathcal{C}^{p}\right)_{hC_{p}}\to\left(\mathcal{C}^{p}\right)_{h\Sigma_{p}}\oto{\otimes}\mathcal{C}.
\]
We shall suppress the subscript $A$ in $\Theta_A$ when the space $A$ is understood from the context. 
\end{defn}

The $\Theta^{p}$ operation is functorial in the following sense.
\begin{lem}
\label{lem:Theta_Functoriality} Let $F\colon\mathcal{C}\to\mathcal{D}$
be a symmetric monoidal functor between symmetric monoidal $\infty$-categories.
For every space $A$, the diagram
\[
\xymatrix@C=3pc{\fun\left(A,\mathcal{C}\right)\ar[d]_{F_{*}}\ar[r]^-{\Theta^{p}} & \fun\left(A\wr C_{p},\mathcal{C}\right)\ar[d]^{F_{*}}\\
\fun\left(A,\mathcal{D}\right)\ar[r]^-{\Theta^{p}} & \fun\left(A\wr C_{p},\mathcal{D}\right)
}
\]
commutes up to homotopy.
\end{lem}

\begin{proof}
The square in question is the outer square of the following diagram
\[
\xymatrix@C=3pc{\fun\left(A,\mathcal{C}\right)\ar[d]_{F_{*}}\ar[r] & \fun(A\wr C_{p},\mathcal{C}_{hC_{p}}^{p})\ar[d]^{F_{*}}\ar[r]^-{\otimes} & \fun\left(A\wr C_{p},\mathcal{C}\right)\ar[d]^{F_{*}}\\
\fun\left(A,\mathcal{D}\right)\ar[r] & \fun(A\wr C_{p},\mathcal{D}_{hC_{p}}^{p})\ar[r]^-{\otimes} & \fun\left(A\wr C_{p},\mathcal{D}\right).
}
\]
The left square commutes by the functoriality of $\mathcal{C}\mapsto\mathcal{C}_{hC_{p}}^{p}$
and the right, since $F$ is symmetric monoidal.
\end{proof}
\begin{defn}
For a map of spaces $q\colon A\to B$, the naturality of \defref{Theta}
gives a commutative square 
\[
\xymatrix@C=3pc{\fun\left(B,\mathcal{C}\right)\ar[d]^{q^{*}}\ar[r]^-{\Theta_B^{p}} & \fun\left(B\wr C_{p},\mathcal{C}\right)\ar[d]^{\left(q\wr C_{p}\right)^{*}}\\
\fun\left(A,\mathcal{C}\right)\ar[r]^-{\Theta_A^{p}} & \fun\left(A\wr C_{p},\mathcal{C}\right).
}
\]
\end{defn}

We call this the $\Theta^{p}$-square of $q$. If
$q$ is $m$-finite, then so is $q\wr C_{p}$. If additionally $\mathcal{C}$
is $\left(m-1\right)$-semiadditive and admits $m$-finite (co)limits,
the $\Theta^{p}$-square is canonically normed.
\begin{example}
For a space $A$, we have a canonical fiber sequence
\[
A^{p}\to A\wr C_{p}\xrightarrow{\pi}BC_{p}.
\]
The $\Theta^{p}$-square of $q\colon A\to\pt$ is 
\[
\xymatrix@C=3pc{\fun\left(\pt,\mathcal{C}\right)\ar[d]^{q^{*}}\ar[r]^-{\Theta^{p}} & \fun\left(BC_{p},\mathcal{C}\right)\ar[d]^{\pi^{*}}\\
\fun\left(A,\mathcal{C}\right)\ar[r]^-{\Theta^{p}} & \fun\left(A\wr C_{p},\mathcal{C}\right).
}
\]
\end{example}

\begin{lem}
\label{lem:Theta_BC}Let $q\colon A\to B$ be a map of spaces and
let $\left(\mathcal{C},\otimes,\one\right)$ be a symmetric monoidal
$\infty$-category that admits all $q$-(co)limits. If $\otimes$
distributes over all $q$-colimits (resp. $q$-limits), then the $\Theta^{p}$-square
satisfies the $\text{BC}_{!}$ (resp. $\text{BC}_{*}$) condition.
\end{lem}

\begin{proof}
We horizontally paste the $\Theta^{p}$-square for $q$ with the square
induced by the pullback diagram 
\[
\xymatrix@C=3pc{A^{p}\ar[d]_{q^{p}}\ar[r]^-{\pi_{A}} & A\wr C_{p}\ar[d]^{q\wr C_{p}}\\
B^{p}\ar[r]^-{\pi_{B}} & B\wr C_{p},
}
\]
to obtain 
\[
\qquad
\vcenter{
\xymatrix@C=3pc{\fun\left(B,\mathcal{C}\right)\ar[d]^{q^{*}}\ar[r]^-{\Theta_{B}^{p}} & \fun\left(B\wr C_{p},\mathcal{C}\right)\ar[d]^{\left(q\wr C_{p}\right)^{*}}\ar[r]^-{\ \pi_{B}^{*}} & \fun\left(B^{p},\mathcal{C}\right)\ar[d]^{\left(q^{p}\right)^{*}}\\
\fun\left(A,\mathcal{C}\right)\ar[r]^-{\Theta_{A}^{p}} & \fun\left(A\wr C_{p},\mathcal{C}\right)\ar[r]^-{\ \pi_{A}^{*}} & \fun\left(A^{p},\mathcal{C}\right).
}}
\qquad\left(*\right)
\]
The right square $\square_{R}$ satisfies both $\text{BC}$-conditions
by \lemref{Base_Change_BC}. Since $\pi_{B}^{*}$ is conservative
($\pi_{B}$ is surjective on connected components), by \corref{Horizontal_Pasting_BC}(2),
it is enough to show that the outer square $\square$ satisfies the
$\text{BC}_{!}$ (resp. $\text{BC}_{*}$) condition. We can now write
$\square$ as a horizontal pasting of two squares $\square_{L}'$
and $\square_{R}'$ in a different way:
\[
\xymatrix@C=3pc{\fun\left(B,\mathcal{C}\right)\ar[d]^{q^{*}}\ar[r]^-{\Delta} & \fun\left(B,\mathcal{C}\right)^{p}\ar[d]^{\left(q^{*}\right)^{p}}\ar[r]^-{\boxtimes^{p}} & \fun\left(B^{p},\mathcal{C}\right)\ar[d]^{\left(q^{p}\right)^{*}}\\
\fun\left(A,\mathcal{C}\right)\ar[r]^-{\Delta} & \fun\left(A,\mathcal{C}\right)^{p}\ar[r]^-{\boxtimes^{p}} & \fun\left(A^{p},\mathcal{C}\right).
}
\]

The square $\square_{L}'$ satisfies the $\text{BC}$-conditions trivially and
$\square_{R}'$ by \propref{BC_External_Product}. 
\end{proof}
\begin{prop}
\label{prop:Theta_Ambi}Let $\left(\mathcal{C},\otimes,\one\right)$ be an $m$-semiadditively symmetric monoidal $\infty$-category and let $q\colon A\to B$ be an $m$-finite map of spaces. The corresponding $\Theta^{p}$-square is ambidextrous. 
\end{prop}

\begin{proof}
Since $\mathcal{C}$ is $m$-semiadditive, the $\Theta^{p}$-square
for $q$ is iso-normed and hence it suffices to show that it is weakly
ambidextrous. Namely, that the associated norm-diagram commutes. The
proof is very similar to the argument given in \thmref{Ambi_Functors},
and therefore we shall use similar notation and indicate only the
changes that need to be made. We proceed by induction on $m$ using
the diagram of spaces
\[
\qquad
\vcenter{
\xymatrix@C=3pc{A\ar[rd]^{\delta}\ar@/^{1pc}/@{=}[rrd]\ar@/_{1pc}/@{=}[rdd]\\
 & A\times_{B}A\ar[d]^{\pi_{2}}\ar[r]^{\pi_{1}} & A\ar[d]^{q}\\
 & A\ar[r]^{q} & B.
}}
\qquad\left(\heartsuit\right)
\]
Denoting $\tilde{\left(-\right)}=\left(-\right)\wr C_{p}$, we consider
the diagram of functors from $\fun(A,\mathcal{C})$ to $\fun(A\wr C_{p},\mathcal{C})$
(where all unnamed arrows are $\bc$-maps)
\[
\xymatrix@C=3pc{\tilde{q}^{*}\tilde{q}_{!}\red{\Theta_{A}^{p}}{}\ar[dd]^{\beta_{!}}\ar[r]^-{\beta^{-1}} & \left(\tilde{\pi}_{2}\right)_{!}\left(\tilde{\pi}_{1}\right)^{*}\red{\Theta_{A}^{p}}{}\ar@{-->}[dd]^{\wr}\ar[r]^-{\mu_{\tilde{\delta}}} & \left(\tilde{\pi}_{2}\right)_{!}\tilde{\delta}_{!}\tilde{\delta}^{*}\left(\tilde{\pi}_{1}\right)^{*}\red{\Theta_{A}^{p}}{}\ar@{-->}[d]^{\wr}\ar@/^{1pc}/[rrdd]^{\sim}\\
 &  & \left(\tilde{\pi}_{2}\right)_{!}\tilde{\delta}_{!}\tilde{\delta}^{*}\red{\Theta_{A\times_{B}A}^{p}}{}\left(\pi_{1}\right)^{*}\ar@{-->}[d]^{\wr}\\
\tilde{q}^{*}\red{\Theta_{B}^{p}}{}q_{!}\ar[dd]^{\wr} & \left(\tilde{\pi}_{2}\right)_{!}\red{\Theta_{A\times_{B}A}^{p}}{}\left(\pi_{1}\right)^{*}\ar@{-->}[dd]\ar@{-->}[ru]^{\mu_{\tilde{\delta}}}\ar@{-->}[rd]_{\mu_{\delta}} & \left(\tilde{\pi}_{2}\right)_{!}\tilde{\delta}_{!}\red{\Theta_{A}^{p}}{}\delta^{*}\left(\pi_{1}\right)^{*}\ar@{-->}[d]\ar@{-->}[rr]^{\sim} &  & \red{\Theta_{A}^{p}}.\\
 &  & \left(\tilde{\pi}_{2}\right)_{!}\red{\Theta_{A\times_{B}A}^{p}}{}\delta_{!}\delta^{*}\left(\pi_{1}\right)^{*}\ar@{-->}[d]\\
\red{\Theta_{A}^{p}}{}q^{*}q_{!}\ar[r]^-{\beta^{-1}} & \red{\Theta_{A}^{p}}{}\left(\pi_{2}\right)_{!}\left(\pi_{1}\right)^{*}\ar[r]^-{\mu_{\delta}} & \red{\Theta_{A}^{p}}{}\left(\pi_{2}\right)_{!}\delta_{!}\delta^{*}\left(\pi_{1}\right)^{*}\ar@/_{1pc}/[rruu]_{\sim}
}
\]
By \lemref{Triangle_Unit_Counit_Norm_Diagram}(1), it suffices to
show that the above (solid) diagram commutes. As in the proof of \thmref{Ambi_Functors},
all the parts except for the rectangle on the left and the triangle in the middle, commute for formal reasons. The functor 
$\left(-\right)\wr C_{p}\colon \mathcal{S}\to \mathcal{S}$
preserves fiber products (\lemref{Wr_Fiber_Product}) and therefore $\tilde{\delta}$ can be identified with the diagonal of $\tilde{q}$. 
By \lemref{Theta_BC}, the $\bc_{!}$
map in the middle triangle is an isomorphism. Thus, the middle triangle
commutes by the inductive hypothesis and \lemref{Triangle_Unit_Counit_Norm_Diagram}(2).
As for the rectangle, we apply a similar argument to the one in \thmref{Ambi_Functors},
using again that the functor $\left(-\right)\wr C_{p}$ preserves
fiber products, and the commutative cubical diagram
\[
\qquad
\vcenter{
\xymatrix@R=1pc@C=1pc{\fun\left(B,\mathcal{C}\right)\ar[dd]\sb(0.3){q^{*}}\ar[rrr]\sp(0.6){\red{\Theta_{B}^{p}}{}}\ar[rd]^{q^{*}}\ar@{..>}[rrrrd] &  &  & \fun\left(B\wr C_{p},\mathcal{C}\right)\ar@{->}'[d][dd]\sp(0.3){\tilde{q}^{*}}\ar[rd]^{\tilde{q}^{*}}\\
 & \fun\left(A,\mathcal{C}\right)\ar[dd]\sb(0.3){\pi_{2}^{*}}\ar[rrr]\sp(0.4){\red{\Theta_{A}^{p}}{}} &  &  & \fun\left(A\wr C_{p},\mathcal{C}\right)\ar[dd]\sp(0.65){\tilde{\pi}_{2}^{*}}\\
\fun\left(A,\mathcal{C}\right)\ar@{->}'[r][rrr]\sp(0.4){\red{\Theta_{A}^{p}}{}\qquad}\ar[rd]^{\pi_{1}^{*}}\ar@{..>}[rrrrd] &  &  & \fun\left(A\wr C_{p},\mathcal{C}\right)\ar[rd]^{\tilde{\pi}_{1}^{*}}\\
 & \fun\left(A\times_{B}A,\mathcal{C}\right)\ar[rrr]\sp(0.4){\red{\Theta_{A\times_{B}A}^{p}}{}} &  &  & \fun\left((A\times_{B}A)\wr C_{p},\mathcal{C}\right).}
}
\qquad\left(\spadesuit\right)
\]
\end{proof}
\begin{thm}
\label{thm:Theta_Integral} Let $\mathcal{C}$ be an $m$-semiadditively
symmetric monoidal $\infty$-category and $q\colon A\to B$ an $m$-finite
map of spaces. For every $X,Y\in\fun\left(B,\mathcal{C}\right)$ and
$f\colon q^{*}X\to q^{*}Y$, we have
\[
\Theta_{B}^{p}\left(\int\limits _{q}f\right)=\int\limits _{q\wr C_{p}}\Theta_{A}^{p}\left(f\right)\quad\in\hom_{h\fun\left(B\wr C_{p},\mathcal{C}\right)}\left(\Theta^{p}X,\Theta^{p}Y\right).
\]
\end{thm}

\begin{proof}
By \lemref{Theta_BC}, the $\Theta^{p}$-square satisfies the $\bc$
conditions, and by \propref{Theta_Ambi}, it is ambidextrous. Thus,
the claim follows from \propref{Integral_Ambi}.
\end{proof}

\subsubsection{Additivity of Theta}

We now investigate the interaction of $\Theta^{p}$ with addition
of morphisms. Let $\mathcal{C}$ be a $0$-semiadditively symmetric
monoidal $\infty$-category. Given two objects $X,Y\in\mathcal{C}$
and two maps $f,g\colon X\to Y$, we can express $f+g$ as an integral
of the pair $\left(f,g\right)$ over $q\colon\pt\sqcup\pt\to\pt$
(see \exaref{Integral_Sum}). Applying \thmref{Theta_Integral} to
this special case and analyzing the result, we will derive a formula
of the form
\[
\Theta^{p}\left(f+g\right)=\Theta^{p}\left(f\right)+\Theta^{p}\left(g\right)+\text{``induced terms''}.
\]

The $\Theta^{p}$-square for $q\colon\pt\sqcup\pt\to\pt$ is
\[
\qquad
\vcenter{
\xymatrix{\fun\left(\pt,\mathcal{C}\right)\ar[d]^{q^{*}}\ar[rr]^{\Theta^{p}_{\pt}\qquad} &  & \fun\left(BC_{p},\mathcal{C}\right)\ar[d]^{\left(q\wr C_{p}\right)^{*}}\\
\fun\left(\pt\sqcup\pt,\mathcal{C}\right)\ar[rr]^{\Theta_{\pt\sqcup\pt}^{p}\quad} &  & \fun\left(\left(\pt\sqcup\pt\right)\wr C_{p},\mathcal{C}\right).
}}
\qquad\left(*\right)
\]
Our first goal is to make this diagram more explicit. First, we can
identify $q^{*}$ with the diagonal 
\[
\Delta\colon\mathcal{C}\to\mathcal{C}\times\mathcal{C}.
\]
Next, let $S$ be the set 
\[
S=\left\{ w\in\left\{ x,y\right\} ^{p}\mid w\neq x^{p},y^{p}\right\} ,
\]
with $x,y$ formal variables and let $\overline{S}$ is the set of
orbits of $S$ under the action of $C_{p}$ by cyclic shift. We have
a homotopy equivalence of spaces
\[
\left(\pt\sqcup\pt\right)\wr C_{p}\simeq BC_{p}\sqcup BC_{p}\sqcup\overline{S},
\]
and therefore an equivalence of $\infty$-categories
\[
\fun\left(\left(\pt\sqcup\pt\right)\wr C_{p},\mathcal{C}\right)\simeq\mathcal{C}^{BC_{p}}\times\mathcal{C}^{BC_{p}}\times\prod_{\overline{w}\in\overline{S}}\mathcal{C}.
\]

Choosing a base point map $e\colon\pt\to BC_{p}$, we see that up
to homotopy, we have
\[
q\wr C_{p}= (\Id,\Id, e,\dots, e).
\]

Similarly, the bottom arrow of $\left(*\right)$ can be identified
with a functor
\[
\Phi\colon\mathcal{C}\times\mathcal{C}\to\mathcal{C}^{BC_{p}}\times\mathcal{C}^{BC_{p}}\times\prod_{\overline{w}\in\overline{S}}\mathcal{C},
\]
which we now describe. For each $\overline{w}\in\overline{S}$, let
\[
e_{\overline{w}}\colon\pt\to\left(\pt\sqcup\pt\right)\wr C_{p}
\]
be the map choosing the point $\overline{w}\in\overline{S}$ and let
$e_{w}\colon\pt\to\overline{w}$ be the map choosing the point $w\in\overline{w}$.
Given an element $w\in\left\{ x,y\right\} ^{p}$ we define a functor
$w\left(-,-\right):\mathcal{C}\times\mathcal{C}\to\mathcal{C}$ as
the composition
\[
\xymatrix{\fun\left(\pt\sqcup\pt,\mathcal{C}\right)\ar[r]^-{\Delta} & \fun\left(\pt\sqcup\pt,\mathcal{C}\right)^{p}\ar[r]^-{\boxtimes} & \fun\left(\left(\pt\sqcup\pt\right)^{p},\mathcal{C}\right)\ar[r]^-{w^{*}} & \fun\left(\pt,\mathcal{C}\right).}
\]

Informally, for objects $X,Y\in\mathcal{C}$, we have 
\[
w\left(X,Y\right)=Z_{1}\otimes Z_{2}\otimes\cdots\otimes Z_{p},\qquad Z_{i}=\begin{cases}
X & \text{if}\quad w_{i}=x\\
Y & \text{if}\quad w_{i}=y
\end{cases}.
\]

\begin{lem}
\label{lem:Phi_Theta}
There is a natural isomorphism of functors
\[
\Phi\simeq\left(\Theta^{p}\circ p_{1},\Theta^{p}\circ p_{2},\left\{ w\left(-,-\right)\right\} _{\overline{w}\in S}\right),
\]
where $p_{i}\colon\mathcal{C}\times\mathcal{C}\to\mathcal{C}$ is
the projection to the $i$-th component (it does not matter which
representative $w$ we take for each $\overline{w}\in S$).
\end{lem}

\begin{proof}
The claim about the first two components follows from the commutativity
of the $\Theta^{p}$-square applied to the two inclusion maps $\pt\into\pt\sqcup\pt$.
The pullback square 
\[
\xymatrix{\overline{w}\ar[d]\ar[rr] &  & \left(\pt\sqcup\pt\right)^{p}\ar[d]^{\sigma}\\
\pt\ar[rr]^{e_{\overline{w}}\qquad} &  & \left(\pt\sqcup\pt\right)\wr C_{p}
}
\]
induces the commutative square in the following diagram
\[
\xymatrix{\fun\left(\pt\sqcup\pt,\mathcal{C}\right)\ar[rr]^{\Theta_{\pt\sqcup\pt}^{p}\quad} &  & \fun\left(\left(\pt\sqcup\pt\right)\wr C_{p},\mathcal{C}\right)\ar[d]^{e_{\overline{w}}^{*}}\ar[r]^{\sigma^{*}} & \fun\left(\left(\pt\sqcup\pt\right)^{p},\mathcal{C}\right)\ar[d]\ar@{-->}[rd]^{w^{*}}\\
 &  & \fun\left(\pt,\mathcal{C}\right)\ar[r]^{\Delta} & \fun\left(\overline{w},\mathcal{C}\right)\ar[r]^{e_{w}^{*}} & \fun\left(\pt,\mathcal{C}\right).
}
\]
Observe that the composition of the leftmost horizontal functor and
the left vertical functor is the $\overline{w}$ component of $\Phi$.
Since the composition of the two bottom horizontal functors is the
identity, it suffices to show that the resulting functor
\[
\fun\left(\pt\sqcup\pt,\mathcal{C}\right)\to\fun\left(\pt,\mathcal{C}\right),
\]
obtained from the composition along the entire bottom path of the
diagram, is naturally isomorphic to $w\left(-,-\right)$. Since the
diagram commutes, this is isomorphic to the composition along the
top path of the diagram, which is $w\left(-,-\right)$ by definition.
\end{proof}
Summing up, we have identified the $\Theta^{p}$-square $\left(*\right)$
with the following square
\[
\qquad
\vcenter{
\xymatrix{\mathcal{C}\ar[d]^{\Delta}\ar[rrrrr]^{\Theta^{p}} &  &  &  &  & \mathcal{C}^{BC_{p}}\ar[d]^{\left(\Id,\Id,e,\dots,e\right)^{*}}\\
\mathcal{C}\times\mathcal{C}\ar[rrrrr]^{\left(\Theta^{p}\circ p_{1},\Theta^{p}\circ p_{2},\left\{ w\left(-,-\right)\right\} _{\overline{w}\in\overline{S}}\right)\qquad\qquad} &  &  &  &  & \mathcal{C}^{BC_{p}}\times\mathcal{C}^{BC_{p}}\times\prod\limits _{\overline{w}\in\overline{S}}\mathcal{C}.
}}
\qquad\left(**\right)
\]
Using this we can compute the effect of $\Theta^{p}$ on the sum of
two maps.
\begin{prop}
\label{prop:Theta_Additiive}Let $\mathcal{C}$ be a $0$-semiadditively
symmetric monoidal $\infty$-category, Given $X,Y\in\mathcal{C}$
and a pair of maps $f,g\colon X\to Y,$ we have
\[
\Theta^{p}\left(f+g\right)=\Theta^{p}\left(f\right)+\Theta^{p}\left(g\right)+\sum_{\overline{w}\in\overline{S}}\left(\int\limits _{e}w\left(f,g\right)\right).
\]
\end{prop}

\begin{proof}
The pair $\left(f,g\right)$ can be considered as a map $\left(f,g\right)\colon q^{*}X\to q^{*}Y$.
By \thmref{Theta_Integral}, \lemref{Phi_Theta} and the additivity of the integral (\propref{Integral_Additivity})
we have 
\[
\Theta^{p}\left(f+g\right)=\Theta^{p}\left(\int\limits _{q}\left(f,g\right)\right)=\int\limits _{\left(\Id,\Id,e,\dots,e\right)}\left(\Theta^{p}\left(f\right),\Theta^{p}\left(g\right),\left\{ w\left(f,g\right)\right\} _{\overline{w}\in\overline{S}}\right)
\]
\[
=\Theta^{p}\left(f\right)+\Theta^{p}\left(g\right)+\sum_{\overline{w}\in\overline{S}}\left(\int\limits _{e}w\left(f,g\right)\right).
\]

\end{proof}

\section{Higher Semiadditivity and Additive Derivations}

Let $\mathcal{C}$ be a \emph{stable} symmetric monoidal $\infty$-category such that the tensor product distributes over finite coproducts. For every pair of objects $X,Y\in \mathcal{C}$, the set 
\[
\hom_{h\mathcal{C}}\left(X,Y\right)=\pi_{0}\map_{\mathcal{C}}\left(X,Y\right)
\]
has a canonical structure of an abelian group. Furthermore, if 
\[
X\in\cocalg\left(\mathcal{C}\right),\quad Y\in\calg\left(\mathcal{C}\right),
\]
then the set $\hom_{h\mathcal{C}}\left(X,Y\right)$ acquires a commutative ring structure in the following way. Given $f,g\colon X\to Y$, we define their product as the composition
\[
    X \oto{\mathrm{co-mult}} X\otimes X \oto{f\otimes g} Y\otimes Y \oto{\mathrm{mult}} Y.
\]
Fixing a prime $p$ and assuming further that $\mathcal{C}$
is $1$-semiadditively symmetric monoidal, we will construct in this section an operation (which depends on $p$) 
\[
\delta\colon\hom_{h\mathcal{C}}\left(X,Y\right)\to\hom_{h\mathcal{C}}\left(X,Y\right),
\]
and show that it is an ``additive $p$-derivation''. We begin with
a general discussion of the algebraic notion of an additive $p$-derivation.
We proceed to construct an auxiliary operation $\alpha$ (which does
not require stability) and study its properties. We then specialize
to the stable case, construct the operation $\delta$ above, and study
its behavior on elements of the form $|A|$. Finally,
we shall use the properties of the operation $\delta$ to provide
a general criterion for deducing $\infty$-semiadditivity
of a presentably symmetric monoidal, $1$-semiadditive, stable, $p$-local
$\infty$-category.

\subsection{Additive \(p\)-derivations}

This section is devoted to the algebraic notion of an additive $p$-derivation.
We recall the definition and establish some of its basic properties.

\subsubsection{Definition \& Properties}

The following is a variant on the notion of a $p$-derivation (e.g.
see \cite[Definition 2.1]{buium2005arithmetic}), in which we do not
require the multiplicative property.
\begin{defn}
\label{def:Delta}Let $R$ be a commutative ring. An \emph{additive
$p$-derivation} on $R$, is a function of sets 
\[
\delta\colon R\to R,
\]
that satisfies:
\begin{enumerate}
\item (additivity) $\delta\left(x+y\right)=\delta\left(x\right)+\delta\left(y\right)+\frac{x^{p}+y^{p}-\left(x+y\right)^{p}}{p}$
for all $x,y\in R$.
\item (normalization) $\delta\left(0\right)=\delta\left(1\right)=0.$
\end{enumerate}
The pair $\left(R,\delta\right)$ is called a \emph{semi-$\delta$-ring}.
A \emph{semi-$\delta$-ring homomorphism} from $\left(R,\delta\right)$
to $\left(R',\delta'\right)$, is a ring homomorphism $f\colon R\to R'$,
that satisfies $f\circ\delta=\delta'\circ f$.
\end{defn}

\begin{rem}
\label{rem:Integer_Coeff_Polynom} The expression 
\[
\frac{x^{p}+y^{p}-\left(x+y\right)^{p}}{p}
\]
is actually a polynomial with integer coefficients in the variables
$x$ and $y$ and does not involve division by $p$. In particular,
this is well defined for all $x,y\in R$, even when $R$ has $p$-torsion.
\end{rem}

\begin{rem}
In fact, the condition $\delta\left(0\right)=0$ is superfluous, as
it follows from the additivity property, and we include it in the
definition only for emphasis. 
\end{rem}

The following follows immediately from the definitions. 
\begin{lem}
\label{lem:Delta_Frob}Let $\delta\colon R\to R$ be an additive $p$-derivation
on a commutative ring $R$. The function $\psi\colon R\to R$ given
by 
\[
\psi\left(x\right)=x^{p}+p\delta\left(x\right)
\]
is an additive lift of Frobenius, i.e. it is a homomorphism of abelian groups and agrees with the Frobenius modulo $p$.
\end{lem}

\begin{example}
\label{exa:Delta_Canonical} The following are some examples of additive
$p$-derivations.
\end{example}

\begin{enumerate}
\item For $R$ a subring of $\bb Q$, the \emph{Fermat quotient} 
\[
\tilde{\delta}\left(x\right)=\frac{x-x^{p}}{p}
\]
is an additive $p$-derivation (we shall soon show that it is the
\emph{unique} additive $p$-derivation on any such $R$).
\item The same formula as for the Fermat quotient defines the unique additive $p$-derivation on the ring of $p$-adic integers $\bb Z_{p}$.
\item Fix $m\ge1$. Let $\mathcal{R}_{m}^{\square}$ be the commutative
ring freely generated by formal elements $|A|$, where
$A$ is an $m$-finite space, subject to the relations
\[
|A\sqcup B|=|A|+|B|,\quad|A\times B|=|A||B|.
\]
It is easy to verify that the operation
\[
\delta\left(|A|\right)=|BC_{p}\times A|-|A\wr C_{p}|
\]
is well defined and is an additive $p$-derivation on $\mathcal{R}_{m}^{\square}$.
\end{enumerate}
\begin{defn}
For every $x\in\bb Q$, we denote by $v_{p}\left(x\right)\in\bb Z\cup\left\{ \infty\right\} $
the $p$-adic valuation of $x$.
\end{defn}

The fundamental property of the Fermat quotient is
\begin{lem}
\label{lem:Delta_Valuation}For every $x\in\bb Q$, if $0<v_{p}\left(x\right)<\infty$,
then 
\[
v_{p}\left(\tilde{\delta}\left(x\right)\right)=v_{p}\left(x\right)-1.
\]
\end{lem}

\begin{proof}
Since $v_{p}\left(x\right)>0$, we have 
\[
v_{p}\left(x^{p}\right)=pv_{p}\left(x\right)>v_{p}\left(x\right).
\]
Thus, 
\[
v_{p}\left(\frac{x-x^{p}}{p}\right)=v_{p}\left(x-x^{p}\right)-1=v_{p}\left(x\right)-1.
\]
\end{proof}
\begin{defn}
Let $R$ be a commutative ring. Let $\phi_{0}\colon\bb Z\to R$ be
the unique ring homomorphism and let $S_{R}$ be the set of primes
$p$, such that $\phi_{0}\left(p\right)\in R^{\times}$. We denote
\[
\bb Q_{R}=\bb Z[S_{R}^{-1}]\ss\bb Q
\]
and $\phi\colon\bb Q_{R}\to R$, the unique extension of $\phi_{0}$.
We call an element $x\in R$ \emph{rational} if it is in the image
of $\phi$. By \exaref{Delta_Canonical}(1), $\left(\bb Q_{R},\tilde{\delta}\right)$
is a semi-$\delta$-ring.
\end{defn}

The following elementary lemma will have several useful consequences.
\begin{lem}
\label{lem:Delta_Module}Let $\left(R,\delta\right)$ be a semi-$\delta$-ring
and let $\tilde{\delta}$ denote the Fermat quotient on $\bb Q_{R}$.
For all $t\in\bb Q_{R}$ and $x\in R$, we have
\[
\delta\left(tx\right)=t\delta\left(x\right)+\tilde{\delta}\left(t\right)x^{p}.
\]
\end{lem}

\begin{proof}
Fix $x\in R$ and consider the function $\varphi\colon\bb Q_{R}\to R$
given by
\[
\varphi\left(t\right)=\delta\left(tx\right)-\tilde{\delta}\left(t\right)x^{p}.
\]
Since
\[
\delta\left(tx+sx\right)=\delta\left(tx\right)+\delta\left(sx\right)+\frac{\left(tx\right)^{p}+\left(sx\right)^{p}-\left(tx+sx\right)^{p}}{p}=
\]
\[
\delta\left(tx\right)+\delta\left(sx\right)+\left(\tilde{\delta}\left(t+s\right)-\tilde{\delta}\left(t\right)-\tilde{\delta}\left(s\right)\right)x^{p}=\varphi\left(t\right)+\varphi\left(s\right)+\tilde{\delta}\left(t+s\right)x^{p}.
\]
we get
\[
\varphi\left(t+s\right)=\delta\left(tx+sx\right)-\tilde{\delta}\left(t+s\right)x^{p}=\varphi\left(t\right)+\varphi\left(s\right).
\]
Hence, $\varphi$ is additive and $\varphi\left(1\right)=\delta\left(x\right)$.
Since $\bb Q_{R}$ is a localization of $\bb Z$, $\varphi$ is a
map of $\bb Q_{R}$-modules and we deduce that $\varphi\left(t\right)=t\delta\left(x\right)$
for all $t\in\bb Q_{R}$.
\end{proof}

\subsubsection{$p$-Local Rings}

In the case where $R$ is a \emph{$p$-local} commutative ring, which
is the case we are mainly interested in, the existence of an additive
$p$-derivation on $R$ has several interesting implications.
\begin{prop}
\label{prop:Delta_Torsion_Nilpotent}Let $\left(R,\delta\right)$
be a $p$-local semi-$\delta$-ring. If $x\in R$ is torsion, then
$x$ is nilpotent.
\end{prop}

\begin{proof}
Since $R$ is $p$-local, if $x$ is torsion, then there is $d\in\bb N$,
such that $p^{d}x=0$. By \lemref{Delta_Module}, we have
\[
0=\delta\left(0\right)=\delta\left(p^{d}x\right)=p^{d}\delta\left(x\right)+\tilde{\delta}\left(p^{d}\right)x^{p}.
\]
Multiplying by $x$, we obtain $\tilde{\delta}\left(p^{d}\right)x^{p+1}=0$.
By \lemref{Delta_Valuation}, $v_p\left(\tilde{\delta}\left(p^{d}\right)\right)=d-1$,
and since $R$ is $p$-local, we get $p^{d-1}x^{p+1}=0$. Iterating
this $d$ times we get $x^{\left(p+1\right)^{d}}=0$.
\end{proof}
\begin{prop}
\label{prop:Delta_Injective} Let $\left(R,\delta\right)$ be a non-zero
$p$-local semi-$\delta$-ring. The map $\phi\colon\bb Q_{R}\to R$
is an injective semi-$\delta$-ring homomorphism. In particular $\tilde{\delta}$
is the unique additive $p$-derivation on $\bb Q_{R}$.
\end{prop}

\begin{proof}
Applying \lemref{Delta_Module} to $x=1$, we see that $\phi\circ\tilde{\delta}=\delta\circ\phi$.
If $\phi$ is non-injective, then so is $\phi_{0}\colon\bb Z\to R$
and hence $1\in R$ is torsion. By \propref{Delta_Torsion_Nilpotent},
$1$ is nilpotent and hence $R=0$.
\end{proof}
\begin{rem}
For a non-zero $p$-local semi-$\delta$-ring $\left(R,\delta\right)$,
we abuse notation by identifying $\bb Q_{R}$ with the subset of rational
elements of $R$. There are two options:
\begin{enumerate}
\item If $p\in R^{\times}$, then $\bb Q_{R}=\bb Q\ss R$ and all non-zero rational elements are invertible.
\item If $p\notin R^{\times}$, then $\bb Q_{R}=\bb Z_{\left(p\right)}\ss R$,
and $x\in\bb Q_{R}$ is invertible if and only if $v_{p}\left(x\right)=0$.
\end{enumerate}
\end{rem}

\begin{prop}
\label{prop:Delta_Torsion_Ideal}Let $\left(R,\delta\right)$ be a
$p$-local semi-$\delta$-ring. The ideal $I_{\tor}\ss R$ of torsion
elements is closed under $\delta$.
\end{prop}

\begin{proof}
For $x\in I_{\tor}$, there is $d\in\bb N$, such that $p^{d}x=0$.
By \lemref{Delta_Module}, 
\[
0=\delta\left(p^{d+1}x\right)=p^{d+1}\delta\left(x\right)+\tilde{\delta}\left(p^{d+1}\right)x^{p}.
\]
By \lemref{Delta_Valuation}, $v_p\left(\tilde{\delta}\left(p^{d+1}\right)\right)=d$
and therefore $\tilde{\delta}\left(p^{d+1}\right)x^{p}=0$. We get
$p^{d+1}\delta\left(x\right)=0$ and hence $\delta\left(x\right)\in I_{\tor}$.
\end{proof}
\begin{defn}
For every commutative ring $R$, we define $I_{\tor}\ss R$ to be
the ideal of torsion elements, and $R^{\tf}=R/I_{\tor}$ to be the
torsion free ring obtained from $R$. 
\end{defn}

The following proposition will allow us to ``ignore torsion'' when
dealing with questions of invertibility in $p$-local semi-$\delta$-ring.
First,
\begin{defn}
Given a ring homomorphism $f\colon R\to S$, we say that $f$ \emph{detects
invertibility} if for every $x\in R$, if $f\left(x\right)$ is invertible,
then $x$ is invertible.
\end{defn}

\begin{prop}
\label{prop:Delta_Torsion_Free}Let $\left(R,\delta\right)$ be a
$p$-local semi-$\delta$-ring. There is a unique additive $p$-derivation
$\delta$ on $R^{\tf}$, such that the quotient map $g\colon R\onto R^{\tf}$
is a homomorphism of semi-$\delta$-rings. In addition, $g$ detects
invertibility.
\end{prop}

\begin{proof}
Let $x\in R$ and $y\in I_{\tor}$. We have 
\[
\delta\left(x+y\right)-\delta\left(x\right)=\delta\left(y\right)+\left(\frac{x^{p}+y^{p}-\left(x+y\right)^{p}}{p}\right)\in I_{\tor}
\]
since $\delta\left(y\right)\in I_{\tor}$ by \propref{Delta_Torsion_Ideal}
and the expression in parenthesis is a multiple of $y$. Thus, 
\[
\delta\left(x+I_{\tor}\right):=\delta\left(x\right)+I_{\tor}
\]
is a well defined function on $R^{\tf}$. The operation $\delta$
is an additive $p$-derivation and makes $g$ a homomorphism of semi-$\delta$-rings.
The operation $\delta$ is unique by the surjectivity of $g$. For
the second claim, the kernel of $g$ consists of nilpotent elements
by \propref{Delta_Torsion_Nilpotent} and hence $g$ detects invertibility.
\end{proof}

\subsection{The Alpha Operation}

Let $\mathcal{C}$ be a $0$-semiadditively symmetric monoidal $\infty$-category
and let 
\[
X\in\cocalg\left(\mathcal{C}\right),\qquad Y\in\calg\left(\mathcal{C}\right).
\]
Fix a prime $p$. The set 
\[
\hom_{h\mathcal{C}}\left(X,Y\right)=\pi_{0}\map_{\mathcal{C}}\left(X,Y\right)
\]
has a structure of a \emph{commutative rig} (i.e. like a ring, but
without additive inverses). Assuming further that $\mathcal{C}$ is
$1$-semiadditively symmetric monoidal, we construct an operation
$\alpha$ (which depends on $p$) on $\hom_{h\mathcal{C}}\left(X,Y\right)$
and study its properties and interaction with the rig structure.

Throughout the section we denote 
\[
\pt\oto eBC_{p}\oto r\pt.
\]

\subsubsection{Definition and Naturality}

The $\bb E_{\infty}$-coalgebra and $\bb E_{\infty}$-algebra structures,
on $X$ and $Y$ respectively, provide symmetric comultiplication
and multiplication maps:
\[
\overline{t}_{X}\colon X\to\left(X^{\otimes p}\right)^{hC_{p}}=r_{*}\Theta^{p}\left(X\right)
\]
\[
\overline{m}_{Y}\colon r_{!}\Theta^{p}\left(Y\right)=\left(Y^{\otimes p}\right)_{hC_{p}}\to Y.
\]
These maps have mates
\[
t_{X}\colon r^{*}X\to\Theta^{p}\left(X\right),\qquad m_{Y}\colon\Theta^{p}\left(Y\right)\to r^{*}Y,
\]
such that
\[
e^{*}t_{X}\colon X=e^{*}r^{*}X\to e^{*}\Theta^{p}\left(X\right)=X^{\otimes p}
\]
\[
e^{*}m_{Y}\colon Y^{\otimes p}=e^{*}\Theta^{p}\left(Y\right)\to e^{*}r^{*}Y=Y,
\]
are the ordinary comultiplication and multiplication maps. 
\begin{defn}
\label{def:Alpha}Let $\mathcal{C}$ be a $1$-semiadditively symmetric
monoidal $\infty$-category and let 
\[
X\in\cocalg\left(\mathcal{C}\right),\qquad Y\in\calg\left(\mathcal{C}\right).
\]
\begin{enumerate}
\item Given $g\colon\Theta^{p}\left(X\right)\to\Theta^{p}\left(Y\right)$,
we define $\overline{\alpha}\left(g\right)\colon X\to Y$ to be either
of the compositions in the commutative diagram
\[
\xymatrix@C=3pc{X\ar[r]^{\overline{t}_{X}\quad} & r_{*}\Theta^{p}\left(X\right)\ar[d]^{g}\ar[r]^{\nm_{r}^{-1}} & r_{!}\Theta^{p}\left(X\right)\ar[d]^{g}\\
 & r_{*}\Theta^{p}\left(Y\right)\ar[r]^{\nm_{r}^{-1}} & r_{!}\Theta^{p}\left(Y\right)\ar[r]^{\quad\overline{m}_{Y}} & Y.
}
\]
\item Given $f\colon X\to Y$, we define $\alpha\left(f\right)=\overline{\alpha}\left(\Theta^{p}\left(f\right)\right)$.
\end{enumerate}
\end{defn}

\begin{rem}
\label{rem:H_Infty}In fact, the definition of $\alpha$ uses only
the $H_{\infty}$-algebra structure of $Y$ and the $H_{\infty}$-coalgebra
structure of $X$. Moreover, everything we state and prove in this
section about the properties of $\alpha$ holds when we replace $\bb E_{\infty}$
with $H_{\infty}$. 
\end{rem}

\begin{lem}
\label{lem:Alpha_Bar_Additive}The map $\overline{\alpha}\colon \pi_0\map(\Theta^p X,\Theta^p Y) \to \pi_0\map(X,Y)$
is additive. 
\end{lem}
\begin{proof}
    Since $r_*$ is an additive functor, it induces an additive map 
    \[
        \pi_0\map(\Theta^p X,\Theta^p Y) \to \pi_0\map(r_*\Theta^p X,r_*\Theta^p Y).
    \]
    The operation $\bar{\alpha}$ consists of the application of this followed by pre- and post-composition with fixed maps in a $0$-semiadditive $\infty$-category. 
\end{proof}

The operation $\alpha$ is natural with respect to (co)algebra homomorphisms
in the following sense.
\begin{lem}
\label{lem:Alpha_Naturality} Let $\mathcal{C}$ be a $1$-semiadditively
symmetric monoidal $\infty$-category and let 
\[
X,X'\in\cocalg\left(\mathcal{C}\right),\qquad Y,Y'\in\calg\left(\mathcal{C}\right).
\]
Given maps $g\colon Y\to Y'$ and $h\colon X'\to X$ of commutative
algebras and coalgebras respectively, for every map $f\colon X\to Y$,
we have 
\[
\alpha\left(g\circ f\circ h\right)=g\circ\alpha\left(f\right)\circ h\quad\in\hom_{h\mathcal{C}}\left(X',Y'\right).
\]
\end{lem}

\begin{proof}
Consider the diagram
\[
\xymatrix@C=3pc{X'\ar@{.>}[r]^{\overline{t}_{X'}\ }\ar@{.>}[d]^{h} & r_{*}\Theta^{p}X'\ar[r]^{\nm_{r}^{-1}}\ar[d]^{h}\ar@{.>}[dddr] & r_{!}\Theta^{p}X'\ar[d]^{h}\\
X\ar[r]^{\overline{t}_{X}\ }\ar@{.>}[drrr] & r_{*}\Theta^{p}X\ar[r]^{\nm_{r}^{-1}}\ar[d]^{f} & r_{!}\Theta^{p}X\ar[d]^{f}\\
 & r_{*}\Theta^{p}Y\ar[r]^{\nm_{r}^{-1}}\ar[d]^{g} & r_{!}\Theta^{p}Y\ar[r]^{\ \overline{m}_{Y}}\ar[d]^{g} & Y\ar@{.>}[d]^{g}\\
 & r_{*}\Theta^{p}Y'\ar[r]^{\nm_{r}^{-1}} & r_{!}\Theta^{p}Y'\ar@{.>}[r]^{\ \overline{m}_{Y'}} & Y'.
}
\]
The squares in the middle column commute by the naturality of the
norm map. The homotopy rendering the bottom right square commutative
is provided by the data that makes $g$ into a morphism of commutative
algebras and similarly for the upper left square and $h$. The composition
along one of the dotted paths is $\alpha\left(g\circ f\circ h\right)$,
while composition along the other dotted path is $g\circ\alpha\left(f\right)\circ h$,
which completes the proof.
\end{proof}
The operation $\alpha$ is also functorial in the following sense.
\begin{lem}
\label{lem:Alpha_Functoriality}Let $F\colon\mathcal{C}\to\mathcal{D}$
be a $1$-semiadditive symmetric monoidal functor between two $1$-semiadditively
symmetric monoidal $\infty$-categories, and let $X\in\cocalg\left(\mathcal{C}\right)$
and $Y\in\calg\left(\mathcal{C}\right)$. The induced map of commutative
rings
\[
F\colon\hom_{h\mathcal{C}}\left(X,Y\right)\to\hom_{h\mathcal{D}}\left(FX,FY\right),
\]
commutes with the operation $\alpha$.
\end{lem}

\begin{proof}
Given a map $g\colon X\to Y$, consider the following diagram:
\[
\xymatrix@C=3pc{ & \red{F}\left(r_{*}\Theta^{p}\left(X\right)\right)\ar[d]_{\wr}^{\beta_{*}}\ar[r]^{\nm_{r}^{-1}}  & \red{F}\left(r_{!}\Theta^{p}\left(X\right)\ar[r]^{g}\right) & \red{F}\left(r_{!}\Theta^{p}\left(Y\right)\ar[rd]^{\quad\overline{m}_{Y}}\right)\\
\red{F}X\ar[ru]^{\overline{t}_{X}\quad}\ar[rd]_{\overline{t}_{FX}\quad} & r_{*}\red{F}\left(\Theta^{p}\left(X\right)\right)\ar@{-}[d]^{\wr}\ar[r]^{\nm_{r}^{-1}} & r_{!}\red{F}\left(\Theta^{p}\left(X\right)\right)\ar@{-}[d]^{\wr}\ar[u]_{\wr}^{\beta_{!}}\ar[r]^{g} & r_{!}\red{F}\left(\Theta^{p}\left(Y\right)\right)\ar@{-}[d]^{\wr}\ar[u]_{\wr}^{\beta_{!}} & \red{F}Y.\\
 & r_{*}\Theta^{p}\left(\red{F}X\right)\ar[r]^{\nm_{r}^{-1}} & r_{!}\Theta^{p}\left(\red{F}X\right)\ar[r]^{g} & r_{!}\Theta^{p}\left(\red{F}Y\right)\ar[ru]_{\quad\overline{m}_{\red{F}Y}}
}
\]

The vertical isomorphisms in the bottom squares are defined by \lemref{Theta_Functoriality} and the squares commute by the interchange law. The top left square commutes by the
ambidexterity of the $\left(F,r\right)$-square (\thmref{Ambi_Functors})
and the top right square by naturality of the $\bc_{!}$ map. The
triangles commute by the definition of the commutative coalgebra (resp.\!
algebra) structure on $F\left(X\right)$ (resp.\! $F\left(Y\right)$).
Thus, the composition along the top path, which is $F\left(\alpha\left(g\right)\right)$,
is homotopic to the composition along the bottom path, which is $\alpha\left(F\left(g\right)\right)$.
\end{proof}

\subsubsection{Additivity of Alpha}

Our next goal is to understand the interaction of $\alpha$ with sums.
For this, we first need to describe the effect of $\overline{\alpha}$
on ``induced maps''. Recall the notation
\[
\pt\oto eBC_{p}\oto r\pt.
\]

\begin{lem}
\label{lem:Alpha_Induced}Let $\mathcal{C}$ be a $1$-semiadditively
symmetric monoidal $\infty$-category and let $X\in\cocalg\left(\mathcal{C}\right)$
and $Y\in\calg\left(\mathcal{C}\right)$. For every map
\[
h\colon X^{\otimes p}=e^{*}\Theta^{p}\left(X\right)\to e^{*}\Theta^{p}\left(Y\right)=Y^{\otimes p},
\]
the map $\overline{\alpha}\left(\int\limits _{e}h\right)$ is homotopic
to the composition 
\[
X\oto{e^{*}t_{X}}X^{\otimes p}\oto hY^{\otimes p}\oto{e^{*}m_{Y}}Y.
\]
\end{lem}

\begin{proof}
Unwinding the definition of the integral, the map $\int\limits _{e}h$
is homotopic to the composition of the following maps 
\[
\Theta^{p}\left(X\right)\oto{u_{*}^{e}}e_{*}e^{*}\Theta^{p}\left(X\right)\oto he_{*}e^{*}\Theta^{p}\left(X\right)\oto{\nm_{e}^{-1}}e_{!}e^{*}\Theta^{p}\left(X\right)\oto{c_{!}^{e}}\Theta^{p}\left(X\right).
\]

Plugging this into the definition of $\overline{\alpha}$, we get
that $\overline{\alpha}\left(\int\limits _{e}h\right)$ equals the
composition along the top and then right path in the following diagram
\[
\xymatrix{X\ar[rrd]_{e^{*}t_{X}}\ar[r]^{\overline{t}_{X}\quad} & r_{*}\Theta^{p}\left(X\right)\ar[r]^{u_{*}^{e}\quad} & r_{*}e_{*}e^{*}\Theta^{p}\left(X\right)\ar@{=}[d]\ar[r]^{h} & r_{*}e_{*}e^{*}\Theta^{p}\left(Y\right)\ar@{=}[d]\ar[r]^{\nm_{e}^{-1}} & r_{*}e_{!}e^{*}\Theta^{p}\left(Y\right)\ar[d]^{\nm_{r}^{-1}}\ar[r]^{c_{!}^{e}} & r_{*}\Theta^{p}\left(Y\right)\ar[d]^{\nm_{r}^{-1}}\\
 &  & e^{*}\Theta^{p}\left(X\right)\ar[r]^{h} & e^{*}\Theta^{p}\left(Y\right)\ar[rrd]_{e^{*}m_{Y}}\ar@{=}[r] & r_{!}e_{!}e^{*}\Theta^{p}\left(Y\right)\ar[r]^{c_{!}^{e}} & r_{!}\Theta^{p}\left(Y\right)\ar[d]^{\overline{m}_{Y}}\\
 &  &  &  &  & Y.
}
\]

We denote this diagram by $\left(*\right)$. The left square commutes
for trivial reasons, the right square by the interchange law and the
middle by 
\[
\nm_{r}^{-1}\circ\nm_{e}^{-1}=\left(\nm_{e}\circ\nm_{r}\right)^{-1}=\left(\nm_{re}\right)^{-1}=\Id.
\]

To see that the left triangle commutes, consider the diagram
\[
\xymatrix@C=3pc{X\ar[rd]_{u_{*}^{r}}\ar[r]^{\overline{t}_{X}\quad} & r_{*}\Theta^{p}\left(X\right)\ar[r]^{u_{*}^{e}\quad} & r_{*}e_{*}e^{*}\Theta^{p}\left(X\right)\\
 & r_{*}r^{*}X\ar[u]_{t_{X}}\ar[r]^{u_{*}^{e}} & r_{*}e_{*}e^{*}r^{*}X\ar@{=}[r]\ar[u]_{t_{X}} & X\ar[lu]_-{e^{*}t_{X}}.
}
\]
The square commutes by naturality, and the left triangle by the definition
of mates. Note that the composition along the bottom path is the unit
of the composed adjunction 
\[
\Id=e^{*}r^{*}\dashv r_{*}e_{*}=\Id,
\]
and hence is the identity map. It follows that the left triangle in
$\left(*\right)$ is commutative. The proof that the right triangle
in $\left(*\right)$ commutes is completely analogous. Thus, $\left(*\right)$
is commutative and $\overline{\alpha}\left(\int\limits _{e}h\right)$
equals the composition along the bottom diagonal path in $\left(*\right)$,
which completes the proof.
\end{proof}
The main property of $\alpha$ is that it satisfies the following
``addition formula''.
\begin{prop}
\label{prop:Alpha_Additvity}Let $\mathcal{C}$ be a $1$-semiadditively
symmetric monoidal $\infty$-category and let 
\[
X\in\cocalg\left(\mathcal{C}\right),\quad Y\in\calg\left(\mathcal{C}\right).
\]
For every $f,g\colon X\to Y$, we have 
\[
\alpha\left(f+g\right)=\alpha\left(f\right)+\alpha\left(g\right)+\frac{\left(f+g\right)^{p}-f^{p}-g^{p}}{p}\quad\in\hom_{h\mathcal{C}}\left(X,Y\right)
\]
(as in \remref{Integer_Coeff_Polynom}, this expression does not actually
involve division by $p$).
\end{prop}

\begin{proof}
Since $\overline{\alpha}$ is additive (see \lemref{Alpha_Bar_Additive}),
we get by \propref{Theta_Additiive},
\[
\alpha\left(f+g\right)=\overline{\alpha}\left(\Theta^{p}\left(f+g\right)\right)=\overline{\alpha}\left(\Theta^{p}\left(f\right)+\Theta^{p}\left(g\right)+\sum_{\overline{w}\in\overline{S}}\left(\int\limits _{e}w\left(f,g\right)\right)\right)
\]
\[
=\alpha\left(f\right)+\alpha\left(g\right)+\sum_{\overline{w}\in\overline{S}}\overline{\alpha}\left(\int\limits _{e}w\left(f,g\right)\right).
\]
Now, by \lemref{Alpha_Induced}, the map $\overline{\alpha}\left(\int\limits _{e}w\left(f,g\right)\right)$
is homotopic to the composition 
\[
X\oto{e^{*}t_{X}}X^{\otimes p}\oto{w\left(f,g\right)}Y^{\otimes p}\oto{e^{*}m_{Y}}Y.
\]

This is by definition $f^{w_{x}}g^{w_{y}}$, where $w_{x}$ and $w_{y}$
are the number of $x$-s and $y$-s in $w$ respectively and this
completes the proof.
\end{proof}

\subsubsection{Alpha and The Unit}

We shall now apply the above discussion of the operation $\alpha$
to the special case where $X=Y=\one$ is the unit of a symmetric monoidal
$\infty$-category $\mathcal{C}$. The unit $\one\in\mathcal{C}$
has a unique $\mathbb{E}_\infty$-algebra structure and this structure makes
it initial in $\text{CAlg}\left(\mathcal{C}\right)$. The same argument
applied to $\mathcal{C}^{op}$ shows that $\one$ has also a unique
$\mathbb{E}_\infty$-coalgebra structure and it is terminal with respect to
it. 
\begin{defn}
Let $\left(\mathcal{C},\otimes,\one\right)$ be a symmetric monoidal
$\infty$-category. We denote 
\[
R_{\mathcal{C}}=\hom_{h\mathcal{C}}\left(\one,\one\right)
\]
as a commutative monoid. If $\mathcal{C}$ is $0$-semiadditive, then
$R$ is naturally a commutative \emph{rig} and if $\mathcal{C}$ is
stable, then it is a commutative \emph{ring}. Given a symmetric monoidal
functor $F\colon\mathcal{C}\to\mathcal{D},$ the induced map $\varphi:R_{\mathcal{C}}\to R_{\mathcal{D}}$
is a monoid homomorphism. It is also a rig (resp. ring) homomorphism,
when $\mathcal{C}$ and $\mathcal{D}$ are 0-semiadditive (resp. stable)
and $F$ is a $0$-semiadditive functor.
\end{defn}

The goal of this section is to study the operation $\alpha$ on $R_{\mathcal{C}}$.
We begin with a few preliminaries. Recall the notation 
\[
\pt\oto eBC_{p}\oto r\pt.
\]

\begin{lem}
Let $\left(\mathcal{C},\otimes,\one\right)$ be a symmetric monoidal
$\infty$-category. The action of $C_{p}$ on $\one^{\otimes p}\simeq\one$
is trivial. Namely, $\Theta^{p}\left(\one\right)=r^{*}\one$. 
\end{lem}

\begin{proof}
We shall show more generally that the action of $\Sigma_{k}$ on $\one^{\otimes k}\simeq\one$
is trivial. The forgetful functor $U\colon\calg\left(\mathcal{C}\right)\to\mathcal{C}$
is symmetric monoidal with respect to the coproduct on $\calg\left(\mathcal{C}\right)$
and $\otimes$ on $\mathcal{C}$ (by \cite[Example 3.2.4.4]{ha} and
\cite[Proposition 3.2.4.7]{ha}). In particular, for every commutative algebra
$A$, the action of $\Sigma_{k}$ on $U\left(A\right)^{\otimes k}$
is induced by the action of $\Sigma_{k}$ on $A^{\sqcup k}$. Since
$\one$ has a canonical commutative algebra structure, and as an object of 
$\calg\left(\mathcal{C}\right)$ it is \emph{initial}
(\cite[Corollary 3.2.1.9]{ha}), any $\Sigma_{k}$ action on it as
a commutative algebra is trivial.
\end{proof}
It follows by the above that $r_{!}\Theta^{p}\left(\one\right)\simeq r_{!}r^{*}\one$
and $r_{*}\Theta^{p}\left(\one\right)\simeq r_{*}r^{*}\one$. 
\begin{lem}
\label{lem:Co_Algebra_Unit}Let $\left(\mathcal{C},\otimes,\one\right)$
be a symmetric monoidal $\infty$-category. The maps 
\[
\overline{t}_{\one}\colon\one\to r_{*}\Theta^{p}\left(\one\right),\quad\overline{m}_{\one}\colon r_{!}\Theta^{p}\left(\one\right)\to\one,
\]
induced from the commutative algebra and coalgebra structures, are
equivalent to the unit and counit maps (respectively)
\[
u_{*}\colon\one\to r_{*}r^{*}\one,\qquad c_{!}\colon r_{!}r^{*}\one\to\one.
\]
\end{lem}

\begin{proof}
This is equivalent to showing that the mate (in both cases) is the
identity map $r^{*}\one\to r^{*}\one$. The algebra structure on $\one\in\mathcal{C}$
is induced from the algebra structure on $\underline{\one}\in\calg\left(\mathcal{C}\right)$,
where $\calg\left(\mathcal{C}\right)$ is endowed with the coCartesian
symmetric monoidal structure and in which $\underline{\one}$ is initial
(\cite[Corollary 3.2.1.9]{ha}). Now, the object $r^{*}\underline{\one}$
is initial in $\fun\left(BC_{p},\calg\left(\mathcal{C}\right)\right)$,
and therefore the only map $r^{*}\underline{\one}\to r^{*}\underline{\one}$
is the identity. A similar argument applies for the comultiplication
map.
\end{proof}
As a consequence, we can describe the effect of $\alpha$ on any element
of $R_{\mathcal{C}}$ using the integral operation.
\begin{prop}
\label{prop:Alpha_One}Let $\left(\mathcal{C},\otimes,\one\right)$
be a $1$-semiadditively symmetric monoidal $\infty$-category. For
every $f\in R_{\mathcal{C}}$, we have 
\[
\alpha\left(f\right)=\int\limits _{BC_{p}}\Theta^{p}\left(f\right)\quad\in\mathcal{R}_{\mathcal{C}}.
\]
\end{prop}

\begin{proof}
Unwinding the definition of $\overline{\alpha}$ (\defref{Alpha})
in this case and using \lemref{Co_Algebra_Unit}, we get $\overline{\alpha}\left(-\right)=\int\limits _{BC_{p}}\left(-\right)$.
Hence, 
\[
\alpha\left(f\right)=\overline{\alpha}\left(\Theta^{p}\left(f\right)\right)=
\int\limits _{BC_{p}}\Theta^{p}\left(f\right).
\]
\end{proof}
In particular, we get an explicit formula for the operation $\alpha$
on elements of the form $|A|\in\mathcal{R}_{\mathcal{C}}$.
\begin{thm}
\label{thm:Alpha_Box} Let $\mathcal{C}$ be an $m$-semiadditively
symmetric monoidal $\infty$-category for $m\ge1$. For every $m$-finite
space $A$, we have 
\[
\alpha\left(|A|\right)=|A\wr C_{p}|\quad\in\mathcal{R}_{\mathcal{C}}.
\]
\end{thm}

\begin{proof}
Consider the maps
\[
q\colon A\to\pt,\quad\pi=q\wr C_{p}\colon A\wr C_{p}\to BC_{p},\quad r\colon BC_{p}\to\pt.
\]

By definition of $\alpha$, \propref{Alpha_One}, the definition of
$|A|$, the ambidexterity of the $\Theta^{p}$-square (\thmref{Theta_Integral})
and Fubini's Theorem (\propref{Fubini}) (in that order) we have
\[
\alpha\left(|A|\right)=\overline{\alpha}\left(\Theta^{p}\left(|A|\right)\right)=\int\limits _{r}\Theta^{p}\left(|A|\right)=\int\limits _{r}\Theta^{p}\left(\int\limits _{q}\Id_{\one}\right)=\int\limits _{r}\int\limits _{\pi}\Theta^{p}\left(\Id_{\one}\right)=\int\limits _{r\pi}\Id_{\one}=|A\wr C_{p}|.
\]
\end{proof}
As a consequence, we can identify the action of $\alpha$ on the identity
element of the rig $\hom_{h\mathcal{C}}\left(X,Y\right)$, for any
$X\in\cocalg\left(\mathcal{C}\right)$ and $Y\in\calg\left(\mathcal{C}\right)$.
\begin{lem}
\label{lem:Alpha_Normalization}Let $\mathcal{C}$ be a $1$-semiadditively
symmetric monoidal $\infty$-category and let 
\[
X\in\cocalg\left(\mathcal{C}\right),\quad Y\in\calg\left(\mathcal{C}\right).
\]
Denoting $\mathcal{R}=\hom_{h\mathcal{C}}\left(X,Y\right)$, we have
\[
\alpha\left(1_{\mathcal{R}}\right)=|BC_{p}|\circ 1_{\mathcal{R}}\quad\in\mathcal{R},
\]
where $1_{\mathcal{R}}\in\mathcal{R}$ is the multiplicative unit element.
\end{lem}

\begin{proof}
The map $1_{\mathcal{R}}\colon X\to Y$ is the composition of the canonical
maps $X\oto x\one\oto yY$, encoding the counit and unit of the coalgebra
and algebra structures of $X$ and $Y$ respectively. The maps $x$
and $y$ are naturally maps of commutative coalgebras and commutative
algebras respectively. By \lemref{Alpha_Naturality}, we have 
\[
\alpha\left(1_{\mathcal{R}}\right)=\alpha\left(y\circ1\circ x\right)=y\circ\alpha\left(1\right)\circ x,
\]
where $1\in \mathcal{R}_{\mathcal{C}}$ is the multiplicative unit element. Observing that $1=|\pt|$ and using \thmref{Alpha_Box},
we get (we can commute $|BC_p|$ because it is a natural transformation) 
\[
y\circ\alpha\left(1\right)\circ x=y\circ\alpha\left(|\pt|\right)\circ x=y\circ|BC_{p}|\circ x=|BC_{p}|\circ y\circ x=|BC_{p}|\circ1_{\mathcal{R}}.
\]
\end{proof}

\subsection{Higher Semiadditivity and Stability}

In this section, we specialize to the \emph{stable} case. Using the
operation $\alpha$ and stability, we construct additive $p$-derivations
and use their properties to formulate a general detection principle
for higher semiadditivity. 

\subsubsection{Stability and Additive $p$-Derivations}
\begin{defn}
\label{def:Delta_Semi_Add}Let $\mathcal{C}$ be a \emph{stable} $1$-semiadditively
symmetric monoidal $\infty$-category with
\[
X\in\cocalg\left(\mathcal{C}\right),\quad Y\in\calg\left(\mathcal{C}\right),
\]
and so $R=\hom_{h\mathcal{C}}\left(X,Y\right)$ is a commutative \emph{ring}.
We define an operation $\delta\colon R\to R$ by 
\[
\delta\left(f\right)=|BC_{p}|f-\alpha\left(f\right),
\]
 for every $f\in R$. In particular, this applies to $\mathcal{R}_{\mathcal{C}}=\hom_{h\mathcal{C}}\left(\one,\one\right)$.
\end{defn}

\begin{thm}
\label{thm:Delta_Semi_Add}Let $\mathcal{C}$ be a stable $1$-semiadditively
symmetric monoidal $\infty$-category with 
\[
X\in\cocalg\left(\mathcal{C}\right),\quad Y\in\calg\left(\mathcal{C}\right).
\]
The operation $\delta$ from \defref{Delta_Semi_Add} is an additive
$p$-derivation on $R=\hom_{h\mathcal{C}}\left(X,Y\right)$.
\end{thm}

\begin{proof}
The additivity condition follows from \propref{Alpha_Additvity} and
the normalization follows from \lemref{Alpha_Normalization}. 
\end{proof}
The additive $p$-derivation of \thmref{Delta_Semi_Add} is natural
in the following sense.
\begin{prop}
\label{prop:Delta_Semi_Add_Naturality} Let $\mathcal{C}$ be a stable
$1$-semiadditively symmetric monoidal $\infty$-category with
\[
X,X'\in\cocalg\left(\mathcal{C}\right),\qquad Y,Y'\in\calg\left(\mathcal{C}\right).
\]
Given maps $g\colon Y\to Y'$ and $h\colon X'\to X$ of commutative
algebras and coalgebras respectively, the function
\[
g\circ(-)\circ h\colon\hom_{h\mathcal{C}}\left(X,Y\right)\to\hom_{h\mathcal{C}}\left(X',Y'\right)
\]
is a homomorphism of semi-$\delta$-rings.
\end{prop}

\begin{proof}
This follows from \lemref{Alpha_Naturality} and naturality of $|BC_{p}|$.
\end{proof}
The additive $p$-derivation of \thmref{Delta_Semi_Add} is also functorial
in the following sense.
\begin{prop}
\label{prop:Delta_Semi_Add_Functoriality}Let $F\colon\mathcal{C}\to\mathcal{D}$
be a symmetric monoidal $1$-semiadditive functor between \emph{stable}
$1$-semiadditively symmetric monoidal $\infty$-categories. Given
\[
X\in\cocalg\left(\mathcal{C}\right),\quad Y\in\calg\left(\mathcal{C}\right),
\]
the map 
\[
F\colon\hom_{h\mathcal{C}}\left(X,Y\right)\to\hom_{h\mathcal{D}}\left(FX,FY\right),
\]
is a homomorphism of semi-$\delta$-rings.
\end{prop}

\begin{proof}
By \lemref{Alpha_Functoriality}, $F$ preserves $\alpha$,
and by \corref{Integral_Functor}, $F$ preserves multiplication
by $|BC_{p}|$. Combined with ordinary additivity, it follows
that $F$ preserves $\delta$.
\end{proof}
The theory of $p$-local semi-$\delta$-rings has the following consequence
for stable, $p$-local, $1$-semiadditive, symmetric monoidal  $\infty$-categories.
\begin{cor}
\label{cor:Semiadd_Torsion_Nilpotent}Let $\mathcal{C}$ be a stable,
$p$-local, $1$-semiadditively symmetric monoidal $\infty$-category
with
\[
X\in\cocalg\left(\mathcal{C}\right),\qquad Y\in\calg\left(\mathcal{C}\right),
\]
and consider the commutative ring $R=\hom_{h\mathcal{C}}\left(X,Y\right)$.
Every torsion element of $R$ is nilpotent. In particular, if $\bb Q\otimes R=0$,
then $R=0$.
\end{cor}

\begin{proof}
The commutative ring $R$ is $p$-local and admits an additive $p$-derivation
by \thmref{Delta_Semi_Add}, and so the result follows by \propref{Delta_Torsion_Nilpotent}.
The last claim follows by considering the element $1\in R$.
\end{proof}

\subsubsection{Detection Principle for Higher Semiadditivity}

We now formulate the main detection principle for $m$-semiadditivity
for symmetric monoidal, stable, $p$-local $\infty$-categories. For
convenience, we formulate these results for presentable $\infty$-categories
and colimit preserving functors, though what we actually use is only
the existence and preservation of certain limits and colimits. 
\begin{lem}
\label{lem:Amenable_Semi_Add} Let $m\ge 1$ and let $\mathcal{C}$ be an $m$-semiadditive presentably symmetric monoidal, stable, $p$-local $\infty$-category.
If there exists a connected $m$-finite $p$-space $A$, such that
$\pi_{m}\left(A\right)\neq0$ and $|A|_{\one}$ is an isomorphism,
then $\mathcal{C}$ is $\left(m+1\right)$-semiadditive.
\end{lem}

\begin{proof}
Since $m\ge1$, the space $B^{m+1}C_{p}$ is connected. Since the $\infty$-category $\mathcal{C}$ is $m$-semiadditive,
the map $q\colon B^{m+1}C_{p}\to\pt$ is weakly $\mathcal{C}$-ambidextrous.
By \cite[Corollary 4.4.23]{HopkinsLurie}, it suffices to show that
$q$ is $\mathcal{C}$-ambidextrous. Since $m$-finite $p$-spaces
are nilpotent, and we assumed that $\pi_m(A)\neq0$, there is a fiber sequence $A\to B\oto{\pi}B^{m+1}C_{p}$
with $B$ an $m$-finite space. Since $|A|_{\one}$ is
invertible, by \lemref{Box_Unit}(2), $A$ is $\mathcal{C}$-amenable.
Hence, by \propref{Amenable_Space}, the space $B^{m+1}C_{p}$ is
$\mathcal{C}$-ambidextrous.
\end{proof}

We can exploit the extra structure given by the additive
$p$-derivation on $\mathcal{R}_{\mathcal{C}}$ to find a space $A$
as in \lemref{Amenable_Semi_Add}.
\begin{prop}
\label{prop:Bootstrap_Algebraic}Let $m\ge 1$ and let $\mathcal{C}$ be an $m$-semiadditive presentably symmetric monoidal, stable, $p$-local $\infty$-category. 
Let $h\colon\mathcal{R}_{\mathcal{C}}\to S$
be a semi-$\delta$-ring homomorphism that detects invertibility,
and such that $h\left(|BC_{p}|\right),h\left(|B^{m}C_{p}|\right)\in S$
are rational and non-zero. Then $\mathcal{C}$ is $\left(m+1\right)$-semiadditive.
\end{prop}

\begin{proof}
A space $A$ will be called \emph{$h$-good} if
\begin{itemize}
\item [(a)] $A$ is a connected $m$-finite $p$-space, such that $\pi_{m}\left(A\right)\neq0$.
\item [(b)] $h\left(|A|\right)$ is rational.
\end{itemize}
By \lemref{Amenable_Semi_Add}, it is enough to show that there exists
an $h$-good space $A$, such that $|A|$ is invertible
in $R_{\mathcal{C}}$. Since $h$ detects invertibility, it suffices
to find such $A$ with $h\left(|A|\right)$ invertible
in $S$. By assumption, $h\left(|B^{m}C_{p}|\right)$ is
rational and therefore $B^{m}C_{p}$ is $h$-good. If $p\in S^{\times}$,
then all non-zero rational elements in $S$ are invertible and we
are done by the assumption that $h\left(|B^{m}C_{p}|\right)\neq0$.
Hence, we assume that $p\notin S^{\times}$. In this case, a rational
element $x\in S$ is invertible if and only if $v_{p}\left(x\right)=0$.
Denoting $v\left(A\right)=v_{p}\left(h\left(|A|\right)\right)$,
it is enough to show that there exists an $h$-good space $A$ with
$v\left(A\right)=0$. 

Since $h\left(|B^{m}C_{p}|\right)$ is non-zero and $p$ is not invertible, we get
$0\le v\left(B^{m}C_{p}\right)<\infty$. It therefore
suffices to show that given an $h$-good space $A$ with $0<v\left(A\right)<\infty$,
there exists an $h$-good space $A'$ with 
$v\left(A'\right) = v\left(A\right)-1$.
For this, we exploit the operation $\delta$. We compute using \thmref{Alpha_Box} and \corref{Distributivity}:
\[
\delta\left(|A|\right)=|BC_{p}||A|-\alpha\left(|A|\right)=|BC_{p}||A|-|A\wr C_{p}|=|BC_{p}\times A|-|A\wr C_{p}|.
\]
Thus,
\[
\delta\left(h\left(|A|\right)\right)=h\left(\delta\left(|A|\right)\right)=h\left(|BC_{p}\times A|\right)-h\left(|A\wr C_{p}|\right).
\]
Since by assumption $h\left(|BC_{p}|\right)$ is rational,
then by \corref{Distributivity} we get that
\[
h\left(|BC_{p}\times A|\right)=h\left(|BC_{p}|\right)h\left(|A|\right)
\]
is also rational, and moreover, as $p \not\in S^{\times}$, we obtain $v(A)\le v\left(BC_p \times A\right)$. Furthermore, since $h\left(|A|\right)$ is rational, by \propref{Delta_Injective},
the same is true for $\delta\left(h\left(|A|\right)\right)$.
Therefore, 
\[
h\left(|A\wr C_{p}|\right)=h\left(|BC_{p}\times A|\right)-h\left(\delta\left(|A|\right)\right)
\]
is also rational. 
It is clear that $A\wr C_{p}$ satisfies (a), and so is $h$-good. Since $0<v\left(A\right)<\infty$, by \lemref{Delta_Valuation}, we get
 $v_{p}\left(\delta\left(h\left(|A|\right)\right)\right)=v\left(A\right)-1$.
Thus, $v\left(A\wr C_{p}\right) = v\left(A\right)-1$ and this completes the
proof. 
\end{proof}

\begin{rem}
The proof did not actually use anything specific to the space $B^{m}C_{p}$.
It would have sufficed to have some good space $A$ with $h\left(|A|\right)$
rational and non-zero. The space $B^{m}C_{p}$ is just the ``simplest''
one.
\end{rem}

In practice, the situation of \propref{Bootstrap_Algebraic} arises
as follows.
\begin{prop}
\label{prop:Bootstrap_Categorical} Let $m\ge1$, and let $F\colon\mathcal{C}\to\mathcal{D}$
be a colimit preserving symmetric monoidal functor between presentably
symmetric monoidal, stable, $p$-local, $m$-semiadditive $\infty$-categories.
Assume that the map $\varphi\colon \mathcal{R}_{\mathcal{C}}\to \mathcal{R}_{\mathcal{D}}$,
induced by $F$, detects invertibility and that the images of $|BC_{p}|_{\mathcal{D}},|B^{m}C_{p}|_{\mathcal{D}}\in\mathcal{R}_{\mathcal{D}}$
in the ring $\mathcal{R}_{\mathcal{D}}^{\tf}$ are rational and non-zero.
Then $\mathcal{C}$ and $\mathcal{D}$ are $\left(m+1\right)$-semiadditive.
\end{prop}

\begin{proof}
It is enough to prove that $\mathcal{C}$ is $\left(m+1\right)$-semiadditive,
since by \corref{Semi_Add_Mode}, this implies that $\mathcal{D}$
is $\left(m+1\right)$-semiadditive. We shall apply \propref{Bootstrap_Algebraic}
to the composition 
\[
\mathcal{R}_{\mathcal{C}}\oto{\varphi}\mathcal{R}_{\mathcal{D}}\oto g\mathcal{R}_{\mathcal{D}}^{\tf},
\]
where $g$ is the canonical projection. By \propref{Delta_Semi_Add_Functoriality},
$\varphi$ is a semi-$\delta$-ring homomorphism and it detects invertibility
by assumption. On the other hand, $g$ is a semi-$\delta$-ring homomorphism
and it detects invertibility by \propref{Delta_Torsion_Free}. It
is only left to observe that $\varphi\left(|A|_{\mathcal{C}}\right)=|A|_{\mathcal{D}}$,
which follows from \corref{Integral_Functor}.
\end{proof}
We conclude with a variant of \propref{Bootstrap_Categorical}, in
which the condition on the elements $|B^{m}C_{p}|_{\mathcal{D}}$,
is replaced by a condition on the closely related elements $\dim_{\mathcal{D}}\left(B^{m}C_{p}\right)$,
and which assembles together the individual statements for different
$m\in\bb N$.
\begin{thm}
\label{thm:Bootstrap_Machine}(Bootstrap Machine) Let $1\le m\le\infty$
and let $F\colon\mathcal{C}\to\mathcal{D}$ be a colimit preserving
symmetric monoidal functor between presentably symmetric monoidal,
stable, $p$-local $\infty$-categories. Assume that
\begin{enumerate}
\item $\mathcal{C}$ is $1$-semiadditive. 
\item The map $\varphi\colon R_{\mathcal{C}}\to R_{\mathcal{D}}$, induced
by $F$, detects invertibility. 
\item For every $0\le k<m$, if the space $B^{k}C_{p}$ is dualizable in
$\mathcal{D}$, then the image of $\dim_{\mathcal{D}}\left(B^{k}C_{p}\right)$
in $\mathcal{R}_{\mathcal{D}}^{\tf}$ is rational and non-zero. 
\end{enumerate}
Then $\mathcal{C}$ and $\mathcal{D}$ are $m$-semiadditive.
\end{thm}

\begin{proof}
It suffices to show that $\mathcal{C}$ is $m$-semiadditive, since by \corref{Semi_Add_Mode}, $\mathcal{D}$ is then also $m$-semiadditive.
We prove by induction on $k$, that the images of the elements $|B^{i}C_{p}|_{\mathcal{D}}$ in $\mathcal{R}_{\mathcal{D}}^{\tf}$ are rational and non-zero for all $0\le i < k$, and that $\mathcal{C}$ is $k$-semiadditive. 
The base case $k=1$
holds by assumption (1) and the fact that $|C_{p}|_{\mathcal{D}}=p$
is rational and nonzero in $\mathcal{R}_{\mathcal{D}}^{\tf}$, since
the unique ring homomorphism $\bb Z\to\mathcal{R}_{\mathcal{D}}^{\tf}$
is injective by \propref{Delta_Injective}. Assuming the inductive
hypothesis for some $k<m$, we first prove that $|B^{k}C_{p}|_{\mathcal{D}}$ in $\mathcal{R}_{\mathcal{D}}^{\tf}$ are rational and non-zero.
By \corref{Dim_Sym}, $B^{k}C_{p}$ is
dualizable in $\mathcal{D}$, and we have 
\[
\dim_{\mathcal{D}}\left(B^{k}C_{p}\right)=|B^{k}C_{p}|_{\mathcal{D}}|B^{k-1}C_{p}|_{\mathcal{D}}\in\mathcal{R}_{\mathcal{D}}^{\tf}.
\]
By assumption (3), $\dim_{\mathcal{D}}\left(B^{k}C_{p}\right)$ is
rational and non-zero and by the inductive hypothesis, the image of
$|B^{k-1}C_{p}|_{\mathcal{D}}$ in $\mathcal{R}_{\mathcal{D}}^{\tf}$
is rational and non-zero as well. Consequently, the image of $|B^{k}C_{p}|_{\mathcal{D}}$
in $\mathcal{R}_{\mathcal{D}}^{\tf}$ must also be rational and non-zero
since $\mathcal{R}_{\mathcal{D}}^{\tf}$ is torsion-free. 
We shall now deduce that $\mathcal{C}$ is
$\left(k+1\right)$-semiadditive by applying \propref{Bootstrap_Categorical}
to the functor $F$. Since $|B^{k}C_{p}|_{\mathcal{D}}$ is rational and non-zero, it suffices to show that  $|BC_{p}|_{\mathcal{D}}$ is rational and non-zero.
For $k=1$ there is nothing to prove and for $k\ge2$ this follows by the inductive hypothesis.
\end{proof}
\begin{rem}
The proof shows that the assumptions of the theorem above \emph{imply}
that the spaces $B^{k}C_{p}$ are dualizable in $\mathcal{D}$. Thus,
in retrospect, the ``if'' in assumption (3) is superfluous.
\end{rem}

\subsection{Nil-conservativitiy }
In this subsection we introduce and study a natural condition on a symmetric monoidal functor $\mathcal{C}\to\mathcal{D},$
which ensures that the induced map 
$\mathcal{R}_{\mathcal{C}}\to \mathcal{R}_{\mathcal{D}}$
detects invertibility.
For simplicity, we shall work throughout under the assumption of \emph{presentability},
though most of the arguments do not require the full strength of this
assumption.  

\begin{defn}
\label{def:Nil_Conservativity}We call a monoidal colimit preserving functor  $F\colon\mathcal{C}\to\mathcal{D},$
between stable presentably monoidal $\infty$-categories \emph{nil-conservative}, if for every ring
$R\in\alg(\mathcal{C})$, if $F(R)=0$ then $R=0$\footnote{This notion is closely related  to the notion of ``nil-faithfulness'' defined in \cite{BalmerNil}.}.
\end{defn}

The fundamental example of nil-conservativity in chromatic homotopy theory is provided by the Nilpotence Theorem (\propref{Nilpotence_Support}).
It is immediate from the definition that:
\begin{lem}
\label{lem:Nil_Cancelation} Let $F\colon{\cal C}\to{\cal D}$ and
$G\colon{\cal D}\to{\cal E}$ be monoidal colimit preserving functors between stable presentably monoidal $\infty$-categories.
\begin{enumerate}
\item If $F$ is conservative it is nil-conservative. 
\item If $F$ and $G$ are nil-conservative then $GF$ is nil-conservative. 
\item If $GF$ is nil-conservative then $F$ is nil-conservative. 
\end{enumerate}
\end{lem}

The property of nil-conservativity has a useful
equivalent characterization in terms of conservativity on dualizable
modules. For this we shall need a non-symmetric version of the known fact that
dualizable objects are closed under cofibers in the stable setting.
\begin{lem}
\label{lem:Dualizable_Cofiber}Let $\mathcal{C}$ be a stable presentably monoidal $\infty$-category 
and let $R,S\in\alg(\mathcal{C})$. For every cofiber sequence 
\[
X\to Y\to Z\quad\in\quad_{S}\BMod_{R}(\mathcal{C}),
\]
if two out of $X,Y$, and $Z$ are right dualizable, then so is the
third.
\end{lem}

\begin{proof}
We treat the case that $X$ and $Y$ are right dualizable (the other
cases are analogous). Given $\mathcal{M}\in\LMod_{\mathcal{C}}(\Pr^L)$
we have a functor 
\[
X\otimes_{R}(-)\colon\LMod_{R}(\mathcal{M})\to\LMod_{S}(\mathcal{M}).
\]
Moreover, given a morphism $\mathcal{M}^{\prime}\oto U\mathcal{M}^{\prime\prime}$
in $\LMod_{\mathcal{C}}(\Pr^L)$, we have a commutative diagram
\[
\xymatrix{\LMod_{R}(\mathcal{M}^{\prime})\ar[d]_{U}\ar[rr]^{X\otimes_{R}(-)} &  & \LMod_{S}(\mathcal{M}^{\prime})\ar[d]^{U}\\
\LMod_{R}(\mathcal{M}^{\prime\prime})\ar[rr]^{X\otimes_{R}(-)} &  & \LMod_{S}(\mathcal{M}^{\prime\prime}).
}
\]

By the dual of \cite[Proposition 4.6.2.10]{ha}, the module $X$ is right dualizable if
and only if for every $\mathcal{M}\in\LMod_{\mathcal{C}}(\Pr^L)$ the
functor $X\otimes_{R}(-)$ admits a left adjoint $F_{X}$, and for
every map $\mathcal{M}^{\prime}\oto U\mathcal{M}^{\prime\prime}$
in $\LMod_{\mathcal{C}}(\Pr^L)$ the Beck-Chevalley map  $F_{X}^{\prime\prime}U\oto{\beta_{X}}UF_{X}^{\prime}$ in the above diagram 
is an isomorphism\footnote{The statement of \cite[Proposition 4.6.2.10]{ha} considers more general $\mathcal{C}$-left
tensored $\infty$-categories $\mathcal{M}$. The proof however uses
only the special cases $\mathcal{M}=\RMod_{T}(\mathcal{C})$ for $T\in\alg(\mathcal{C})$.}. Since $\mathcal{C}$ is stable, so is every $\mathcal{M}\in\LMod_{\mathcal{C}}(\Pr^L)$
and $\LMod_{S}(\mathcal{M})$.
For every such $\mathcal{M}$, we have a cofiber sequence of functors
\[
X\otimes_{R}(-)\to Y\otimes_{R}(-)\to Z\otimes_{R}(-).
\]
Since the first two admit left adjoints $F_{X}$ and $F_{Y}$ respectively,
so does $Z\otimes_{R}(-)$. Moreover, we have a cofiber sequence of
functors
\[
F_{Z}\to F_{Y}\to F_{X}.
\]

Unwinding the definitions, for every $\mathcal{M}^{\prime}\to\mathcal{M}^{\prime\prime}$
in $\LMod_{\mathcal{C}}(\Pr^L)$, we have a commutative diagram of Beck-Chevalley maps:
\[
\xymatrix{F_{Z}^{\prime\prime}U\ar[d]^{\beta_{Z}}\ar[r] & F_{Y}^{\prime\prime}U\ar[d]^{\beta_{Y}}\ar[r] & F_{X}^{\prime\prime}U\ar[d]^{\beta_{X}}\\
UF_{Z}^{\prime}\ar[r] & UF_{Y}^{\prime}\ar[r] & UF_{X}^{\prime}.
}
\]
Hence, if $\beta_{X}$ and $\beta_{Y}$ are isomorphisms then so is
$\beta_{Z}$. 
\end{proof}
\begin{prop}
\label{prop:Nil_Conservativity_Dualizable}A monoidal colimit preserving functor $F\colon\mathcal{C}\to\mathcal{D}$
between stable presentably monoidal $\infty$-categories is nil-conservative, if and only if for every
$S\in\alg(\mathcal{C})$, the induced functor 
\[
\overline{F}\colon\LMod_{S}(\mathcal{C})\to\LMod_{F(S)}(\mathcal{D})
\]
is conservative when restricted to the full subcategories of right
dualizable modules.
\end{prop}

\begin{proof}
The `if' part follows from the fact that every ring $R$ is right dualizable
as a left module over itself. Conversely, let $f\colon N_{1}\to N_{2}$
be a map of right dualizable left $S$-modules and let $M$ be the
cofiber of $f$, which is also right dualizable (\lemref{Dualizable_Cofiber}).
It suffices to show that if $\overline{F}M=0$, then $M=0$. Let $M^{\vee}\in\RMod_{S}(\mathcal{C})$
denote the right dual of $M$. We have
\[
M^{\vee}\otimes_{S}M=\hom_{S}(M,M)\quad\in\quad\mathcal{C}
\]
the ring of endomorphisms of $M$. Since $F$ is monoidal and preserves
all, and in particular sifted, colimits we have 
\[
F(\hom_{S}(M,M))=F(M^{\vee}\otimes_{S}M)=\overline{F}M^{\vee}\otimes_{FS}\overline{F}M=0.
\]
By assumption, we get $\hom_{S}(M,M)=0$ and hence $M=0$.
\end{proof}
Applying \propref{Nil_Conservativity_Dualizable} to $S=\one_{\mathcal{C}},$
we see that a nil-conservative functor is in particular conservative
on right dualizable objects of $\mathcal{C}$ itself. 

\begin{cor}
\label{cor:Nil_Conservative_Detects_Inv}Let $F\colon\mathcal{C}\to\mathcal{D}$
be a nil-conservative functor. The induced ring
homomorphism $\mathcal{R}_{\mathcal{C}}\to \mathcal{R}_{\mathcal{D}}$
detects invertibility. In particular, if $A$ is a $\mathcal{C}$-ambidextrous
and $\mathcal{D}$-amenable space, then it is also $\mathcal{C}$-amenable. 
\end{cor}

\begin{proof}
This follows from \propref{Nil_Conservativity_Dualizable}, as $\one_{\mathcal{C}}$
is a dualizable object.
\end{proof}

\section{Applications to Chromatic Homotopy Theory
\label{sec:Applications_to_Chromatic}}

In this final section, after fixing some notation and terminology\footnote{We refer the reader to \cite{ravenel2016nilpotence}, for a comprehensive treatment of the fundamentals of chromatic homotopy theory.} 
we apply the general theory developed in the previous sections to chromatic homotopy theory. We begin by studying the consequences of $1$-semiadditivity to nilpotence in the homotopy groups of $\bb E_{\infty}$ (and $H_{\infty}$)-ring spectra and May's conjecture. Then, we prove the main theorem regarding the $\infty$-semiadditivity of $\Sp_{T\left(n\right)}$ and derive some corollaries. Finally, we study higher semiadditivity for localizations with respect to general weak rings (a generalization of a homotopy ring) and the various notions of ``bounded height'' for them.

Throughout, we fix a prime $p$ which will be implicit in all definitions that depend on it, except when explicitly stated otherwise.

\subsection{Generalities of Chromatic homotopy Theory}

We begin with some generalities, mainly to fix terminology and notation.
Let $\left(\Sp,\otimes,\bb S\right)$ be the symmetric monoidal $\infty$-category
of spectra (see \cite[Corollary 4.8.2.19]{ha})\footnote{This $\infty$-category can be also obtained using a symmetric monoidal model category as in \cite{ekmm} or \cite{SymSpectra}.}.

\subsubsection{Localizations, Rings and Modules}

Recall from \cite[Section 5.2.7]{htt} that a functor $L\colon \Sp\to \Sp$ is called a \emph{localization functor} if it factors as a composition $\Sp\to \Sp_L \to \Sp$, where the second functor is fully faithful and the first is its left adjoint. We abuse notation and denote by $L$ also the left adjoint $\Sp\to \Sp_L$ itself.  
We call a map $f$ in $\Sp$ an $L$-equivalence, if $L(f)$ is an isomorphism. 
As in \cite[Definition 2.2.1.6, Example 2.2.1.7]{ha}, a functor $L\colon \Sp \to \Sp$ is said to be compatible with the symmetric monoidal structure, if $L$-equivalences are closed under tensor product with all objects of $\mathcal{C}$.

\begin{defn}
	\label{defn: tensor localization}
	A localization functor $L\colon \Sp \to \Sp$ is called a $\otimes$-localization if $L$ is compatible with the symmetric monoidal structure. 
\end{defn}

Note that a localization functor $L\colon \Sp \to \Sp$ is a $\otimes$-localization, if and only if the $L$-acyclic objects are closed under desuspension. 



\begin{prop}
	\label{prop:Sp_Localization} For every $\otimes$-localization $L\colon \Sp\to \Sp$, the $\infty$-category
	$\Sp_L$ is stable, presentable and admits a structure
	of a presentably symmetric monoidal $\infty$-category $\left(\Sp_{L},\widehat{\otimes},L\bb S \right)$, such that the functor $L\colon \Sp\to \Sp_L$ is symmetric monoidal. Moreover, the inclusion $\Sp_L\into \Sp$ admits a canonical lax symmetric monoidal structure.  
	Finally, for all $X,Y\in \Sp_{L}$ we have
	\[
	X\widehat{\otimes}Y\simeq L\left(X\otimes Y\right).
	\]
\end{prop}

\begin{proof}
Applying \cite[Proposition 5.5.4.15]{htt} to the collection of $L$-equivalences, we deduce that $\Sp_L$ is presentable. Since $L$ is a $\otimes$-localization, all claims except for the stability of $\Sp_L$ follow from \cite[Proposition 2.2.1.9]{ha}. 
Now,  since $\Sp$ is pointed, so is $\Sp_L$ (e.g. from \corref{Semi_Add_Mode}). To show the stability of $\Sp_L$ by \cite[Corollary 1.4.2.27]{ha} it is enough to show that $\Sigma \colon \Sp_L \to \Sp_L$ is an equivalence. Indeed, this functor has an inverse, given by tensoring with $L\left(\Sigma^{-1} \bb{S}\right)$. 
\end{proof}

For every spectrum $E\in\Sp$, we denote by $L_E\colon \Sp \to \Sp$ the $\otimes$-localization with essential image the $E$-local spectra\footnote{This functor is also called Bousfield localization after Bousfield who originally constructed it in \cite{BousLoc}.}. We denote  $\Sp_{L_E}$ by $\Sp_E$  and $L_E(\bb{S})$ by $\bb{S}_E$ .  
For a prime $p$, we shall consider also $\otimes$-localizations $L\colon \Sp_{(p)}\to \Sp_{(p)}$. The analogous results and notation apply to the $p$-local case as well.  

\begin{prop}
Let $E\in\Sp$ and let $R$ be an $E$-local $\bb E_{\infty}$-ring. The
$\infty$-category $\Mod_{R}^{\left(E\right)}$ of left modules over
$R$ in the symmetric monoidal $\infty$-category $\Sp_{E}$, is presentable
and admits a structure of a presentably symmetric monoidal $\infty$-category.
Moreover, we have a free-forgetful adjunction
\[
F_{R}\colon\Sp_{E}\adj\Mod_{R}^{\left(E\right)}\colon U_{R},
\]
in which $F_{R}$ is symmetric monoidal.
\end{prop}

\begin{proof}
\cite[Corollary 4.5.1.5]{ha} identifies modules over $R$ as an $\bb E_{\infty}$-ring
with left modules over $R$ as an $\bb E_{1}$-ring. By \cite[Theorem 4.5.3.1]{ha}
and \cite[Corollary 4.2.3.7]{ha} this $\infty$-category is equipped
with a presentably symmetric monoidal structure. By \cite[Remark 4.2.3.8]{ha}
and \cite[Remark 4.5.3.2]{ha} applied to the map of algebras $\bb S_{E}\to R$,
we have the adjunction $F_{R}\dashv U_{R}$, such that $F_{R}$ is
symmetric monoidal.
\end{proof}

We shall also consider the following much weaker notion of a ``ring'' spectrum:
\begin{defn}
	\label{def:Weak_Ring}
	A \emph{weak ring}\footnote{It is called \emph{$\mu$-spectrum} in \cite[Definition 4.8]{hoveystrickland}}
	is a spectrum $R\in\Sp$, together with a ``unit'' map $u{\colon}\bb S\to R$
	and a ``multiplication'' map $\mu{\colon}R\otimes R\to R$, such that the
	composition 
\[
\xymatrix{
	R\ar[r]^-{u\otimes\Id} & R\otimes R\ar[r]^-{\mu} & R,
}
\]	
	is homotopic to the identity. 
\end{defn}

\begin{example}
	\label{example: telescopic weak rings}
	Every homotopy-ring is a weak ring.
\end{example}

\begin{lem}
	\label{lem:Tensor_Weak_Rings} 
	Let $R$ and $S$ be weak rings. Then $R\otimes S$ 
	is a weak ring. 
\end{lem}

\begin{proof}
This follows directly from the definition.
\end{proof}

Our interest in weak rings stems from the fact that they include
a large class of spectra of interest and have just enough structure
to invoke the Nilpotence Theorem (see \thmref{Nilpotence}).

\subsubsection{Morava Theories}

Given an integer $n\ge0$, let $E_{n}$ be a 2-periodic Morava $E$-theory
of height $n$ with coefficients (for $n\ge1$)

\[
\pi_{*}E_{n}\simeq\bb Z_{p}[[u_{1},\dots,u_{n-1}]][u^{\pm1}],\quad\left|u_{i}\right|=0,\ \left|u\right|=2,
\]

and let $\K\left(n\right)$ be a $2$-periodic Morava $K$-theory
of height $n$ with coefficients ($n\ge1$)
\[
\pi_{*}\K\left(n\right)=\bb F_{p}[u^{\pm1}],\quad\left|u\right|=2.
\]
The spectrum $E_{n}$ admits an $\bb E_{\infty}$-ring structure in
$\Sp$ (by \cite{GoerssH}). The spectrum $\K\left(n\right)$ is obtained from the even $\bb{E}_{\infty}$-ring $E_{n}$ by taking the quotient with respect to the (regular) sequence $(p,u_1,\dots,u_{n-1})$ and hence admits an $\bb{E}_1$-ring structure with an $\bb{E}_{1}$-ring map $E_{n}\to\K\left(n\right)$ (see e.g. \cite{QuotEvenRings}). 
Since $E_{n}$ is $\K\left(n\right)$-local,
we can also view it as an $\bb E_{\infty}$-ring in the $\infty$-category $\Sp_{K\left(n\right)}$. We shall use the notation $\widehat{\Mod}_{E_n}$ for $\Mod_{E_n}^{(K(n))}$.

We shall make an essential use of the dualizability and dimension of Eilenberg-MacLane spaces in $\widehat{\Mod}_{E_n}$. First, there is a general criterion for a space to be dualizabe in $\widehat{\Mod}_{E_n}$: 
\begin{lem}
\label{lem:Morava_Dimension}Let $n\ge0$ and let $X$ be a space.
If 
\[
\dim_{\bb F_{p}}\left(\K\left(n\right)_{0}\left(X\right)\right)=d<\infty\qquad and\qquad\K\left(n\right)_{1}\left(X\right)=0,
\]
then $X$ is dualizable in $\widehat{\Mod}_{E_n}$
and 
\[
    \dim_{\widehat{\Mod}_{E_n}}\left(X\right)=d.
\]
\end{lem}

\begin{proof}
By \cite[Proposition 3.4.3]{HopkinsLurie} (see also \cite[Proposition 8.4]{hoveystrickland}), there is an isomorphism
of $E_{n}$-modules
\[
L_{K\left(n\right)}\left(E_{n}\otimes\Sigma^{\infty}X_{+}\right)\simeq E_{n}^{d},
\]
from which the claim follows immediately.
\end{proof}
\begin{rem}
Using \cite[Proposition 3.4.3]{HopkinsLurie} together with \cite[Proposition 10.11]{mathew_galois}, one can deduce that for every dualizable object $M\in\widehat{\Mod}_{E_n}$,
we have 
\[
\dim_{\widehat{\Mod}_{E_n}}\left(M\right)=\dim_{\bb F_{p}}\left(\pi_{0}\left(\K\left(n\right)\otimes_{E_{n}}M\right)\right)-\dim_{\bb F_{p}}\left(\pi_{1}\left(\K\left(n\right)\otimes_{E_{n}}M\right)\right).
\]
But we shall not need this fact.
\end{rem}

Using the classical computations of Ravenel and Wilson we get the following:
\begin{cor}
\label{cor:Morava_Dimension_EM} For all $k\in\bb N$, we have 
\[
\dim_{\widehat{\Mod}_{E_n}}\left(B^{k}C_{p}\right)=p^{\binom{n}{k}}\quad\in\pi_{0}\left(E_{n}\right).
\]
In particular, these are all rational and non-zero.
\end{cor}

\begin{proof}
By \cite[Theorem 9.2]{RavenelWilson}, we have 
\[
\dim_{\bb F_{p}}\K\left(n\right)_{0}\left(B^{k}C_{p}\right)=p^{\binom{n}{k}}\qquad\text{and}\qquad\K\left(n\right)_{1}\left(B^{k}C_{p}\right)=0.
\]
Hence, the result follows from \lemref{Morava_Dimension}.
\end{proof}

\subsubsection{Telescopic Localizations}

\begin{defn}
A finite $p$-local spectrum $X$, i.e. a compact object in the $\infty$-category
$\Sp_{\left(p\right)}$, is said to be of \emph{type $n$}, if $\K\left(n\right)\otimes X\neq0$
and $\K\left(j\right)\otimes X=0$ for $j=0,\dots,n-1$.
\end{defn}
Let $\X\left(n\right)$ be a finite $p$-local spectrum of type $n$.
Let $\bb D\X\left(n\right)=\underline{\hom}\left(\X\left(n\right),\bb S_{(p)}\right)$
be the Spanier-Whitehead dual of $\X\left(n\right)$. The finite $p$-local
spectrum
\[
R=\bb D\X\left(n\right)\otimes\X\left(n\right)=\underline{\hom}\left(\X\left(n\right),\X\left(n\right)\right),
\]
is also of type $n$ by the K\"unneth isomorphism. By replacing $\X\left(n\right)$
with $R$, we may assume that $\X\left(n\right)$ is an $\bb E_{1}$-ring
(see \cite[Section 4.7.1]{ha}). Every type $n$ spectrum $\X\left(n\right)$
admits a $v_{n}$-self map, which is a map 
\[
v\colon\Sigma^{k}\X\left(n\right)\to\X\left(n\right),
\]
that is an isomorphism on $\K\left(n\right)_{*}X$ and zero on $\K\left(j\right)_{*}\left(X\right)$
for $j\neq n$. We let
\[
\T\left(n\right)=v^{-1}\X\left(n\right)=\colim\limits _{k}\left(\X\left(n\right)\oto v\Sigma^{-k}\X\left(n\right)\oto v\Sigma^{-2k}\X\left(n\right)\oto v\dots\right),
\]
be the telescope on $v$. The canonical map $\X\left(n\right)\to\T\left(n\right)$
exhibits $\T\left(n\right)$ as the $\T\left(n\right)$-localization
of $\X\left(n\right)$ (e.g. \cite[Proposition 3.2]{mahowaldsadofskyloc}).
Since the functor $L_{T\left(n\right)}$ is symmetric monoidal, we
can consider $\T\left(n\right)=L_{T\left(n\right)}\X\left(n\right)$
as an $\bb E_{1}$-ring in $\Sp_{T\left(n\right)}$. By the Thick Subcategory and Periodicity theorems \cite{nilp2}, the localization $\Sp_{T(n)}$ depends only on the prime $p$ and the height $n$ and in particular is independent of the choice of $F(n)$ and $v$.
It is known (e.g.
\cite[Section 6 (4)]{mahowaldsadofskyloc}) that 
\[
\Sp_{K\left(n\right)}\ss\Sp_{T\left(n\right)}\ss\Sp.
\]
Thus, both $E_{n}$ and $\K\left(n\right)$ are also $\T\left(n\right)$-local,
and so we can consider them as an $\bb E_{\infty}$-ring and an $\bb E_{1}$-ring
in $\Sp_{T\left(n\right)}$ respectively. 

\subsubsection{Nilpotence Theorem}

Morava $K$-theories are used in the following definition of support: 
\begin{defn}
Let $L{\colon}\Sp_{(p)} \to \Sp_{(p)}$ be a $\otimes$-localization functor. The (chromatic) \emph{support} of $L$ is the set 
\[\supp(L)=\left\{0\le n\le \infty \mid L\left(\K\left(n\right)\right)\ne 0 \right\}\subseteq \mathbb{N}\cup \{\infty\}.\]	
For $E\in \Sp_{(p)}$ we denote $\supp(E)=\supp(L_E)$. 
\end{defn}

Note that $L_E(X)=0$ if and only if $E\otimes X=0$ and so $\supp(E)$ coincides with the usual notion of chromatic support of a spectrum. By the K\"unneth Theorem we have 
\[\supp(E\otimes E') = \supp(E)\cap \supp(E').\]

\begin{lem} 
	\label{lem:Kn_LSp}
	Let $L{\colon}\Sp_{\left(p\right)}\to\Sp_{\left(p\right)}$
	be a $\otimes$-localization functor and let $0\le n \le \infty$. Then  $n\in \supp\left(L\right)$ if and only if $\Sp_{\K\left(n\right)}\ss\Sp_{L}$. 
\end{lem}

\begin{proof}
	If $\Sp_{\K\left(n\right)} \ss \Sp_{L}$ then in particular $\K\left(n\right) \in \Sp_{L}$ and hence \[L(\K\left(n\right))=\K\left(n\right)\ne 0.\] 
	Conversely, we need to show that if $L\left(X\right)=0$, then $\K\left(n\right)\otimes X$
	is zero. We have
	\[
	L\left(\K\left(n\right)\otimes X\right)\simeq L\left(\K\left(n\right)\right)\,\widehat{\otimes}\,L\left(X\right)=0.
	\]
	Since $\K\left(n\right)\otimes X$ is a direct sum of
	suspended copies of $\K\left(n\right)$, if $\K\left(n\right)\otimes X\neq0$,
	then up to a suspension, $\K\left(n\right)$ is a retract of $\K\left(n\right)\otimes X$.
	This would imply that $L\left(\K\left(n\right)\right)=0$ in contradiction to the hypothesis.  
\end{proof}

\begin{example}\label{exa:Chromatic_Support} The following are examples for the support of some particular types of localization:
	\begin{enumerate}
		\item For a finite spectrum $\X\left(n\right)$ of type $n$, we have by definition
		\[\supp(\X\left(n\right)) = \{n,n+1,\dots,\infty\}.\] 
		\item For every integer $n\ge0$, we have 
		\[\supp(\K\left(n\right))= \supp(\T\left(n\right)) = \{n\}\] (see e.g. \cite[Proposition A.2.13]{RavBook}).
		\item For a non-zero smashing localization $L$, we have 
		\[\supp(L)=\{0,\dots,n\}\] 
		for some $0\le n \le \infty$ (see e.g. \cite[Lemma 4.1]{barthel2015short}).
		\item For the Brown-Comenetz spectrum $I_{\mathbb{Q}/\mathbb{Z}}$, we have $\supp({I_{\mathbb{Q}/\mathbb{Z}}})=\es$ (see e.g. \cite[Proposition 7.4.2]{ravbook2}
	\end{enumerate}  
\end{example}

When considering localizations with respect to non-zero \emph{weak rings}, the Nilpotence Theorem of Devintaz-Hopkins-Smith guarantees that the support can not be empty. Since it is not  usually stated in this generality, we include the argument for deriving it from the standard version. 
\begin{thm}[Devinatz-Hopkins-Smith]
	\label{thm:Nilpotence}
	Let $R$ be a $p$-local weak ring. Then $R=0$ if and only if $\supp(R)=\es$. 
\end{thm}
\begin{proof}
	Consider the unit map $u{\colon}\bb S\to R$. If $\K\left(n\right)\otimes R=0$
	for all $0\le n\le\infty$, then by \cite[Theorem 3(iii)]{nilp2},
	the map $u$ is \emph{smash nilpotent}. Namely, $u^{\otimes r}{\colon}\bb S\to R^{\otimes r}$
	is null for some $r\ge1$. The commutative diagram
	\[
	\xymatrix@R=2pc@C=3pc{\bb S\otimes\bb S\ar[d]_{\Id\otimes u}\ar[r]^{u\otimes u} & R\otimes R\ar[d]^{\mu}\\
		\bb S\otimes R\ar[ru]^{u\otimes\Id}\ar[r]^{\Id} & R
	}
	\]
	shows that $u$ factors through $u\otimes u$. Applying this iteratively,
	we can factor $u$ through the null map $u^{\otimes r}$ and deduce
	that $u$ itself is null. Consequently, the factorization of the identity map of
	$R$ as the composition 
\[
\xymatrix{
	R\ar[r]^-{u\otimes\Id} & R\otimes R\ar[r]^-{\mu} & R,
}
\]	
	implies that it is null and thus $R=0$. 
\end{proof}

This provides the main example of a nil-conservative functor.
\begin{prop}
\label{prop:Nilpotence_Support}Let $R$ be a $p$-local weak ring. The functor 
\[
L\colon\Sp_{R}\to\prod_{n\in\supp(R)}\Sp_{K(n)},
\]
whose $n$-th component is $K(n)$-localization, is nil-conservative.
\end{prop}

\begin{proof}
Let $S$ be an $R$-local ring spectrum. If $L(S)=0$, then $S\otimes K(n)=0$
for all $n\in\supp(R)$. On the other hand, by definition, $R\otimes K(n)=0$ for all
$n\notin\supp(R)$. Consequently, $S\otimes R\otimes K(n)=0$ for
all $n\in\bb N\cup\{\infty\}$. By \lemref{Tensor_Weak_Rings}, $S\otimes R$ is a weak ring and hence by the Nilpotence Theorem (\thmref{Nilpotence}) we get $S\otimes R=0$.
Finally, since $S$ is $R$-local, $S=0$.
\end{proof}

\begin{defn}
\label{def:En_Hat}
Let $\widehat{E}_{n}[-]$ be the composition 
\[
\Sp_{T\left(n\right)}\oto{L_{K\left(n\right)}}\Sp_{K\left(n\right)}\oto{F_{E_{n}}}\widehat{\Mod}_{E_n},
\]
where we abuse notation and write $L_{K\left(n\right)}$ also for
the left adjoint of the inclusion $\Sp_{K\left(n\right)}\ss\Sp_{T\left(n\right)}$.
The functor $\widehat{E}_{n}[-]$ is a colimit preserving
symmetric monoidal functor as a composition of two such.
\end{defn}

\begin{cor}
\label{cor:Nil_Conservativity_Telescopic}For every $0\le n<\infty$,
the functor 
\[
\widehat{E}_{n}[-]\colon \Sp_{T(n)}\to \widehat{\Mod}_{E_n},
\] 
is nil-conservative. Consequently, the canonical map 
$\pi_0\bb{S}_{T(n)}\to \pi_0 E_{n}$
detects invertibility.
\end{cor}

\begin{proof}
By \propref{Nilpotence_Support} and the fact that
$\mathrm{supp}(T(n))=\{n\}$ (\exaref{Chromatic_Support}(2)), we get that $L_{K(n)}\colon\Sp_{T(n)}\to\Sp_{K(n)}$
is nil-conservative. Since 
\[
    E_n \,\widehat{\otimes}\, (-)\colon \Sp_{K(n)}\to \widehat{\Mod}_{E_n}
\]
is conservative, it is in particular nil-conservative and hence the composition $\Sp_{T(n)}\to \widehat{\Mod}_{E_n}$
is nil-conservative. The claim now follows from \corref{Nil_Conservative_Detects_Inv}.
\end{proof}

\begin{rem}
    The fact that the functor 
    \(
        \widehat{E}_{n}[-]\colon \Sp_{T(n)}\to \widehat{\Mod}_{E_n}
    \)
    is conservative on dualizable objects is what gives us the handle on $\Sp_{T(n)}$, which will allow us to prove the $\infty$-semiadditivity of $\Sp_{T(n)}$. However, it has other uses as well. As a consequence of the $\infty$-semiadditivity of $\Sp_{T(n)}$, we have a large supply of dualizable objects in $\Sp_{T(n)}$, 
    including for example all $\pi$-finite spaces. 
    In an upcoming work, we shall exploit this fact together with nil-conservativity to lift the maximal abelian Galois extension of $\Sp_{K(n)}$ to $\Sp_{T(n)}$.
\end{rem}

\subsection{Consequences of 1-Semiadditivity}

In this section, we discuss some applications of the theory of $1$-semiadditivity in stable $\infty$-categories to chromatic homotopy theory.

\subsubsection{Power Operations}

\begin{defn}
    We denote by $\rm{CRing}$ the category of commutative rings and by $\rm{CRing^{\delta}}$ the category of semi-$\delta$-rings and semi-$\delta$-ring homomorphisms.   
\end{defn}

\begin{thm}\label{thm:Frob_Lift}
The functor 
\[
\pi_{0}\colon\calg(\Sp_{T\left(n\right)})\to\rm{CRing}
\]
has a lift to a functor 
\[
\calg(\Sp_{T\left(n\right)})\to \rm{CRing^{\delta}}
\]
along the forgetful functor $\rm{CRing^{\delta}} \to \rm{CRing}$.
\end{thm}

\begin{proof}
The $\infty$-category $\Sp_{T\left(n\right)}$ is $1$-semiadditive
by \cite{Kuhn} and therefore satisfies the conditions of \thmref{Delta_Semi_Add}.
Thus, for every $R\in\calg(\Sp_{T\left(n\right)})$, the commutative
ring 
\[
\pi_{0}R=\hom_{h\Sp_{T\left(n\right)}}(\bb S_{T\left(n\right)},R)
\]
admits an additive $p$-derivation $\delta$. The functoriality follows from \propref{Delta_Semi_Add_Naturality}.
\end{proof}

Note that for a semi-$\delta$-ring $(R,\delta)$, the operation 
\[
    \psi(x) \coloneqq x^p + p\delta(x)\quad \colon \quad R \to R
\]
is an additive group homomorphism which satisfies $\psi(1)=1$ and $\psi(x) \equiv x^p\,\in\,R/pR$. Thus, \thmref{Frob_Lift} also provides a functorial additive lift of Frobenius for $T(n)$-lcoal commutative ring spectra.  

\begin{rem}
    For every $R\in \calg(\Sp_{K(1)})$, Hopkins defined in \cite{HopkinsK1} an additive $p$-derivation 
    \[
        \theta \colon \pi_0(R) \to \pi_0(R).
    \]
    We warn the reader that our operation $\delta$ is \emph{not} the same as this $\theta$. We can describe precisely the relationship between the two by expressing both in terms of the operation $\alpha(x) = \int_{BC_p}x^p$ as follows:
    \begin{align*}
        \delta(x) &= |BC_p|x - \alpha(x) \\
        \theta(x) &= \frac{1}{p-1}(x^p - \alpha(x)).
    \end{align*}
    The additivity property of $\theta$ can be easily deduced from the corresponding property of $\alpha$ (\propref{Alpha_Additvity}).
However, the identity $\theta(1)=0$ amounts to $|BC_p|_{\Sp_{K(1)}}= 1$, a fact which does not generalize to higher heights. On the other hand, the operation $\theta$ also satisfies the multiplicativity rule
    \[ 
        \theta(xy) = \theta(x)\theta(y) + \theta(x)y^p + x^p\theta(y),
    \]
    making the associated lift of Frobenius $x^p+p\theta(x)$ a \emph{ring} homomorphism. 
\end{rem}

\subsubsection{May's Conjecture}\label{sec:May}
\begin{defn}
\label{def:sofic}Let $\mathcal{C}$ be a stable, presentably symmetric
monoidal $\infty$-category. We say that $\mathcal{C}$ is \emph{sofic},\footnote{The term \emph{sofic} is derived from the Hebrew word "sofi" for $finite$ (See \cite{weiss1973subshifts}).}
if there exists a stable, $1$-semiadditive, presentably symmetric
monoidal $\infty$-category $\mathcal{D}$ and a colimit preserving,
conservative, lax symmetric monoidal functor $\mathcal{C}\to\mathcal{D}$.
We call a spectrum $E\in\Sp$ sofic, if $\Sp_{E}$ is sofic.
\end{defn}

Every stable, $1$-semiadditive, presentably symmetric monoidal $\infty$-category
is of course sofic, but the latter condition is considerably weaker.
\begin{example}
The spectrum $H\bb Q$ is sofic and more generally, the spectra $\K\left(n\right)$
and $\T\left(n\right)$ for all $n$. Any sum of sofic spectra is
sofic, and since being sofic depends only on the Bousfield class (that is, the collection of acyclic objects) of
the spectrum, so are the Morava theories $E_{n}$ for all $n$ and
the telescopic localizations of the sphere spectrum $L_{n}^{f}\bb S$.
\end{example}

\begin{thm}
\label{thm:May}Let $E\in\Sp$ be a sofic homotopy commutative ring
spectrum and let $R$ be an $\bb E_{\infty}$-ring. For every $x\in\pi_{*}R$,
if the image of $x$ in $\pi_{*}\left(H\bb Q\otimes R\right)$ is
nilpotent, then the image of $x$ in $\pi_{*}\left(E\otimes R\right)$
is nilpotent. 
\end{thm}

Namely, the single homology theory $H\bb Q$ detects nilpotence in
all sofic homology theories.
\begin{proof}
First, observe that 
\[
\pi_{*}\left(H\bb Q\otimes R\right)\simeq\bb Q\otimes\pi_{*}R.
\]

Replacing $x$ with a suitable power, we can assume that $x$ is torsion
in $\pi_{*}R$. Since the homogeneous components of a torsion element
are torsion, we may assume without loss of generality that $x\in\pi_{k}R$
for some $k$ (i.e. $x$ is homogeneous). Consider the corresponding
map $|x|\colon R\to\Sigma^{-k}R$ given by multiplication
by $x$. The telescope
\[
x^{-1}R=\colim\left(R\oto{|x|}\Sigma^{-k}R\oto{|x|}\Sigma^{-2k}R\oto{|x|}\dots\right)
\]
carries a structure of an $\bb E_{\infty}$-ring and the map $R\to x^{-1}R$
induces the localization by $x$ map on $\pi_{*}$. In particular,
the unit map $\bb S\to x^{-1}R$ is torsion. Let $F\colon\Sp_{E}\to\mathcal{C}$
be a functor as in \defref{sofic}. Since $F$ and $L_{E}$ are both
lax symmetric monoidal and exact, $F\left(L_{E}\left(x^{-1}R\right)\right)$
is an $\bb E_{\infty}$-algebra and its unit is also torsion. By \corref{Semiadd_Torsion_Nilpotent},
$F\left(L_{E}\left(x^{-1}R\right)\right)=0$ and since $F$ is conservative,
$L_{E}\left(x^{-1}R\right)=0$. It follows that $E\otimes x^{-1}R=0$.
Let $\tilde{x}$ be the image of $x$ in $\pi_{*}\left(E\otimes R\right)$.
Since,
\[
\pi_{*}\left(E\otimes x^{-1}R\right)\simeq\pi_{*}\left(\tilde{x}^{-1}\left(E\otimes R\right)\right)\simeq\tilde{x}^{-1}\pi_{*}\left(E\otimes R\right),
\]
we get that $\tilde{x}$ is nilpotent in $\pi_{*}\left(E\otimes R\right)$. 
\end{proof}
\begin{rem}
We could have replaced $\bb E_{\infty}$ by $H_{\infty}$ (see \remref{H_Infty}).
Applying the theorem in this form for $E=\K\left(n\right)$, and using
the Nilpotence Theorem, one can deduce the conjecture of May, that
was proved in \cite{MathewMay}. We also note that the above
theorem can be extended to a general stable presentably symmetric
monoidal $\infty$-category $\mathcal{C}$ with a compact unit (instead
of $\Sp$) and $x\colon I\to R$ any map from an invertible object
$I$ (i.e. an object of the Picard group of $\mathcal{C}$).
\end{rem}

\subsection{Higher Semiadditivity of $T(n)$-Local Spectra}

In this section, we prove the main theorem of the paper. Namely, we show that
the $\infty$-category $\Sp_{T\left(n\right)}$ is $\infty$-semiadditive for all $n\ge0$ and draw some consequences from this. 
Our strategy is to apply the ``Bootstrap Machine'' (\thmref{Bootstrap_Machine})
to the functor $\widehat{E}_n[-]$ given in \defref{En_Hat}.

\begin{thm}
\label{thm:Tn_Semiaddi}
For all $n\ge0$, the $\infty$-categories
$\Sp_{T\left(n\right)}$ and $\widehat{\Mod}_{E_n}$
are $\infty$-semiadditive.
\end{thm}

\begin{proof}
We verify the assumptions (1)-(3) of \thmref{Bootstrap_Machine} for the colimit preserving symmetric monoidal functor
\[
\widehat{E}_{n}[-]\colon\Sp_{T\left(n\right)}\to\widehat{\Mod}_{E_n}.
\] 
Namely, we need to show that
\begin{enumerate}
\item The $\infty$-categories $\Sp_{T\left(n\right)}$ are $1$-semiadditive. 
\item The functor $\widehat{E}_{n}[-]$ detects invertibility.
\item The symmetric monoidal dimensions of the spaces $B^{k}C_{p}$ in $\widehat{\Mod}_{E_n}$ are rational and non-zero.
\end{enumerate}
Claim (1) is proved in \cite{Kuhn}, claim (2) follows from \corref{Nil_Conservativity_Telescopic}, and claim (3) is given by \corref{Morava_Dimension_EM}.
\end{proof}

This readily implies the original result of \cite{HopkinsLurie}.
\begin{cor}
\label{cor:Kn_Semiadd}For all $0\le n<\infty$, the $\infty$-category $\Sp_{K\left(n\right)}$
is $\infty$-semiadditive.
\end{cor}

\begin{proof}
Apply \corref{Semi_Add_Mode} to the localization functor $L_{T\left(n\right)}\colon\Sp_{T\left(n\right)}\to\Sp_{K\left(n\right)}$.
Alternatively, one could just use the same argument as in \thmref{Tn_Semiaddi}.
\end{proof}
By \thmref{Tn_Semiaddi}, both $\infty$-categories $\Sp_{T\left(n\right)}$
and $\widehat{\Mod}_{E_n}$ are $\infty$-semiadditive.
Hence, for every $\pi$-finite space $A$, we have an element $|A|\in\pi_{0}\bb S_{T\left(n\right)}$,
which maps to the corresponding element $|A|\in\pi_{0}E_{n}$
(since the map is induced by a colimit preserving functor). We shall
make some computations regarding these elements and use them to deduce
some new facts about $\Sp_{T\left(n\right)}$.
\begin{lem}
\label{lem:EM_Box_Morava}For every $k,n\ge0$ we have 
\[
|B^{k}C_{p}|_{\widehat{\Mod}_{E_n}}=p^{\binom{n-1}{k}}\quad\in\pi_{0}E_{n}.
\]
\end{lem}

\begin{proof}
By \corref{Dim_Sym} and \corref{Morava_Dimension_EM}, we have 
\[
p^{\binom{n}{k}}=\dim_{\widehat{\Mod}_{E_n}}\left(B^{k}C_{p}\right)=|B^{k}C_{p}||B^{k-1}C_{p}|.
\]
The result now follows by induction on $k$, using the identity 
\[
\binom{n-1}{k}+\binom{n-1}{k-1}=\binom{n}{k}
\]
and the fact that the ring $\pi_{0}E_{n}$ is torsion free.
\end{proof}
\begin{lem}
\label{lem:EM_Box_Tn_Inv}For every $k\ge n\ge0$ the element $|B^{k}C_{p}|_{\Sp_{\T(n)}} \in\pi_{0}\bb S_{T\left(n\right)}$
is invertible.
\end{lem}

\begin{proof}
For $n=0$ this is clear, so we may assume $n\ge 1$. By \corref{Nil_Conservativity_Telescopic}, the map 

\[
f\colon\pi_{0}\bb S_{T\left(n\right)}\to\pi_{0}E_{n}
\]
detects invertibility and by \lemref{EM_Box_Morava}, 
\[
f\left(|B^{k}C_{p}|\right)=p^{\binom{n-1}{k}}=1.
\]
\end{proof}
\begin{thm}
\label{thm:Tn_Contractiblity}
Let $n\ge0$ and let $f\colon A\to B$
be a map with $\pi$-finite $n$-connected homotopy fibers. The induced
map $\Sigma_{+}^{\infty}f\colon\Sigma_{+}^{\infty}A\to\Sigma_{+}^{\infty}B$
is a $\T\left(n\right)$-equivalence. 
\end{thm}

\begin{proof}
We begin with a standard general argument that reduces the statement
to the case $B=\pt$, by passing to the fibers. Consider the equivalence
of $\infty$-categories
\[
\mathcal{S}_{/B}\iso\fun\left(B,\mathcal{S}\right),
\]
given by the Grothendieck construction. Let $X\in\fun\left(B,\mathcal{S}\right)$
be the local system of spaces on $B$, that corresponds to $f$ and
let $Y\in\fun\left(B,\mathcal{S}\right)$ be the constant local system
with value $\pt\in\mathcal{S}$. As $Y$ is terminal, there is an essentially
unique map $X\to Y$, which at each point $b\in B$, is the essentially
unique map from $X_{b}$, the homotopy fiber of $f$ at $b$, to $Y_{b}=\pt$.
We recover $f$, up to homotopy, as the induced map on colimits
\[
A\simeq\colim X\to\colim Y\simeq B.
\]
For each $E\in\Sp$, the functor 
\[
E\otimes\Sigma_{+}^{\infty}\left(-\right)\colon\mathcal{S}\to\Sp
\]
preserves colimits. Therefore, if the induced map for each homotopy
fiber
\[
E\otimes\Sigma_{+}^{\infty}X_{b}\to E\otimes\Sigma_{+}^{\infty}\pt,
\]
 is an isomorphism, then the induced map on colimits is also an isomorphism
\[
E\otimes\Sigma_{+}^{\infty}A\iso E\otimes\Sigma_{+}^{\infty}B.
\]

Now, if $B=\pt$, we have that $A$ is a $\pi$-finite $n$-connected
space. For $n=0$, the claim is obvious, and so we may assume that
$n\ge1$. Therefore, $A$ is simply connected and in particular nilpotent.
Thus, we can refine the Postnikov tower of $A$ to a finite tower
\[
A=A_{0}\to A_{1}\to\dots\to A_{d}=\pt,
\]
such that the homotopy fiber of each $A_{i}\to A_{i+1}$ is of the
form $B^{k}C_{q}$, for $q$ a prime and $k\ge n+1$. It thus suffices
to show that the map 
\[
\Sigma_{+}^{\infty}B^{k}C_{q}\to\Sigma_{+}^{\infty}\pt\simeq\bb S,
\]
induced by 
\[
g\colon B^{k}C_{q}\to\pt,
\]

is a $\T\left(n\right)$-equivalence. For $q\neq p$ this is clear.
For $q=p$ we apply \propref{Amenable_Contractible} to the map $g$.
For this, we need to check that 
\[
|\Omega B^{k}C_{p}|=|B^{k-1}C_{p}|
\]
is invertible in $\pi_{0}\bb S_{T\left(n\right)}$, which follows
from \lemref{EM_Box_Tn_Inv}. Alternatively,
$B^kC_q$ is dualizable in $\Sp_{T(n)}$ by \thmref{Tn_Semiaddi} and \corref{Dim_Sym}, and $L_{K(n)}\colon \Sp_{T(n)}\to \Sp_{K(n)}$ in nil-conservative by \propref{Nilpotence_Support} and \exaref{Chromatic_Support}(2). 
Thus, by \propref{Nil_Conservativity_Dualizable}, it suffices to check that the map $g$ is a $K(n)$-equivalence, which follows from the computation of $K(n)_*B^kC_q$ carried out in \cite{RavenelWilson}. 
\end{proof}

\begin{rem}
The analogous result for $\K\left(n\right)$ instead of $\T\left(n\right)$
is a consequence of the \cite{RavenelWilson} computation of the $\K\left(n\right)$-homology
of Eilenberg-MacLane spaces. A weaker result for $\T\left(n\right)$,
namely that the conclusion holds if the homotopy fibers of $f$ are
$\pi$-finite and $k$-connected for $k\gg0$, can be deduced from
\cite[Theorem 3.1]{Bousfield82}.
\end{rem}

\begin{cor}
	\label{cor:General_Contractibility}Let $n\ge0$ and let $f\colon A\to B$
	be a map with $\pi$-finite $n$-connected homotopy fibers. For every
	localization $L\colon\Sp_{\left(p\right)}\to\Sp_{\left(p\right)}$
	such that $L\left(\X\left(n+1\right)\right)=0$, the induced map 
	\[
	L\left(\Sigma_{+}^{\infty}f\right)\colon L\left(\Sigma_{+}^{\infty}A\right)\to L\left(\Sigma_{+}^{\infty}B\right)
	\]
	is an isomorphism. 
\end{cor}

\begin{proof}
	The condition $L\left(\X\left(n+1\right)\right)=0$ ensures that $L\colon\Sp_{\left(p\right)}\to\Sp_{L}$
	factors through the finite chromatic localization $L_{n}^{f}\colon\Sp_{\left(p\right)}\to\Sp_{L_{n}^{f}}$,
	which is also the localization with respect to the spectrum $T\left(0\right)\oplus\cdots\oplus T\left(n\right)$.
	Hence, the claim follows from \thmref{Tn_Contractiblity}.
\end{proof}

\subsection{Higher Semiadditivity and Weak Rings}
\label{sec:semiadd_and_chrom}
In this section, we study higher semiadditivity for more general localizations of spectra. In particular, with respect to \emph{weak rings}, which are a very weak version of a homotopy ring (see \defref{Weak_Ring}). We begin by studying the chromatic support of a localization and show that three different notions of ``bounded chromatic height'' for weak rings coincide. We then study localizations of the $\infty$-category $\Sp$ with respect to weak rings, which are $1$-semiadditive. 
We show that $p$-locally, those are precisely the intermediate localizations between $\Sp_{K(n)}$ and $\Sp_{T(n)}$. We deduce that such localizations are always $\infty$-semiadditive and also derive a characterization of higher semiadditivity in terms of the Bousfield-Kuhn functor. 

\subsubsection{General Localizations}

We begin with a discussion regarding general $\otimes$-localizations. The following result relates the $1$-semiadditivity of a $\otimes$-localization and the support of the corresponding localization functor. In a sense, a 1-semiadditive $\otimes$-localization of $\Sp_{\left(p\right)}$ is monochromatic of finite height.
\begin{prop}
	\label{prop:Semiadd_Monochrom}Let $L\colon\Sp_{\left(p\right)}\to\Sp_{\left(p\right)}$
	be a $\otimes$-localization functor. If $\Sp_{L}$ is 1-semiadditive,
	then either $\supp(L)= \es$ or $\supp(L)=\{n\}$ for some $0 \le n < \infty$. 
\end{prop}

\begin{proof}
	We start by showing that $\infty \notin \supp(L)$. 
	Assuming the contrary, by \lemref{Kn_LSp}, we get that 
	\[
	H\mathbb{F}_{p}=\K\left(\infty\right)\in\Sp_{\K\left(\infty\right)}\ss\Sp_{L}.
	\]
	This is a contradiction to \corref{Semiadd_Torsion_Nilpotent} as $H\mathbb{F}_{p}$
	is an $\mathbb{E}_{\infty}$-ring.  
	
	It remains to show that $\supp(L)$ cannot contain two different natural numbers. We shall prove this by contradiction. Suppose that there are $0\le m<n<\infty$ such that $m,n\in \supp(L)$. 
	By \lemref{Kn_LSp} again, it follows that $\Sp_{\K\left(m\right)},\Sp_{\K\left(n\right)}\ss\Sp_{L}$.
	In particular, we get $E_{m},E_{n}\in\Sp_{L}$. Consider the object
	\[
	E_{n}\widehat{\otimes}E_{m}=L\left(E_{n}\otimes E_{m}\right)\in\Sp_{L}.
	\]
	
	We begin by showing that  $E_{n}\widehat{\otimes}E_{m}\neq0$. Indeed, since
	$\Sp_{\K\left(m\right)}\ss\Sp_{L}$ we have
	\[
	L_{\K\left(m\right)}\left(E_{n}\widehat{\otimes}E_{m}\right)\simeq L_{\K\left(m\right)}L\left(E_{n}\otimes E_{m}\right)\simeq L_{\K\left(m\right)}\left(E_{n}\otimes E_{m}\right).
	\]
	The spectrum $L_{\K\left(m\right)}\left(E_{n}\otimes E_{m}\right)$
	is non-zero by the K\"unneth isomorphism and the fact that 
	\[
	\K\left(m\right)\otimes E_{m},\,\K\left(m\right)\otimes E_{n}\neq0.
	\]
	
	The object $E_{n}\widehat{\otimes}E_{m}$ is an
	$\bb E_{\infty}$-ring in the $1$-semiadditive $\infty$-category
	$\Sp_{L}$ and therefore we have a well defined element $a=|BC_{p}|\in\pi_{0}\left(E_{n}\widehat{\otimes}E_{m}\right)$.
	By naturality, $a$ is the image of the elements 
	\[
	a_{n}=|BC_{p}|\in\pi_{0}\left(E_{n}\right),\quad a_{m}=|BC_{p}|\in\pi_{0}\left(E_{m}\right)
	\]
	under the canonical maps of $\bb E_{\infty}$-rings $E_{n}\to E_{n}\widehat{\otimes}E_{m}$
	and $E_{m}\to E_{n}\widehat{\otimes}E_{m}$ respectively. However, the computation in \lemref{EM_Box_Morava}
	shows that $a_{n}=p^{n-1}$ and $a_{m}=p^{m-1}$. Hence, their equality is a contradiction to the injectivity of the unit map $\bb{Z}\to \pi_0(E_n \widehat{\otimes} E_m)$, which follows from  \propref{Delta_Injective}. 
\end{proof}

\begin{rem}
	The definition of $a_{n}$ as $|BC_{p}|\in\pi_{0}\left(E_{n}\right)$
	is unambiguous, since by \corref{Integral_Functor} for the symmetric monoidal
	localization functor $\Sp_{L}\to\Sp_{\K\left(n\right)},$ it does
	not matter whether we consider $E_{n}$ as an object of $\Sp_{L}$
	or $\Sp_{\K\left(n\right)}$.
\end{rem}

\begin{rem}
We are not aware of any example of a non-zero $1$-semiadditive $\otimes$-localization $L$, for which $\supp(L) = \es$. 
The techniques of this paper can be used to show that if no such examples exist (as we suspect), then \corref{Semiadd_Collapse} can be generalized to  every  $\otimes$-localization. 
Namely, a $\otimes$-localization of $\Sp$ is $1$-semiadditive if and only if it is $\infty$-semiadditive.
\end{rem}

\subsubsection{Weak Rings}

There are several notions of being ``of height $\le n$'' for the Bousfield
class of a weak ring. The following theorem shows that they are all equivalent.
\begin{thm}
	\label{thm:Height_Below_n} Let $R$ be a non-zero $p$-local weak
	ring and let $0\le n <\infty$. The following are equivalent:
	\begin{enumerate}
		\item $\X\left(n+1\right)\otimes R=0$ for some finite spectrum $\X\left(n+1\right)$ of type $n+1$.
		\item $\Sigma^{\infty}B^{n+1}C_{p}\otimes R=0$. 
		\item $\supp(R) \subseteq \{0,\dots,n\}$.
	\end{enumerate}
\end{thm}

\begin{proof}
	We prove the equivalence by showing first that the implications labeled by solid arrows in the diagram below hold for a general $\otimes$-localization $L$, where the condition
	$R\otimes X=0$ is interpreted as $L\left(X\right)=0$. Then we turn to the remaining implication, labeled by the dashed arrow in the diagram. 
	\[
	\xymatrix{
		                         &                             &                  \\
		\left(1\right)\ar@{=>}[r] & \left(2\right)\ar@{=>}[r] & \left(3\right) \ar@/_1.5pc/@{==>}[ll].}
	\]
	The implication
	$\xymatrix{\left(1\right)\ar@{=>}[r] & \left(2\right)}$ follows from \corref{General_Contractibility}.
	Given (2), assume by contradiction that $L\left(\K\left(m\right)\right)\neq0$
	for some $n< m\le\infty$. On the one hand, by \lemref{Kn_LSp},
	we have 
	\[
	\Sigma^{\infty}B^{n+1}C_{p}\otimes\K\left(m\right)\in\Sp_{K\left(m\right)}\ss\Sp_{L}.
	\]
	On the other hand, by assumption, $L\left(\Sigma^{\infty}B^{n+1}C_{p}\right)=0$ and
	hence (using {\cite{RavenelWilson}}), we have
	\[
	0\neq \Sigma^{\infty}B^{n+1}C_{p}\otimes\K\left(m\right)=L(\Sigma^{\infty}B^{n+1}C_{p}\otimes\K\left(m\right))=L(\Sigma^{\infty}B^{n+1}C_{p})\widehat{\otimes}L\left(\K\left(m\right)\right)=0
	.\]

	It now suffices to show that for a localization with respect to a
	non-zero $p$-local weak ring $R$, the implication 
	$\xymatrix{\left(3\right) \ar@{==>}[r] & \left(1\right)}$ holds as well.
	By Example \ref{example: telescopic weak rings}, we may assume that $\X\left(n+1\right)$ is
	a weak ring and hence by Lemma \lemref{Tensor_Weak_Rings}, the spectrum 
	$\X\left(n+1\right)\otimes R$ is a weak ring as well. Moreover, we have  
	\[\supp(\X\left(n+1\right)\otimes R) = \supp(\X\left(n+1\right)) \cap \supp(R)=\es\] and thus by \thmref{Nilpotence}, we get
	$\X\left(n+1\right)\otimes R=0$.
\end{proof}

\begin{rem}
	Condition $(2)$ in \thmref{Height_Below_n} has an alternative formulation. As in the proof of \thmref{Tn_Contractiblity}, if $E$ is a $p$-local spectrum, $\Sigma^\infty B^{n+1}C_p \otimes E = 0$ if and only if the following condition is satisfied:
	\begin{enumerate}
	\item[$\left(2'\right)$] For every map $f{\colon}A\to B$ of $\pi$-finite spaces,
	that induces an isomorphism on the $n$-th Postnikov truncation, the map 
	\[ \Sigma^\infty_+f \otimes E  \colon  \Sigma^\infty_+A\otimes E  \to \Sigma^\infty_+B\otimes E\] 
	is an isomorphism. 
	\end{enumerate}
\end{rem}
We now show that for localization with respect to a weak ring, being monochromatic is even more closely related to higher semiadditivity. For this we need the following general lemma.
\begin{lem}
	\label{lem:Bousfield_Class_Expansion}Let $E\in\Sp_{\left(p\right)}$.
	For every $n\ge0$, the spectrum $E$ is Bousfield equivalent to 
	\[
	\left(\T\left(0\right)\otimes E\right)\oplus\left(\T\left(1\right)\otimes E\right)\oplus\cdots\oplus\left(\T\left(n\right)\otimes E\right)\oplus\left(\X\left(n+1\right)\otimes E\right).
	\]
\end{lem}

Note that the Bousfield class of $X\otimes Y$ depends only on the
Bousfield classes of $X$ and $Y$. Hence, in the statement and proof
of the above lemma, we are free to choose the $\T\left(i\right)$-s
and $\X\left(n+1\right)$ as we please.
\begin{proof}
	Using the Periodicity Theorem (\cite{nilp2}) we can construct a sequence of finite type $n$ spectra $\X\left(n\right)$
	with $v_{n}$-self maps 
	\[
	v_{n}{\colon}\Sigma^{d_{n}}\X\left(n\right)\to\X\left(n\right),
	\]
	such that
	\begin{enumerate}
		\item $\X\left(0\right)=\bb S_{\left(p\right)}$
		\item $\X\left(n+1\right)$ is the cofiber of $v_{n}$.
		\item $\T\left(n\right)=v_{n}^{-1}\X\left(n\right)$.
	\end{enumerate}
	The claim now follows from a repeated application of \cite[Lemma 1.34]{ravconj}.
\end{proof}
\begin{thm}
	\label{thm:Monochrom} Let $R$ be a non-zero $p$-local weak ring.
	The following are equivalent:
	\begin{enumerate}
		\item  There exists a (necessarily unique) integer $n\ge0$, such that
		\(\Sp_{\K\left(n\right)}\ss\Sp_{R}\ss\Sp_{\T\left(n\right)}.\)
		\item  Either $\Sp_{R}=\Sp_{H\bb Q}$, or $\Omega^{\infty}{\colon}\Sp_{R}\to\mathcal{S_{*}}$
			admits a retract.
		\item $\Sp_{R}$ is $\infty$-semiadditive.
		\item $\Sp_{R}$ is $1$-semiadditive.
		\item $\supp(R) = \{n\}$ for some $0 \le n<\infty$. 
	\end{enumerate}
	Moreover, the integer $n$ in $\left(1\right)$ and $\left(5\right)$
	is the same one.\end{thm}

\begin{proof}
	Consider the following slight variant of  condition (5):
	\begin{enumerate}
	\item[(5)'] Either $\supp(R)=\es$ or $\supp(R)=\{n\}$ for some $0 \le n < \infty$.  
	\end{enumerate}
	
	We shall prove the theorem by verifying all the implications in the following diagram:
\[
\vcenter{
	\xymatrix@R=0.2pc@C=4pc{
								&  & & & &	\\ \\
	\ar@{==>}[d]		   				& (2)\ar@{=>}@/^0.3pc/[rd] & & & \\	
	(1)\ar@{=>}@/^0.35pc/[ru]\ar@{=>}@/_0.35pc/[rd] &  & (4)\ar@{=>}[r] & (5')\ar@{==>}[r] & (5)\ar@<-1pt>@{--}`u[uuullll]`[llllu][llllu] 
				  \ar@<+1pt>@{--}`u[uuullll]`[llllu][llllu] \\		   				    & (3)\ar@{=>}@/_0.3pc/[ru]
	}	
}
\]
	In fact, we show that the implications labeled by solid arrows hold for a general $\otimes$-localization $L$, and those labeled by dashed arrows hold for $L_R$, where $R$ is a non-zero $p$-local weak ring.
	
	We start by showing that $\xymatrix{\left(1\right)\ar@{=>}[r] & \left(2\right).}$ If $n=0$ then 
	\[\Sp_{\K\left(0\right)}=\Sp_{\T\left(0\right)}=\Sp_{\mathbb{Q}}\] and we
	are done. Otherwise, let $\Phi_{n}{\colon}\mathcal{S}_{*}\to \Sp_{T(n)}$
	be the Bousfield-Kuhn functor (see \cite{kuhn1989morava,bousfield2001telescopic}). We get that $L\circ\Phi_{n}$ is a
	retract of $\Omega^{\infty}{\colon}\Sp_{L}\to\mathcal{S}_{*}$. To show that $\xymatrix{\left(1\right)\ar@{=>}[r] & \left(3\right),}$
	consider the symmetric monoidal colimit preserving functor $L{\colon}\Sp_{T(n)}\to \Sp_{L}$.
	The claim now follows from Theorem \ref{thm:Tn_Semiaddi} and Corollary \ref{cor:Semi_Add_Mode}. The implication
	$\xymatrix{\left(2\right)\ar@{=>}[r] & \left(4\right)}$ is proved in \cite[Theorem 2.6]{ClausenAkhil}. Finally, $\xymatrix{\left(3\right)\ar@{=>}[r] & \left(4\right)}$ is
	trivial and $\xymatrix{\left(4\right)\ar@{=>}[r] & \left(5\right)}$ is \propref{Semiadd_Monochrom}. 
	
	 It is left to show the implications $\xymatrix{\left(5\right)' \ar@{==>}[r] & \left(5\right)}$ and $\xymatrix{\left(5\right) \ar@{==>}[r] & \left(1\right),}$ where $L=L_R$ for a non-zero $p$-local weak ring $R$. The first implication follows from \thmref{Nilpotence}. For the second, let $0\le n < \infty$ be such that $\supp(R) = \{n\}$. By \lemref{Kn_LSp}, we have $\Sp_{\K\left(n\right)}\ss\Sp_{R}$. 
	It remains to show that $\Sp_{R} \subseteq \Sp_{\T\left(n\right)}$. By \lemref{Bousfield_Class_Expansion}, the spectrum
	$R$ is Bousfield equivalent to
	\[
	\left(\T\left(0\right)\otimes R\right)\oplus\left(\T\left(1\right)\otimes R\right)\oplus\cdots\oplus\left(\T\left(n\right)\otimes R\right)\oplus\left(\X\left(n+1\right)\otimes R\right).
	\]
	From the assumption $\supp \left(R\right)=\{n\}$ and  \thmref{Height_Below_n}, we get that $\X\left(n+1\right)\otimes R=0$.
	By Example \ref{example: telescopic weak rings} and Lemma \lemref{Tensor_Weak_Rings}
	we may assume that the spectra $\T\left(m\right)\otimes R$ for $m<n$
	are weak rings. Now, for $m\ne n$ we have 
	\[\supp(\T\left(m\right)\otimes R) = \supp(\T\left(m\right))\cap \supp(R)=\{m\}\cap\{n\}=\es\]
	and therefore $\T\left(m\right)\otimes R=0$ by \thmref{Nilpotence}. It follows that
	\[
	\Sp_{R}\simeq\Sp_{\T\left(n\right)\otimes R}\ss\Sp_{\T\left(n\right)}.
	\]
\end{proof}

We conclude by showing that the equivalence of conditions $\left(3 \right)$ and $\left(4 \right)$ in \thmref{Monochrom} holds for general, not necessarily $p$-local, weak rings.



\begin{lem}
	\label{lem:torsion_than_product}Let $E$ be a spectrum and let $\ell$
	be an integer such that $E\oto{\times \ell} E$ is null. The canonical
	functor $F{\colon}\Sp_{E}\to\prod_{p\mid\ell}\Sp_{E_{(p)}}$ is an equivalence,
	where $E_{(p)}$ is the $p$-localization of $E$ at the prime $p$.
\end{lem}

\begin{proof}
	The functor $F$ admits a right adjoint $G$ given on objects by 
\[
	(X_{p})_{\{p\mid \ell\}}\mapsto\bigoplus_{p\mid \ell} X_{p}.
\]
	It suffices to show that $F$ is conservative and that the counit of the adjunction is an isomorphism. Since multiplication by $\ell$ is null on $E$, all homotopy groups of $E$ are $\ell$-torsion. It follows that the canonical map $E \to \bigoplus_{p\mid \ell} E_{(p)}$ is an isomorphism on homotopy groups and hence an isomorphism. Thus, $F$ is conservative. The components of the counit are given by
\[
	L_{E_{(p)}}\left(\bigoplus_{q\mid \ell}X_{q}\right) \to X_{p}.
\] 
	Since $L_{E_{(p)}}$ is exact, it is enough to show that $L_{E_{(p)}}(X_q)=0$ for all $q \neq p$. Indeed, multiplication by $p$ acts invertibly on $X_q$ and nilpotently on $E_{(p)}$, hence $E_{(p)}\otimes X_q=0$.    
\end{proof}

\begin{cor} \label{cor:Semiadd_Collapse}
	Let $R\in \Sp$ be a (not necessarily $p$-local) weak ring. Then $\Sp_{R}$
	is 1-semiadditive if and only if it is $\infty$-semiadditive. 
\end{cor}

\begin{proof}
	Denote by $R_{(p)}$ the $p$-local weak ring $R\otimes\mathbb{S}_{(p)}$.
	By \corref{Semi_Add_Mode} applied to the localization functor $F_{p}\colon \Sp_{R}\to \Sp_{R_{(p)}}$, the $\infty$-category $\Sp_{R_{(p)}}$ is 1-semiadditive
	and hence by \thmref{Monochrom} it is $\infty$-semiadditive. We divide into cases according to whether $R\otimes H\mathbb{Q}$ vanishes or not. 
If $R\otimes H\mathbb{Q}=0$, then the unit $u_{R}{\colon}\mathbb{S}\to R$ has finite order $\ell$ in $\pi_{0}R$. Hence, 
\[
	\ell \cdot \Id_R =\ell \cdot \mu_R(u_R \otimes \Id_R)= 
	\mu_R((\ell \cdot u_R) \otimes \Id_R)=0.
\] 
By \lemref{torsion_than_product} we have 
\[\Sp_{(R)}\cong \prod_{p\mid \ell} \Sp_{R_{(p)}}\]
and by \corref{product_semiadditive} it is $\infty$-semiadditive. Now, consider the case where $H\mathbb{Q}\otimes R\ne0$. 
For every prime $p$, since $\Sp_{R_{(p)}}$ is 1-semiadditive and \[K(0)\otimes R_{(p)}=H\mathbb{Q}\otimes R\ne0,\] we get from \thmref{Monochrom} that $\supp(R_{(p)})=\{0\}$. 
By \thmref{Height_Below_n} applied to the Moore spectrum $M(p)=\X\left(1\right)$, we obtain 
\[
	R\otimes M(p) \simeq R_{(p)}\otimes M(p)=0.
\] 
It follows that $R\in \Sp_{H\mathbb{Q}}=\Mod_{H\mathbb{Q}}$ and hence $\Sp_{R}=\Sp_{H\mathbb{Q}}$ is $\infty$-semiadditive. 
\end{proof}

\bibliographystyle{alpha}
\phantomsection\addcontentsline{toc}{section}{\refname}\bibliography{AmbidexterityFinal}

\end{document}